\documentclass[11pt]{amsart}
\usepackage{amssymb}
\usepackage{graphicx}                            
\usepackage{hyperref}
\usepackage{setspace}                            

\newtheorem{theorem}{Theorem}[section]

\newenvironment{adfenumerate}{
\begin{enumerate}
\setlength{\itemsep}{0.5mm}
\setlength{\parskip}{0mm}
\setlength{\parsep}{0mm}
}{
\end{enumerate}
}

\newcommand{\adfmod}[1]{~(\mathrm{mod}~#1)}
\newcommand{\adfhide}[1]{}
\newcommand{\ADFvfyParStart}[1]{\\\ [#1]}
\newcommand{\adfDgap}{\vskip 2mm}              
\newcommand{\adfLgap}{\vskip 0mm}              
\newcommand{\adfsplit}{\par}                   
\newcommand{\adfBfont}{\normalsize}            

\begin{document}
\title{Group divisible designs with block size five: direct constructions}
\author{A. D. Forbes}
\address{LSBU Business School,
London South Bank University,
103 Borough Road,
London SE1 0AA, UK.}
\email{anthony.d.forbes@gmail.com}
\date{version 2.1 \today}

\begin{abstract}
We give direct constructions for 233 group divisible designs with block size five, mostly of type $g^u m^1$, $m > 0$.
\end{abstract}

\maketitle


\section{Introduction}\label{sec:Group divisible designs}
For our purpose, a {\em group divisible design} with block size 5, $5$-GDD,
and type $g_1^{u_1} g_2^{u_2} \dots g_r^{u_r}$ is an ordered triple ($V,\mathcal{G},\mathcal{B}$) such that:
\begin{adfenumerate}
\item[(i)]{$V$
is a base set of cardinality $u_1 g_1 + u_2 g_2 + \dots + u_r g_r$;}
\item[(ii)]{$\mathcal{G}$
is a partition of $V$ into $u_i$ subsets of cardinality $g_i$, $i = 1, 2, \dots, r$, called \textit{groups};}
\item[(iii)]{$\mathcal{B}$
is a non-empty collection of 5-subsets of $V$ called \textit{blocks}; and}
\item[(iv)]{each pair of elements from distinct groups occurs in precisely one block but no pair of
elements from the same group occurs in any block.}
\end{adfenumerate}

The sole purpose of this paper is to prove Theorem~\ref{thm:5-GDD-ADF}, below, which asserts the existence of 233 specific 5-GDDs.
Direct constructions are given in each case, but it is possible that some might arise from known structures.
The paper is cited by \cite{Forbes2022G} to which it is intended to act as a supplement.
We hope it will be of some use to researchers.
\begin{theorem}\label{thm:5-GDD-ADF}
There exist $5$-$\mathrm{GDD}$s of types
$ 1^{48} 9^1 $,   
$ 1^{60} 9^1 $,   
$ 1^{60} 13^1 $,   
$ 1^{68} 9^1 $,   
$ 1^{72} 17^1 $,   
$ 1^{80} 9^1 $,   
$ 1^{80} 13^1 $,   
$ 1^{80} 25^1 $,   
$ 1^{84} 21^1 $,   
$ 1^{88} 9^1 $,   
$ 1^{92} 17^1 $,   
$ 1^{96} 25^1 $,   
$ 1^{100} 9^1 $,   
$ 1^{100} 13^1 $,   
$ 1^{100} 25^1 $,   
$ 1^{104} 21^1 $,   
$ 1^{108} 9^1 $,   
$ 1^{108} 29^1 $,   
$ 1^{112} 17^1 $,   
$ 1^{124} 21^1 $,   
$ 1^{128} 9^1 $,   
$ 1^{128} 29^1 $,   
$ 1^{132} 17^1 $,   
$ 1^{132} 37^1 $,   
$ 1^{136} 25^1 $,   
$ 1^{144} 21^1 $,   
$ 1^{144} 41^1 $,   
$ 1^{148} 29^1 $,   
$ 1^{152} 17^1 $,   
$ 1^{152} 37^1 $,   
$ 1^{156} 25^1 $,   
$ 1^{156} 45^1 $,   
$ 1^{160} 33^1 $,   
$ 1^{164} 21^1 $,   
$ 1^{164} 41^1 $,   
$ 1^{168} 29^1 $,   
$ 1^{168} 49^1 $,   
$ 1^{172} 17^1 $,   
$ 1^{176} 25^1 $,   
$ 1^{176} 45^1 $,   
$ 1^{184} 21^1 $,   
$ 1^{184} 41^1 $,   
$ 1^{188} 29^1 $,   
$ 1^{192} 17^1 $,   
$ 1^{192} 37^1 $,   
$ 1^{192} 57^1 $,   
$ 1^{196} 25^1 $,   
$ 1^{196} 45^1 $,   
$ 2^{32} 14^1 $,   
$ 2^{40} 6^1 $,   %
$ 2^{48} 18^1 $,   
$ 2^{52} 14^1 $,   
$ 2^{56} 10^1 $,   
$ 2^{60} 6^1 $,   
$ 3^{20} 11^1 $,   
$ 3^{28} 7^1 $,   
$ 3^{32} 11^1 $,   
$ 3^{36} 15^1 $,   
$ 5^{24} 25^1 $,   
$ 5^{28} 25^1 $,   
$ 5^{32} 25^1 $,   
$ 5^{32} 45^1 $,   
$ 5^{36} 45^1 $,   
$ 5^{40} 45^1 $,   
$ 5^{44} 25^1 $,   
$ 5^{44} 45^1 $,   
$ 6^{16} 10^1 $,   
$ 8^8 12^1 $,   
$ 8^{10} 16^1 $,   
$ 8^{10} 20^1 $,   
$ 8^{14} 28^1 $,   
$ 8^{15} 4^1 $,   
$ 8^{15} 16^1 $,   
$ 8^{15} 24^1 $,   
$ 8^{15} 36^1 $,   
$ 8^{16} 20^1 $,   
$ 8^{17} 16^1 $,   
$ 8^{18} 32^1 $,   
$ 8^{20} 44^1 $,   
$ 8^{21} 20^1 $,   
$ 8^{21} 40^1 $,   
$ 8^{22} 16^1 $,   
$ 8^{22} 36^1 $,   
$ 8^{23} 32^1 $,   
$ 8^{24} 28^1 $,   
$ 8^{25} 4^1 $,   
$ 8^{25} 24^1 $,   
$ 8^{26} 20^1 $,   
$ 8^{27} 16^1 $,   
$ 8^{28} 12^1 $,   
$ 9^8 1^1 $,   
$ 9^{12} 13^1 $,   
$ 9^{16} 5^1 $,   
$ 9^{16} 25^1 $,   
$ 9^{20} 17^1 $,   
$ 9^{20} 29^1 $,   
$ 9^{20} 37^1 $,   
$ 9^{20} 49^1 $,   
$ 9^{28} 1^1 $,   
$ 10^{10} 18^1 $,   %
$ 10^{20} 38^1 $,   
$ 11^{20} 19^1 $,   
$ 12^5 8^1 $,   
$ 12^{10} 8^1 $,   
$ 12^{10} 16^1 $,   
$ 12^{10} 28^1 $,   
$ 12^{11} 20^1 $,   
$ 12^{12} 4^1 $,   
$ 12^{12} 24^1 $,   
$ 12^{13} 8^1 $,   
$ 12^{13} 28^1 $,   
$ 12^{14} 32^1 $,   
$ 12^{15} 8^1 $,   
$ 12^{15} 16^1 $,   
$ 12^{15} 28^1 $,   
$ 12^{15} 36^1 $,   
$ 12^{15} 48^1 $,   
$ 12^{16} 20^1 $,   
$ 12^{16} 40^1 $,   
$ 12^{17} 4^1 $,   
$ 12^{17} 24^1 $,   
$ 12^{17} 44^1 $,   
$ 12^{18} 8^1 $,   
$ 12^{18} 28^1 $,   
$ 12^{18} 48^1 $,   
$ 12^{19} 32^1 $,   
$ 12^{19} 52^1 $,   
$ 12^{20} 8^1 $,   
$ 12^{21} 20^1 $,   
$ 12^{21} 40^1 $,   
$ 12^{21} 60^1 $,   
$ 12^{22} 4^1 $,   
$ 12^{22} 24^1 $,   
$ 12^{22} 44^1 $,   
$ 12^{22} 64^1 $,   
$ 12^{23} 28^1 $,   
$ 12^{23} 48^1 $,   
$ 12^{23} 68^1 $,   
$ 12^{24} 32^1 $,   
$ 12^{24} 52^1 $,   
$ 12^{27} 4^1 $,   
$ 13^8 17^1 $,   
$ 13^{12} 1^1 $,   
$ 13^{12} 21^1 $,   
$ 13^{12} 41^1 $,   
$ 13^{16} 5^1 $,   
$ 13^{16} 25^1 $,   
$ 13^{20} 1^1 $,   
$ 14^8 6^1 $,   
$ 16^6 20^1 $,   
$ 16^7 12^1 $,   
$ 16^8 4^1 $,   
$ 16^9 36^1 $,   
$ 16^{10} 8^1 $,   
$ 16^{10} 20^1 $,   
$ 16^{10} 28^1 $,   
$ 16^{10} 36^1 $,   
$ 16^{10} 40^1 $,   
$ 16^{11} 20^1 $,   
$ 16^{11} 40^1 $,   
$ 16^{12} 12^1 $,   
$ 16^{12} 52^1 $,   
$ 16^{13} 4^1 $,   
$ 16^{13} 24^1 $,   
$ 16^{13} 44^1 $,   
$ 16^{14} 36^1 $,   
$ 16^{15} 8^1 $,   
$ 16^{15} 28^1 $,   
$ 16^{16} 20^1 $,   
$ 16^{16} 40^1 $,   
$ 16^{17} 12^1 $,   
$ 17^8 13^1 $,   
$ 17^8 33^1 $,   
$ 17^{12} 9^1 $,   
$ 17^{12} 29^1 $,   
$ 17^{12} 49^1 $,   
$ 17^{16} 5^1 $,   
$ 20^8 40^1 $,   
$ 20^9 40^1 $,   
$ 20^{10} 36^1 $,   %
$ 20^{10} 40^1 $,   
$ 20^{11} 40^1 $,   
$ 21^8 9^1 $,   
$ 21^8 29^1 $,   
$ 21^{12} 17^1 $,   
$ 23^8 7^1 $,   
$ 24^6 20^1 $,   
$ 24^7 8^1 $,   %
$ 24^7 28^1 $,   
$ 24^8 16^1 $,   
$ 24^8 36^1 $,   
$ 24^9 4^1 $,   
$ 24^9 44^1 $,   
$ 24^{10} 4^1 $,   
$ 24^{10} 12^1 $,   
$ 24^{10} 32^1 $,   
$ 24^{11} 20^1 $,   
$ 25^8 5^1 $,   
$ 25^8 45^1 $,   
$ 28^6 20^1 $,   
$ 28^6 40^1 $,   
$ 28^7 16^1 $,   
$ 28^7 36^1 $,   
$ 28^8 12^1 $,   
$ 28^8 32^1 $,   
$ 28^8 52^1 $,   
$ 28^9 8^1 $,   
$ 28^9 48^1 $,   
$ 28^{10} 4^1 $,   
$ 28^{10} 24^1 $,   
$ 29^8 1^1 $,   
$ 29^8 21^1 $,   
$ 32^6 20^1 $,   
$ 32^6 40^1 $,   
$ 32^7 4^1 $,   
$ 32^7 24^1 $,   
$ 32^7 44^1 $,   
$ 32^8 8^1 $,   
$ 32^8 28^1 $,   
$ 32^9 12^1 $,   
$ 36^6 20^1 $,   
$ 36^6 40^1 $,   
$ 36^7 12^1 $,   
$ 36^7 32^1 $,   
$ 36^8 24^1 $,   
$ 40^6 20^1 $,   
$ 44^6 20^1 $,   
$ 44^6 40^1 $,   
$ 44^7 8^1 $,   
$ 48^6 20^1 $,   
$ 1^{16} 9^5 $,   
$ 1^{46} 5^3 $ and    
$ 4^5 8^5 $.   


\end{theorem}

\begin{proof}
\adfDgap
\noindent{\boldmath $ 1^{48} 9^{1} $}~
With the point set $\{0, 1, \dots, 56\}$ partitioned into
 residue classes modulo $48$ for $\{0, 1, \dots, 47\}$, and
 $\{48, 49, \dots, 56\}$,
 the design is generated from

\adfLgap {\adfBfont
$\{0, 1, 3, 11, 32\}$,
$\{0, 4, 18, 43, 52\}$,
$\{0, 6, 13, 28, 51\}$,\adfsplit
$\{0, 12, 24, 36, 56\}$

}
\adfLgap \noindent by the mapping:
$x \mapsto x +  j \adfmod{48}$ for $x < 48$,
$x \mapsto (x +  j \adfmod{8}) + 48$ for $48 \le x < 56$,
$56 \mapsto 56$,
$0 \le j < 48$
 for the first three blocks,
$0 \le j < 12$
 for the last block.
\ADFvfyParStart{(57, ((3, 48, ((48, 1), (8, 1), (1, 1))), (1, 12, ((48, 1), (8, 1), (1, 1)))), ((1, 48), (9, 1)))} 

\adfDgap
\noindent{\boldmath $ 1^{60} 9^{1} $}~
With the point set $\{0, 1, \dots, 68\}$ partitioned into
 residue classes modulo $60$ for $\{0, 1, \dots, 59\}$, and
 $\{60, 61, \dots, 68\}$,
 the design is generated from

\adfLgap {\adfBfont
$\{0, 1, 7, 58, 66\}$,
$\{0, 5, 39, 41, 44\}$,
$\{0, 8, 31, 43, 59\}$,\adfsplit
$\{0, 4, 17, 37, 50\}$,
$\{0, 11, 22, 42, 60\}$,
$\{0, 6, 25, 32, 62\}$,\adfsplit
$\{1, 5, 15, 23, 65\}$,
$\{0, 15, 30, 45, 68\}$,
$\{0, 12, 24, 36, 48\}$

}
\adfLgap \noindent by the mapping:
$x \mapsto x + 2 j \adfmod{60}$ for $x < 60$,
$x \mapsto (x +  j \adfmod{6}) + 60$ for $60 \le x < 66$,
$x \mapsto (x +  j \adfmod{2}) + 66$ for $66 \le x < 68$,
$68 \mapsto 68$,
$0 \le j < 30$
 for the first seven blocks,
$0 \le j < 15$
 for the next block,
$0 \le j < 6$
 for the last block.
\ADFvfyParStart{(69, ((7, 30, ((60, 2), (6, 1), (2, 1), (1, 1))), (1, 15, ((60, 2), (6, 1), (2, 1), (1, 1))), (1, 6, ((60, 2), (6, 1), (2, 1), (1, 1)))), ((1, 60), (9, 1)))} 

\adfDgap
\noindent{\boldmath $ 1^{60} 13^{1} $}~
With the point set $\{0, 1, \dots, 72\}$ partitioned into
 residue classes modulo $60$ for $\{0, 1, \dots, 59\}$, and
 $\{60, 61, \dots, 72\}$,
 the design is generated from

\adfLgap {\adfBfont
$\{0, 2, 34, 59, 61\}$,
$\{0, 10, 21, 23, 67\}$,
$\{0, 1, 35, 41, 65\}$,\adfsplit
$\{0, 3, 12, 19, 55\}$,
$\{0, 5, 16, 36, 53\}$,
$\{0, 6, 39, 52, 66\}$,\adfsplit
$\{0, 4, 22, 31, 68\}$,
$\{1, 5, 15, 43, 60\}$,
$\{0, 15, 30, 45, 72\}$

}
\adfLgap \noindent by the mapping:
$x \mapsto x + 2 j \adfmod{60}$ for $x < 60$,
$x \mapsto (x + 2 j \adfmod{12}) + 60$ for $60 \le x < 72$,
$72 \mapsto 72$,
$0 \le j < 30$
 for the first eight blocks,
$0 \le j < 15$
 for the last block.
\ADFvfyParStart{(73, ((8, 30, ((60, 2), (12, 2), (1, 1))), (1, 15, ((60, 2), (12, 2), (1, 1)))), ((1, 60), (13, 1)))} 

\adfDgap
\noindent{\boldmath $ 1^{68} 9^{1} $}~
With the point set $\{0, 1, \dots, 76\}$ partitioned into
 residue classes modulo $68$ for $\{0, 1, \dots, 67\}$, and
 $\{68, 69, \dots, 76\}$,
 the design is generated from

\adfLgap {\adfBfont
$\{0, 1, 7, 66, 69\}$,
$\{0, 5, 15, 62, 71\}$,
$\{0, 10, 41, 67, 70\}$,\adfsplit
$\{0, 19, 33, 38, 68\}$,
$\{0, 4, 12, 26, 44\}$,
$\{0, 20, 43, 55, 59\}$,\adfsplit
$\{0, 25, 27, 47, 65\}$,
$\{0, 13, 37, 45, 52\}$,
$\{0, 17, 34, 51, 76\}$

}
\adfLgap \noindent by the mapping:
$x \mapsto x + 2 j \adfmod{68}$ for $x < 68$,
$x \mapsto (x - 68 + 4 j \adfmod{8}) + 68$ for $68 \le x < 76$,
$76 \mapsto 76$,
$0 \le j < 34$
 for the first eight blocks,
$0 \le j < 17$
 for the last block.
\ADFvfyParStart{(77, ((8, 34, ((68, 2), (8, 4), (1, 1))), (1, 17, ((68, 2), (8, 4), (1, 1)))), ((1, 68), (9, 1)))} 

\adfDgap
\noindent{\boldmath $ 1^{72} 17^{1} $}~
With the point set $\{0, 1, \dots, 88\}$ partitioned into
 residue classes modulo $72$ for $\{0, 1, \dots, 71\}$, and
 $\{72, 73, \dots, 88\}$,
 the design is generated from

\adfLgap {\adfBfont
$\{0, 1, 3, 66, 84\}$,
$\{0, 5, 17, 27, 51\}$,
$\{0, 8, 39, 52, 76\}$,\adfsplit
$\{0, 14, 30, 49, 83\}$,
$\{0, 4, 15, 47, 78\}$,
$\{0, 18, 36, 54, 88\}$

}
\adfLgap \noindent by the mapping:
$x \mapsto x +  j \adfmod{72}$ for $x < 72$,
$x \mapsto (x +  j \adfmod{12}) + 72$ for $72 \le x < 84$,
$x \mapsto (x +  j \adfmod{4}) + 84$ for $84 \le x < 88$,
$88 \mapsto 88$,
$0 \le j < 72$
 for the first five blocks,
$0 \le j < 18$
 for the last block.
\ADFvfyParStart{(89, ((5, 72, ((72, 1), (12, 1), (4, 1), (1, 1))), (1, 18, ((72, 1), (12, 1), (4, 1), (1, 1)))), ((1, 72), (17, 1)))} 

\adfDgap
\noindent{\boldmath $ 1^{80} 9^{1} $}~
With the point set $\{0, 1, \dots, 88\}$ partitioned into
 residue classes modulo $80$ for $\{0, 1, \dots, 79\}$, and
 $\{80, 81, \dots, 88\}$,
 the design is generated from

\adfLgap {\adfBfont
$\{0, 47, 51, 59, 80\}$,
$\{0, 1, 74, 78, 82\}$,
$\{0, 5, 19, 70, 87\}$,\adfsplit
$\{0, 21, 45, 58, 86\}$,
$\{0, 8, 34, 65, 71\}$,
$\{0, 9, 35, 62, 79\}$,\adfsplit
$\{0, 12, 25, 36, 50\}$,
$\{0, 11, 33, 49, 52\}$,
$\{0, 23, 41, 73, 75\}$,\adfsplit
$\{0, 20, 40, 60, 88\}$,
$\{1, 21, 41, 61, 88\}$,
$\{0, 16, 32, 48, 64\}$

}
\adfLgap \noindent by the mapping:
$x \mapsto x + 2 j \adfmod{80}$ for $x < 80$,
$x \mapsto (x +  j \adfmod{8}) + 80$ for $80 \le x < 88$,
$88 \mapsto 88$,
$0 \le j < 40$
 for the first nine blocks,
$0 \le j < 10$
 for the next two blocks,
$0 \le j < 8$
 for the last block.
\ADFvfyParStart{(89, ((9, 40, ((80, 2), (8, 1), (1, 1))), (2, 10, ((80, 2), (8, 1), (1, 1))), (1, 8, ((80, 2), (8, 1), (1, 1)))), ((1, 80), (9, 1)))} 

\adfDgap
\noindent{\boldmath $ 1^{80} 13^{1} $}~
With the point set $\{0, 1, \dots, 92\}$ partitioned into
 residue classes modulo $80$ for $\{0, 1, \dots, 79\}$, and
 $\{80, 81, \dots, 92\}$,
 the design is generated from

\adfLgap {\adfBfont
$\{0, 1, 3, 74, 88\}$,
$\{0, 4, 21, 47, 52\}$,
$\{0, 8, 22, 35, 64\}$,\adfsplit
$\{0, 10, 46, 65, 82\}$,
$\{0, 11, 23, 41, 86\}$,
$\{0, 20, 40, 60, 92\}$

}
\adfLgap \noindent by the mapping:
$x \mapsto x +  j \adfmod{80}$ for $x < 80$,
$x \mapsto (x +  j \adfmod{8}) + 80$ for $80 \le x < 88$,
$x \mapsto (x +  j \adfmod{4}) + 88$ for $88 \le x < 92$,
$92 \mapsto 92$,
$0 \le j < 80$
 for the first five blocks,
$0 \le j < 20$
 for the last block.
\ADFvfyParStart{(93, ((5, 80, ((80, 1), (8, 1), (4, 1), (1, 1))), (1, 20, ((80, 1), (8, 1), (4, 1), (1, 1)))), ((1, 80), (13, 1)))} 

\adfDgap
\noindent{\boldmath $ 1^{80} 25^{1} $}~
With the point set $\{0, 1, \dots, 104\}$ partitioned into
 residue classes modulo $80$ for $\{0, 1, \dots, 79\}$, and
 $\{80, 81, \dots, 104\}$,
 the design is generated from

\adfLgap {\adfBfont
$\{0, 5, 38, 49, 98\}$,
$\{0, 1, 4, 23, 102\}$,
$\{0, 6, 52, 73, 96\}$,\adfsplit
$\{0, 9, 26, 50, 91\}$,
$\{0, 2, 14, 29, 88\}$,
$\{0, 8, 18, 43, 95\}$,\adfsplit
$\{0, 20, 40, 60, 100\}$,
$\{0, 16, 32, 48, 64\}$

}
\adfLgap \noindent by the mapping:
$x \mapsto x +  j \adfmod{80}$ for $x < 80$,
$x \mapsto (x +  j \adfmod{20}) + 80$ for $80 \le x < 100$,
$x \mapsto (x +  j \adfmod{5}) + 100$ for $x \ge 100$,
$0 \le j < 80$
 for the first six blocks,
$0 \le j < 20$
 for the next block,
$0 \le j < 16$
 for the last block.
\ADFvfyParStart{(105, ((6, 80, ((80, 1), (20, 1), (5, 1))), (1, 20, ((80, 1), (20, 1), (5, 1))), (1, 16, ((80, 1), (20, 1), (5, 1)))), ((1, 80), (25, 1)))} 

\adfDgap
\noindent{\boldmath $ 1^{84} 21^{1} $}~
With the point set $\{0, 1, \dots, 104\}$ partitioned into
 residue classes modulo $84$ for $\{0, 1, \dots, 83\}$, and
 $\{84, 85, \dots, 104\}$,
 the design is generated from

\adfLgap {\adfBfont
$\{0, 1, 4, 6, 91\}$,
$\{0, 7, 15, 29, 96\}$,
$\{0, 9, 28, 40, 60\}$,\adfsplit
$\{0, 11, 37, 54, 94\}$,
$\{0, 10, 35, 48, 86\}$,
$\{0, 16, 39, 66, 95\}$,\adfsplit
$\{0, 21, 42, 63, 84\}$

}
\adfLgap \noindent by the mapping:
$x \mapsto x +  j \adfmod{84}$ for $x < 84$,
$x \mapsto (x +  j \adfmod{21}) + 84$ for $x \ge 84$,
$0 \le j < 84$
 for the first six blocks,
$0 \le j < 21$
 for the last block.
\ADFvfyParStart{(105, ((6, 84, ((84, 1), (21, 1))), (1, 21, ((84, 1), (21, 1)))), ((1, 84), (21, 1)))} 

\adfDgap
\noindent{\boldmath $ 1^{88} 9^{1} $}~
With the point set $\{0, 1, \dots, 96\}$ partitioned into
 residue classes modulo $88$ for $\{0, 1, \dots, 87\}$, and
 $\{88, 89, \dots, 96\}$,
 the design is generated from

\adfLgap {\adfBfont
$\{0, 1, 3, 11, 15\}$,
$\{0, 7, 23, 42, 63\}$,
$\{0, 9, 27, 47, 64\}$,\adfsplit
$\{0, 5, 31, 59, 88\}$,
$\{0, 6, 45, 58, 92\}$,
$\{0, 22, 44, 66, 96\}$

}
\adfLgap \noindent by the mapping:
$x \mapsto x +  j \adfmod{88}$ for $x < 88$,
$x \mapsto (x +  j \adfmod{8}) + 88$ for $88 \le x < 96$,
$96 \mapsto 96$,
$0 \le j < 88$
 for the first five blocks,
$0 \le j < 22$
 for the last block.
\ADFvfyParStart{(97, ((5, 88, ((88, 1), (8, 1), (1, 1))), (1, 22, ((88, 1), (8, 1), (1, 1)))), ((1, 88), (9, 1)))} 

\adfDgap
\noindent{\boldmath $ 1^{92} 17^{1} $}~
With the point set $\{0, 1, \dots, 108\}$ partitioned into
 residue classes modulo $92$ for $\{0, 1, \dots, 91\}$, and
 $\{92, 93, \dots, 108\}$,
 the design is generated from

\adfLgap {\adfBfont
$\{65, 107, 64, 78, 59\}$,
$\{102, 21, 10, 51, 48\}$,
$\{51, 44, 104, 53, 54\}$,\adfsplit
$\{71, 97, 29, 0, 86\}$,
$\{0, 2, 15, 49, 95\}$,
$\{0, 17, 34, 55, 98\}$,\adfsplit
$\{0, 18, 51, 61, 92\}$,
$\{0, 25, 39, 50, 93\}$,
$\{0, 5, 31, 40, 70\}$,\adfsplit
$\{0, 4, 16, 64, 72\}$,
$\{0, 19, 37, 59, 66\}$,
$\{1, 5, 17, 65, 73\}$,\adfsplit
$\{0, 23, 46, 69, 108\}$

}
\adfLgap \noindent by the mapping:
$x \mapsto x + 2 j \adfmod{92}$ for $x < 92$,
$x \mapsto (x - 92 + 8 j \adfmod{16}) + 92$ for $92 \le x < 108$,
$108 \mapsto 108$,
$0 \le j < 46$
 for the first 12 blocks,
$0 \le j < 23$
 for the last block.
\ADFvfyParStart{(109, ((12, 46, ((92, 2), (16, 8), (1, 1))), (1, 23, ((92, 2), (16, 8), (1, 1)))), ((1, 92), (17, 1)))} 

\adfDgap
\noindent{\boldmath $ 1^{96} 25^{1} $}~
With the point set $\{0, 1, \dots, 120\}$ partitioned into
 residue classes modulo $96$ for $\{0, 1, \dots, 95\}$, and
 $\{96, 97, \dots, 120\}$,
 the design is generated from

\adfLgap {\adfBfont
$\{44, 80, 52, 55, 98\}$,
$\{100, 72, 73, 11, 15\}$,
$\{0, 2, 12, 53, 108\}$,\adfsplit
$\{0, 6, 15, 80, 101\}$,
$\{0, 7, 21, 40, 104\}$,
$\{0, 13, 30, 67, 111\}$,\adfsplit
$\{0, 5, 23, 49, 69\}$,
$\{0, 24, 48, 72, 120\}$

}
\adfLgap \noindent by the mapping:
$x \mapsto x +  j \adfmod{96}$ for $x < 96$,
$x \mapsto (x +  j \adfmod{24}) + 96$ for $96 \le x < 120$,
$120 \mapsto 120$,
$0 \le j < 96$
 for the first seven blocks,
$0 \le j < 24$
 for the last block.
\ADFvfyParStart{(121, ((7, 96, ((96, 1), (24, 1), (1, 1))), (1, 24, ((96, 1), (24, 1), (1, 1)))), ((1, 96), (25, 1)))} 

\adfDgap
\noindent{\boldmath $ 1^{100} 9^{1} $}~
With the point set $\{0, 1, \dots, 108\}$ partitioned into
 residue classes modulo $100$ for $\{0, 1, \dots, 99\}$, and
 $\{100, 101, \dots, 108\}$,
 the design is generated from

\adfLgap {\adfBfont
$\{53, 101, 75, 76, 18\}$,
$\{30, 96, 11, 92, 43\}$,
$\{75, 64, 103, 38, 29\}$,\adfsplit
$\{0, 1, 7, 98, 100\}$,
$\{0, 23, 70, 97, 102\}$,
$\{0, 16, 55, 68, 85\}$,\adfsplit
$\{0, 5, 33, 44, 93\}$,
$\{0, 6, 14, 24, 78\}$,
$\{0, 12, 43, 79, 95\}$,\adfsplit
$\{0, 29, 63, 71, 73\}$,
$\{0, 21, 41, 45, 59\}$,
$\{0, 25, 50, 75, 108\}$,\adfsplit
$\{0, 20, 40, 60, 80\}$

}
\adfLgap \noindent by the mapping:
$x \mapsto x + 2 j \adfmod{100}$ for $x < 100$,
$x \mapsto (x - 100 + 4 j \adfmod{8}) + 100$ for $100 \le x < 108$,
$108 \mapsto 108$,
$0 \le j < 50$
 for the first 11 blocks,
$0 \le j < 25$
 for the next block,
$0 \le j < 10$
 for the last block.
\ADFvfyParStart{(109, ((11, 50, ((100, 2), (8, 4), (1, 1))), (1, 25, ((100, 2), (8, 4), (1, 1))), (1, 10, ((100, 2), (8, 4), (1, 1)))), ((1, 100), (9, 1)))} 

\adfDgap
\noindent{\boldmath $ 1^{100} 13^{1} $}~
With the point set $\{0, 1, \dots, 112\}$ partitioned into
 residue classes modulo $100$ for $\{0, 1, \dots, 99\}$, and
 $\{100, 101, \dots, 112\}$,
 the design is generated from

\adfLgap {\adfBfont
$\{107, 32, 94, 3, 9\}$,
$\{5, 105, 91, 32, 10\}$,
$\{0, 1, 43, 46, 100\}$,\adfsplit
$\{0, 2, 10, 26, 66\}$,
$\{0, 4, 32, 53, 65\}$,
$\{0, 7, 18, 70, 87\}$,\adfsplit
$\{0, 25, 50, 75, 112\}$

}
\adfLgap \noindent by the mapping:
$x \mapsto x +  j \adfmod{100}$ for $x < 100$,
$x \mapsto (x - 100 + 3 j \adfmod{12}) + 100$ for $100 \le x < 112$,
$112 \mapsto 112$,
$0 \le j < 100$
 for the first six blocks,
$0 \le j < 25$
 for the last block.
\ADFvfyParStart{(113, ((6, 100, ((100, 1), (12, 3), (1, 1))), (1, 25, ((100, 1), (12, 3), (1, 1)))), ((1, 100), (13, 1)))} 

\adfDgap
\noindent{\boldmath $ 1^{100} 25^{1} $}~
With the point set $\{0, 1, \dots, 124\}$ partitioned into
 residue classes modulo $100$ for $\{0, 1, \dots, 99\}$, and
 $\{100, 101, \dots, 124\}$,
 the design is generated from

\adfLgap {\adfBfont
$\{104, 90, 92, 3, 36\}$,
$\{49, 71, 13, 97, 103\}$,
$\{0, 1, 6, 83, 111\}$,\adfsplit
$\{0, 9, 43, 90, 117\}$,
$\{0, 7, 15, 45, 86\}$,
$\{0, 12, 39, 63, 109\}$,\adfsplit
$\{0, 3, 68, 72, 116\}$,
$\{0, 25, 50, 75, 100\}$,
$\{0, 20, 40, 60, 80\}$

}
\adfLgap \noindent by the mapping:
$x \mapsto x +  j \adfmod{100}$ for $x < 100$,
$x \mapsto (x +  j \adfmod{25}) + 100$ for $x \ge 100$,
$0 \le j < 100$
 for the first seven blocks,
$0 \le j < 25$
 for the next block,
$0 \le j < 20$
 for the last block.
\ADFvfyParStart{(125, ((7, 100, ((100, 1), (25, 1))), (1, 25, ((100, 1), (25, 1))), (1, 20, ((100, 1), (25, 1)))), ((1, 100), (25, 1)))} 

\adfDgap
\noindent{\boldmath $ 1^{104} 21^{1} $}~
With the point set $\{0, 1, \dots, 124\}$ partitioned into
 residue classes modulo $104$ for $\{0, 1, \dots, 103\}$, and
 $\{104, 105, \dots, 124\}$,
 the design is generated from

\adfLgap {\adfBfont
$\{22, 57, 113, 43, 56\}$,
$\{56, 121, 102, 71, 93\}$,
$\{0, 2, 5, 43, 105\}$,\adfsplit
$\{0, 7, 25, 54, 107\}$,
$\{0, 11, 30, 53, 104\}$,
$\{0, 6, 16, 33, 65\}$,\adfsplit
$\{0, 4, 28, 40, 48\}$,
$\{0, 26, 52, 78, 124\}$

}
\adfLgap \noindent by the mapping:
$x \mapsto x +  j \adfmod{104}$ for $x < 104$,
$x \mapsto (x - 104 + 5 j \adfmod{20}) + 104$ for $104 \le x < 124$,
$124 \mapsto 124$,
$0 \le j < 104$
 for the first seven blocks,
$0 \le j < 26$
 for the last block.
\ADFvfyParStart{(125, ((7, 104, ((104, 1), (20, 5), (1, 1))), (1, 26, ((104, 1), (20, 5), (1, 1)))), ((1, 104), (21, 1)))} 

\adfDgap
\noindent{\boldmath $ 1^{108} 9^{1} $}~
With the point set $\{0, 1, \dots, 116\}$ partitioned into
 residue classes modulo $108$ for $\{0, 1, \dots, 107\}$, and
 $\{108, 109, \dots, 116\}$,
 the design is generated from

\adfLgap {\adfBfont
$\{83, 81, 102, 40, 111\}$,
$\{69, 60, 98, 112, 47\}$,
$\{0, 1, 5, 31, 49\}$,\adfsplit
$\{0, 8, 28, 42, 53\}$,
$\{0, 3, 15, 39, 76\}$,
$\{0, 6, 16, 23, 56\}$,\adfsplit
$\{0, 27, 54, 81, 116\}$

}
\adfLgap \noindent by the mapping:
$x \mapsto x +  j \adfmod{108}$ for $x < 108$,
$x \mapsto (x - 108 + 2 j \adfmod{8}) + 108$ for $108 \le x < 116$,
$116 \mapsto 116$,
$0 \le j < 108$
 for the first six blocks,
$0 \le j < 27$
 for the last block.
\ADFvfyParStart{(117, ((6, 108, ((108, 1), (8, 2), (1, 1))), (1, 27, ((108, 1), (8, 2), (1, 1)))), ((1, 108), (9, 1)))} 

\adfDgap
\noindent{\boldmath $ 1^{108} 29^{1} $}~
With the point set $\{0, 1, \dots, 136\}$ partitioned into
 residue classes modulo $108$ for $\{0, 1, \dots, 107\}$, and
 $\{108, 109, \dots, 136\}$,
 the design is generated from

\adfLgap {\adfBfont
$\{37, 103, 12, 124, 30\}$,
$\{98, 55, 126, 104, 53\}$,
$\{76, 127, 71, 102, 105\}$,\adfsplit
$\{0, 9, 19, 30, 121\}$,
$\{0, 4, 62, 75, 112\}$,
$\{0, 8, 20, 48, 84\}$,\adfsplit
$\{0, 1, 15, 56, 114\}$,
$\{0, 16, 39, 86, 119\}$,
$\{0, 27, 54, 81, 136\}$

}
\adfLgap \noindent by the mapping:
$x \mapsto x +  j \adfmod{108}$ for $x < 108$,
$x \mapsto (x +  j \adfmod{12}) + 108$ for $108 \le x < 120$,
$x \mapsto (x - 120 + 4 j \adfmod{16}) + 120$ for $120 \le x < 136$,
$136 \mapsto 136$,
$0 \le j < 108$
 for the first eight blocks,
$0 \le j < 27$
 for the last block.
\ADFvfyParStart{(137, ((8, 108, ((108, 1), (12, 1), (16, 4), (1, 1))), (1, 27, ((108, 1), (12, 1), (16, 4), (1, 1)))), ((1, 108), (29, 1)))} 

\adfDgap
\noindent{\boldmath $ 1^{112} 17^{1} $}~
With the point set $\{0, 1, \dots, 128\}$ partitioned into
 residue classes modulo $112$ for $\{0, 1, \dots, 111\}$, and
 $\{112, 113, \dots, 128\}$,
 the design is generated from

\adfLgap {\adfBfont
$\{104, 2, 96, 103, 5\}$,
$\{86, 15, 27, 10, 113\}$,
$\{0, 2, 32, 48, 70\}$,\adfsplit
$\{0, 4, 19, 69, 115\}$,
$\{0, 6, 45, 85, 122\}$,
$\{0, 20, 51, 75, 124\}$,\adfsplit
$\{0, 9, 34, 63, 86\}$,
$\{0, 28, 56, 84, 128\}$

}
\adfLgap \noindent by the mapping:
$x \mapsto x +  j \adfmod{112}$ for $x < 112$,
$x \mapsto (x +  j \adfmod{16}) + 112$ for $112 \le x < 128$,
$128 \mapsto 128$,
$0 \le j < 112$
 for the first seven blocks,
$0 \le j < 28$
 for the last block.
\ADFvfyParStart{(129, ((7, 112, ((112, 1), (16, 1), (1, 1))), (1, 28, ((112, 1), (16, 1), (1, 1)))), ((1, 112), (17, 1)))} 

\adfDgap
\noindent{\boldmath $ 1^{124} 21^{1} $}~
With the point set $\{0, 1, \dots, 144\}$ partitioned into
 residue classes modulo $124$ for $\{0, 1, \dots, 123\}$, and
 $\{124, 125, \dots, 144\}$,
 the design is generated from

\adfLgap {\adfBfont
$\{64, 23, 66, 136, 117\}$,
$\{122, 16, 125, 119, 9\}$,
$\{0, 1, 6, 35, 124\}$,\adfsplit
$\{0, 9, 55, 70, 127\}$,
$\{0, 10, 23, 97, 128\}$,
$\{0, 16, 40, 65, 98\}$,\adfsplit
$\{0, 17, 39, 77, 96\}$,
$\{0, 4, 36, 48, 56\}$,
$\{0, 31, 62, 93, 144\}$

}
\adfLgap \noindent by the mapping:
$x \mapsto x +  j \adfmod{124}$ for $x < 124$,
$x \mapsto (x - 124 + 5 j \adfmod{20}) + 124$ for $124 \le x < 144$,
$144 \mapsto 144$,
$0 \le j < 124$
 for the first eight blocks,
$0 \le j < 31$
 for the last block.
\ADFvfyParStart{(145, ((8, 124, ((124, 1), (20, 5), (1, 1))), (1, 31, ((124, 1), (20, 5), (1, 1)))), ((1, 124), (21, 1)))} 

\adfDgap
\noindent{\boldmath $ 1^{128} 9^{1} $}~
With the point set $\{0, 1, \dots, 136\}$ partitioned into
 residue classes modulo $128$ for $\{0, 1, \dots, 127\}$, and
 $\{128, 129, \dots, 136\}$,
 the design is generated from

\adfLgap {\adfBfont
$\{78, 4, 128, 66, 21\}$,
$\{93, 130, 34, 115, 65\}$,
$\{0, 1, 3, 9, 14\}$,\adfsplit
$\{0, 7, 34, 44, 82\}$,
$\{0, 4, 25, 67, 102\}$,
$\{0, 15, 33, 56, 85\}$,\adfsplit
$\{0, 16, 36, 55, 104\}$,
$\{0, 32, 64, 96, 136\}$

}
\adfLgap \noindent by the mapping:
$x \mapsto x +  j \adfmod{128}$ for $x < 128$,
$x \mapsto (x +  j \adfmod{8}) + 128$ for $128 \le x < 136$,
$136 \mapsto 136$,
$0 \le j < 128$
 for the first seven blocks,
$0 \le j < 32$
 for the last block.
\ADFvfyParStart{(137, ((7, 128, ((128, 1), (8, 1), (1, 1))), (1, 32, ((128, 1), (8, 1), (1, 1)))), ((1, 128), (9, 1)))} 

\adfDgap
\noindent{\boldmath $ 1^{128} 29^{1} $}~
With the point set $\{0, 1, \dots, 156\}$ partitioned into
 residue classes modulo $128$ for $\{0, 1, \dots, 127\}$, and
 $\{128, 129, \dots, 156\}$,
 the design is generated from

\adfLgap {\adfBfont
$\{61, 62, 89, 145, 71\}$,
$\{120, 53, 89, 11, 128\}$,
$\{100, 104, 21, 133, 124\}$,\adfsplit
$\{0, 2, 76, 123, 145\}$,
$\{0, 13, 39, 98, 152\}$,
$\{0, 16, 33, 62, 106\}$,\adfsplit
$\{0, 14, 51, 72, 130\}$,
$\{0, 3, 11, 68, 116\}$,
$\{0, 6, 41, 94, 140\}$,\adfsplit
$\{0, 32, 64, 96, 156\}$

}
\adfLgap \noindent by the mapping:
$x \mapsto x +  j \adfmod{128}$ for $x < 128$,
$x \mapsto (x +  j \adfmod{16}) + 128$ for $128 \le x < 144$,
$x \mapsto (x +  j \adfmod{8}) + 144$ for $144 \le x < 152$,
$x \mapsto (x +  j \adfmod{4}) + 152$ for $152 \le x < 156$,
$156 \mapsto 156$,
$0 \le j < 128$
 for the first nine blocks,
$0 \le j < 32$
 for the last block.
\ADFvfyParStart{(157, ((9, 128, ((128, 1), (16, 1), (8, 1), (4, 1), (1, 1))), (1, 32, ((128, 1), (16, 1), (8, 1), (4, 1), (1, 1)))), ((1, 128), (29, 1)))} 

\adfDgap
\noindent{\boldmath $ 1^{132} 17^{1} $}~
With the point set $\{0, 1, \dots, 148\}$ partitioned into
 residue classes modulo $132$ for $\{0, 1, \dots, 131\}$, and
 $\{132, 133, \dots, 148\}$,
 the design is generated from

\adfLgap {\adfBfont
$\{102, 126, 4, 6, 64\}$,
$\{23, 113, 109, 92, 143\}$,
$\{0, 1, 23, 26, 144\}$,\adfsplit
$\{0, 6, 37, 55, 84\}$,
$\{0, 7, 20, 35, 88\}$,
$\{0, 8, 73, 100, 141\}$,\adfsplit
$\{0, 9, 50, 89, 136\}$,
$\{0, 5, 16, 61, 118\}$,
$\{0, 33, 66, 99, 148\}$

}
\adfLgap \noindent by the mapping:
$x \mapsto x +  j \adfmod{132}$ for $x < 132$,
$x \mapsto (x +  j \adfmod{12}) + 132$ for $132 \le x < 144$,
$x \mapsto (x +  j \adfmod{4}) + 144$ for $144 \le x < 148$,
$148 \mapsto 148$,
$0 \le j < 132$
 for the first eight blocks,
$0 \le j < 33$
 for the last block.
\ADFvfyParStart{(149, ((8, 132, ((132, 1), (12, 1), (4, 1), (1, 1))), (1, 33, ((132, 1), (12, 1), (4, 1), (1, 1)))), ((1, 132), (17, 1)))} 

\adfDgap
\noindent{\boldmath $ 1^{132} 37^{1} $}~
With the point set $\{0, 1, \dots, 168\}$ partitioned into
 residue classes modulo $132$ for $\{0, 1, \dots, 131\}$, and
 $\{132, 133, \dots, 168\}$,
 the design is generated from

\adfLgap {\adfBfont
$\{29, 49, 145, 109, 99\}$,
$\{4, 108, 123, 2, 10\}$,
$\{158, 73, 72, 131, 30\}$,\adfsplit
$\{0, 3, 49, 54, 165\}$,
$\{0, 4, 69, 91, 134\}$,
$\{0, 7, 68, 95, 143\}$,\adfsplit
$\{0, 12, 30, 115, 137\}$,
$\{0, 21, 57, 97, 148\}$,
$\{0, 9, 32, 116, 155\}$,\adfsplit
$\{0, 14, 53, 108, 153\}$,
$\{0, 33, 66, 99, 132\}$

}
\adfLgap \noindent by the mapping:
$x \mapsto x +  j \adfmod{132}$ for $x < 132$,
$x \mapsto (x +  j \adfmod{33}) + 132$ for $132 \le x < 165$,
$x \mapsto (x - 165 +  j \adfmod{4}) + 165$ for $x \ge 165$,
$0 \le j < 132$
 for the first ten blocks,
$0 \le j < 33$
 for the last block.
\ADFvfyParStart{(169, ((10, 132, ((132, 1), (33, 1), (4, 1))), (1, 33, ((132, 1), (33, 1), (4, 1)))), ((1, 132), (37, 1)))} 

\adfDgap
\noindent{\boldmath $ 1^{136} 25^{1} $}~
With the point set $\{0, 1, \dots, 160\}$ partitioned into
 residue classes modulo $136$ for $\{0, 1, \dots, 135\}$, and
 $\{136, 137, \dots, 160\}$,
 the design is generated from

\adfLgap {\adfBfont
$\{148, 103, 109, 12, 19\}$,
$\{146, 51, 29, 87, 14\}$,
$\{0, 1, 10, 134, 148\}$,\adfsplit
$\{0, 4, 17, 42, 146\}$,
$\{0, 21, 44, 75, 108\}$,
$\{0, 11, 27, 80, 104\}$,\adfsplit
$\{0, 8, 26, 55, 96\}$,
$\{0, 20, 50, 85, 153\}$,
$\{0, 5, 62, 122, 150\}$,\adfsplit
$\{0, 34, 68, 102, 160\}$

}
\adfLgap \noindent by the mapping:
$x \mapsto x +  j \adfmod{136}$ for $x < 136$,
$x \mapsto (x - 136 + 3 j \adfmod{24}) + 136$ for $136 \le x < 160$,
$160 \mapsto 160$,
$0 \le j < 136$
 for the first nine blocks,
$0 \le j < 34$
 for the last block.
\ADFvfyParStart{(161, ((9, 136, ((136, 1), (24, 3), (1, 1))), (1, 34, ((136, 1), (24, 3), (1, 1)))), ((1, 136), (25, 1)))} 

\adfDgap
\noindent{\boldmath $ 1^{144} 21^{1} $}~
With the point set $\{0, 1, \dots, 164\}$ partitioned into
 residue classes modulo $144$ for $\{0, 1, \dots, 143\}$, and
 $\{144, 145, \dots, 164\}$,
 the design is generated from

\adfLgap {\adfBfont
$\{125, 72, 112, 120, 11\}$,
$\{50, 107, 33, 138, 88\}$,
$\{0, 22, 132, 86, 112\}$,\adfsplit
$\{0, 6, 21, 103, 160\}$,
$\{0, 2, 11, 79, 95\}$,
$\{0, 7, 78, 107, 144\}$,\adfsplit
$\{0, 4, 18, 63, 158\}$,
$\{0, 1, 25, 28, 148\}$,
$\{0, 10, 33, 102, 151\}$,\adfsplit
$\{0, 36, 72, 108, 164\}$

}
\adfLgap \noindent by the mapping:
$x \mapsto x +  j \adfmod{144}$ for $x < 144$,
$x \mapsto (x +  j \adfmod{16}) + 144$ for $144 \le x < 160$,
$x \mapsto (x +  j \adfmod{4}) + 160$ for $160 \le x < 164$,
$164 \mapsto 164$,
$0 \le j < 144$
 for the first nine blocks,
$0 \le j < 36$
 for the last block.
\ADFvfyParStart{(165, ((9, 144, ((144, 1), (16, 1), (4, 1), (1, 1))), (1, 36, ((144, 1), (16, 1), (4, 1), (1, 1)))), ((1, 144), (21, 1)))} 

\adfDgap
\noindent{\boldmath $ 1^{144} 41^{1} $}~
With the point set $\{0, 1, \dots, 184\}$ partitioned into
 residue classes modulo $144$ for $\{0, 1, \dots, 143\}$, and
 $\{144, 145, \dots, 184\}$,
 the design is generated from

\adfLgap {\adfBfont
$\{147, 45, 1, 70, 38\}$,
$\{26, 166, 4, 38, 99\}$,
$\{175, 93, 117, 94, 53\}$,\adfsplit
$\{107, 23, 13, 56, 91\}$,
$\{0, 2, 5, 131, 180\}$,
$\{0, 9, 62, 114, 169\}$,\adfsplit
$\{0, 4, 21, 67, 172\}$,
$\{0, 6, 96, 124, 171\}$,
$\{0, 8, 27, 65, 152\}$,\adfsplit
$\{0, 11, 58, 113, 178\}$,
$\{0, 14, 70, 99, 148\}$,
$\{0, 36, 72, 108, 184\}$

}
\adfLgap \noindent by the mapping:
$x \mapsto x +  j \adfmod{144}$ for $x < 144$,
$x \mapsto (x +  j \adfmod{36}) + 144$ for $144 \le x < 180$,
$x \mapsto (x +  j \adfmod{4}) + 180$ for $180 \le x < 184$,
$184 \mapsto 184$,
$0 \le j < 144$
 for the first 11 blocks,
$0 \le j < 36$
 for the last block.
\ADFvfyParStart{(185, ((11, 144, ((144, 1), (36, 1), (4, 1), (1, 1))), (1, 36, ((144, 1), (36, 1), (4, 1), (1, 1)))), ((1, 144), (41, 1)))} 

\adfDgap
\noindent{\boldmath $ 1^{148} 29^{1} $}~
With the point set $\{0, 1, \dots, 176\}$ partitioned into
 residue classes modulo $148$ for $\{0, 1, \dots, 147\}$, and
 $\{148, 149, \dots, 176\}$,
 the design is generated from

\adfLgap {\adfBfont
$\{154, 90, 12, 33, 55\}$,
$\{148, 51, 54, 52, 45\}$,
$\{105, 34, 60, 7, 158\}$,\adfsplit
$\{0, 5, 19, 30, 149\}$,
$\{0, 10, 41, 75, 152\}$,
$\{0, 15, 33, 82, 153\}$,\adfsplit
$\{0, 17, 63, 86, 150\}$,
$\{0, 4, 42, 55, 94\}$,
$\{0, 12, 40, 72, 101\}$,\adfsplit
$\{0, 8, 24, 44, 92\}$,
$\{0, 37, 74, 111, 176\}$

}
\adfLgap \noindent by the mapping:
$x \mapsto x +  j \adfmod{148}$ for $x < 148$,
$x \mapsto (x - 148 + 7 j \adfmod{28}) + 148$ for $148 \le x < 176$,
$176 \mapsto 176$,
$0 \le j < 148$
 for the first ten blocks,
$0 \le j < 37$
 for the last block.
\ADFvfyParStart{(177, ((10, 148, ((148, 1), (28, 7), (1, 1))), (1, 37, ((148, 1), (28, 7), (1, 1)))), ((1, 148), (29, 1)))} 

\adfDgap
\noindent{\boldmath $ 1^{152} 17^{1} $}~
With the point set $\{0, 1, \dots, 168\}$ partitioned into
 residue classes modulo $152$ for $\{0, 1, \dots, 151\}$, and
 $\{152, 153, \dots, 168\}$,
 the design is generated from

\adfLgap {\adfBfont
$\{83, 55, 96, 53, 38\}$,
$\{31, 4, 131, 165, 78\}$,
$\{64, 68, 152, 31, 9\}$,\adfsplit
$\{0, 1, 19, 92, 164\}$,
$\{0, 9, 21, 98, 165\}$,
$\{0, 10, 44, 80, 112\}$,\adfsplit
$\{0, 7, 31, 88, 117\}$,
$\{0, 3, 8, 14, 104\}$,
$\{0, 16, 39, 65, 85\}$,\adfsplit
$\{0, 38, 76, 114, 168\}$

}
\adfLgap \noindent by the mapping:
$x \mapsto x +  j \adfmod{152}$ for $x < 152$,
$x \mapsto (x - 152 + 2 j \adfmod{16}) + 152$ for $152 \le x < 168$,
$168 \mapsto 168$,
$0 \le j < 152$
 for the first nine blocks,
$0 \le j < 38$
 for the last block.
\ADFvfyParStart{(169, ((9, 152, ((152, 1), (16, 2), (1, 1))), (1, 38, ((152, 1), (16, 2), (1, 1)))), ((1, 152), (17, 1)))} 

\adfDgap
\noindent{\boldmath $ 1^{152} 37^{1} $}~
With the point set $\{0, 1, \dots, 188\}$ partitioned into
 residue classes modulo $152$ for $\{0, 1, \dots, 151\}$, and
 $\{152, 153, \dots, 188\}$,
 the design is generated from

\adfLgap {\adfBfont
$\{181, 146, 57, 149, 84\}$,
$\{72, 115, 154, 82, 101\}$,
$\{148, 97, 187, 114, 31\}$,\adfsplit
$\{82, 7, 166, 137, 52\}$,
$\{0, 1, 5, 58, 157\}$,
$\{0, 8, 24, 49, 64\}$,\adfsplit
$\{0, 6, 18, 54, 126\}$,
$\{0, 9, 46, 93, 156\}$,
$\{0, 7, 20, 81, 176\}$,\adfsplit
$\{0, 11, 42, 124, 171\}$,
$\{0, 2, 23, 102, 175\}$,
$\{0, 38, 76, 114, 188\}$

}
\adfLgap \noindent by the mapping:
$x \mapsto x +  j \adfmod{152}$ for $x < 152$,
$x \mapsto (x - 152 + 4 j \adfmod{32}) + 152$ for $152 \le x < 184$,
$x \mapsto (x +  j \adfmod{4}) + 184$ for $184 \le x < 188$,
$188 \mapsto 188$,
$0 \le j < 152$
 for the first 11 blocks,
$0 \le j < 38$
 for the last block.
\ADFvfyParStart{(189, ((11, 152, ((152, 1), (32, 4), (4, 1), (1, 1))), (1, 38, ((152, 1), (32, 4), (4, 1), (1, 1)))), ((1, 152), (37, 1)))} 

\adfDgap
\noindent{\boldmath $ 1^{156} 25^{1} $}~
With the point set $\{0, 1, \dots, 180\}$ partitioned into
 residue classes modulo $156$ for $\{0, 1, \dots, 155\}$, and
 $\{156, 157, \dots, 180\}$,
 the design is generated from

\adfLgap {\adfBfont
$\{32, 117, 8, 74, 57\}$,
$\{125, 55, 33, 85, 113\}$,
$\{41, 103, 40, 175, 114\}$,\adfsplit
$\{22, 18, 162, 68, 155\}$,
$\{0, 5, 20, 53, 125\}$,
$\{0, 14, 79, 135, 158\}$,\adfsplit
$\{0, 27, 61, 115, 162\}$,
$\{0, 3, 9, 16, 127\}$,
$\{0, 18, 55, 130, 171\}$,\adfsplit
$\{0, 2, 59, 148, 177\}$,
$\{0, 39, 78, 117, 180\}$

}
\adfLgap \noindent by the mapping:
$x \mapsto x +  j \adfmod{156}$ for $x < 156$,
$x \mapsto (x - 156 + 2 j \adfmod{24}) + 156$ for $156 \le x < 180$,
$180 \mapsto 180$,
$0 \le j < 156$
 for the first ten blocks,
$0 \le j < 39$
 for the last block.
\ADFvfyParStart{(181, ((10, 156, ((156, 1), (24, 2), (1, 1))), (1, 39, ((156, 1), (24, 2), (1, 1)))), ((1, 156), (25, 1)))} 

\adfDgap
\noindent{\boldmath $ 1^{156} 45^{1} $}~
With the point set $\{0, 1, \dots, 200\}$ partitioned into
 residue classes modulo $156$ for $\{0, 1, \dots, 155\}$, and
 $\{156, 157, \dots, 200\}$,
 the design is generated from

\adfLgap {\adfBfont
$\{155, 128, 193, 57, 154\}$,
$\{177, 141, 110, 95, 97\}$,
$\{138, 20, 133, 166, 127\}$,\adfsplit
$\{194, 85, 104, 127, 34\}$,
$\{80, 84, 62, 183, 117\}$,
$\{0, 10, 99, 131, 162\}$,\adfsplit
$\{0, 12, 40, 96, 120\}$,
$\{0, 7, 41, 61, 175\}$,
$\{0, 17, 47, 111, 166\}$,\adfsplit
$\{0, 8, 29, 81, 185\}$,
$\{0, 3, 90, 106, 173\}$,
$\{0, 9, 77, 91, 167\}$,\adfsplit
$\{0, 39, 78, 117, 200\}$

}
\adfLgap \noindent by the mapping:
$x \mapsto x +  j \adfmod{156}$ for $x < 156$,
$x \mapsto (x - 156 + 3 j \adfmod{36}) + 156$ for $156 \le x < 192$,
$x \mapsto (x + 2 j \adfmod{8}) + 192$ for $192 \le x < 200$,
$200 \mapsto 200$,
$0 \le j < 156$
 for the first 12 blocks,
$0 \le j < 39$
 for the last block.
\ADFvfyParStart{(201, ((12, 156, ((156, 1), (36, 3), (8, 2), (1, 1))), (1, 39, ((156, 1), (36, 3), (8, 2), (1, 1)))), ((1, 156), (45, 1)))} 

\adfDgap
\noindent{\boldmath $ 1^{160} 33^{1} $}~
With the point set $\{0, 1, \dots, 192\}$ partitioned into
 residue classes modulo $160$ for $\{0, 1, \dots, 159\}$, and
 $\{160, 161, \dots, 192\}$,
 the design is generated from

\adfLgap {\adfBfont
$\{58, 188, 97, 53, 148\}$,
$\{73, 84, 97, 167, 42\}$,
$\{161, 48, 66, 39, 150\}$,\adfsplit
$\{17, 151, 5, 161, 51\}$,
$\{0, 1, 75, 79, 184\}$,
$\{0, 7, 61, 104, 172\}$,\adfsplit
$\{0, 8, 33, 140, 161\}$,
$\{0, 15, 62, 92, 178\}$,
$\{0, 17, 67, 89, 108\}$,\adfsplit
$\{0, 3, 32, 38, 48\}$,
$\{0, 2, 23, 59, 96\}$,
$\{0, 40, 80, 120, 192\}$

}
\adfLgap \noindent by the mapping:
$x \mapsto x +  j \adfmod{160}$ for $x < 160$,
$x \mapsto (x +  j \adfmod{32}) + 160$ for $160 \le x < 192$,
$192 \mapsto 192$,
$0 \le j < 160$
 for the first 11 blocks,
$0 \le j < 40$
 for the last block.
\ADFvfyParStart{(193, ((11, 160, ((160, 1), (32, 1), (1, 1))), (1, 40, ((160, 1), (32, 1), (1, 1)))), ((1, 160), (33, 1)))} 

\adfDgap
\noindent{\boldmath $ 1^{164} 21^{1} $}~
With the point set $\{0, 1, \dots, 184\}$ partitioned into
 residue classes modulo $164$ for $\{0, 1, \dots, 163\}$, and
 $\{164, 165, \dots, 184\}$,
 the design is generated from

\adfLgap {\adfBfont
$\{136, 87, 3, 31, 135\}$,
$\{81, 135, 120, 62, 102\}$,
$\{105, 7, 170, 62, 88\}$,\adfsplit
$\{0, 2, 71, 77, 168\}$,
$\{0, 9, 46, 99, 167\}$,
$\{0, 10, 23, 117, 164\}$,\adfsplit
$\{0, 22, 51, 101, 166\}$,
$\{0, 3, 7, 129, 159\}$,
$\{0, 14, 34, 78, 102\}$,\adfsplit
$\{0, 11, 27, 72, 139\}$,
$\{0, 41, 82, 123, 184\}$

}
\adfLgap \noindent by the mapping:
$x \mapsto x +  j \adfmod{164}$ for $x < 164$,
$x \mapsto (x - 164 + 5 j \adfmod{20}) + 164$ for $164 \le x < 184$,
$184 \mapsto 184$,
$0 \le j < 164$
 for the first ten blocks,
$0 \le j < 41$
 for the last block.
\ADFvfyParStart{(185, ((10, 164, ((164, 1), (20, 5), (1, 1))), (1, 41, ((164, 1), (20, 5), (1, 1)))), ((1, 164), (21, 1)))} 

\adfDgap
\noindent{\boldmath $ 1^{164} 41^{1} $}~
With the point set $\{0, 1, \dots, 204\}$ partitioned into
 residue classes modulo $164$ for $\{0, 1, \dots, 163\}$, and
 $\{164, 165, \dots, 204\}$,
 the design is generated from

\adfLgap {\adfBfont
$\{19, 76, 176, 22, 45\}$,
$\{78, 201, 25, 7, 88\}$,
$\{145, 106, 203, 55, 72\}$,\adfsplit
$\{39, 168, 9, 52, 94\}$,
$\{0, 1, 6, 95, 164\}$,
$\{0, 2, 27, 49, 170\}$,\adfsplit
$\{0, 7, 21, 126, 165\}$,
$\{0, 9, 67, 86, 169\}$,
$\{0, 11, 46, 61, 172\}$,\adfsplit
$\{0, 29, 62, 127, 167\}$,
$\{0, 4, 60, 76, 100\}$,
$\{0, 8, 20, 52, 136\}$,\adfsplit
$\{0, 41, 82, 123, 204\}$

}
\adfLgap \noindent by the mapping:
$x \mapsto x +  j \adfmod{164}$ for $x < 164$,
$x \mapsto (x - 164 + 10 j \adfmod{40}) + 164$ for $164 \le x < 204$,
$204 \mapsto 204$,
$0 \le j < 164$
 for the first 12 blocks,
$0 \le j < 41$
 for the last block.
\ADFvfyParStart{(205, ((12, 164, ((164, 1), (40, 10), (1, 1))), (1, 41, ((164, 1), (40, 10), (1, 1)))), ((1, 164), (41, 1)))} 

\adfDgap
\noindent{\boldmath $ 1^{168} 29^{1} $}~
With the point set $\{0, 1, \dots, 196\}$ partitioned into
 residue classes modulo $168$ for $\{0, 1, \dots, 167\}$, and
 $\{168, 169, \dots, 196\}$,
 the design is generated from

\adfLgap {\adfBfont
$\{190, 159, 102, 2, 25\}$,
$\{161, 67, 58, 191, 122\}$,
$\{32, 194, 39, 54, 68\}$,\adfsplit
$\{20, 65, 137, 25, 163\}$,
$\{0, 1, 4, 62, 148\}$,
$\{0, 10, 28, 88, 115\}$,\adfsplit
$\{0, 17, 66, 109, 184\}$,
$\{0, 8, 122, 135, 186\}$,
$\{0, 2, 99, 118, 179\}$,\adfsplit
$\{0, 6, 37, 130, 175\}$,
$\{0, 12, 47, 79, 95\}$,
$\{0, 42, 84, 126, 196\}$

}
\adfLgap \noindent by the mapping:
$x \mapsto x +  j \adfmod{168}$ for $x < 168$,
$x \mapsto (x +  j \adfmod{28}) + 168$ for $168 \le x < 196$,
$196 \mapsto 196$,
$0 \le j < 168$
 for the first 11 blocks,
$0 \le j < 42$
 for the last block.
\ADFvfyParStart{(197, ((11, 168, ((168, 1), (28, 1), (1, 1))), (1, 42, ((168, 1), (28, 1), (1, 1)))), ((1, 168), (29, 1)))} 

\adfDgap
\noindent{\boldmath $ 1^{168} 49^{1} $}~
With the point set $\{0, 1, \dots, 216\}$ partitioned into
 residue classes modulo $168$ for $\{0, 1, \dots, 167\}$, and
 $\{168, 169, \dots, 216\}$,
 the design is generated from

\adfLgap {\adfBfont
$\{190, 2, 20, 134, 132\}$,
$\{196, 9, 156, 85, 162\}$,
$\{206, 21, 96, 49, 37\}$,\adfsplit
$\{121, 134, 68, 209, 36\}$,
$\{39, 158, 84, 169, 157\}$,
$\{0, 3, 141, 146, 194\}$,\adfsplit
$\{0, 4, 68, 158, 176\}$,
$\{0, 19, 52, 86, 182\}$,
$\{0, 9, 40, 105, 129\}$,\adfsplit
$\{0, 8, 43, 69, 215\}$,
$\{0, 7, 51, 62, 211\}$,
$\{0, 17, 46, 127, 209\}$,\adfsplit
$\{0, 20, 57, 80, 171\}$,
$\{0, 42, 84, 126, 216\}$

}
\adfLgap \noindent by the mapping:
$x \mapsto x +  j \adfmod{168}$ for $x < 168$,
$x \mapsto (x - 168 + 2 j \adfmod{48}) + 168$ for $168 \le x < 216$,
$216 \mapsto 216$,
$0 \le j < 168$
 for the first 13 blocks,
$0 \le j < 42$
 for the last block.
\ADFvfyParStart{(217, ((13, 168, ((168, 1), (48, 2), (1, 1))), (1, 42, ((168, 1), (48, 2), (1, 1)))), ((1, 168), (49, 1)))} 

\adfDgap
\noindent{\boldmath $ 1^{172} 17^{1} $}~
With the point set $\{0, 1, \dots, 188\}$ partitioned into
 residue classes modulo $172$ for $\{0, 1, \dots, 171\}$, and
 $\{172, 173, \dots, 188\}$,
 the design is generated from

\adfLgap {\adfBfont
$\{91, 84, 164, 53, 31\}$,
$\{113, 19, 59, 8, 56\}$,
$\{136, 163, 111, 144, 8\}$,\adfsplit
$\{0, 5, 47, 106, 173\}$,
$\{0, 9, 35, 98, 174\}$,
$\{0, 15, 49, 90, 172\}$,\adfsplit
$\{0, 21, 79, 102, 175\}$,
$\{0, 1, 46, 96, 108\}$,
$\{0, 13, 29, 85, 117\}$,\adfsplit
$\{0, 2, 6, 20, 30\}$,
$\{0, 43, 86, 129, 188\}$

}
\adfLgap \noindent by the mapping:
$x \mapsto x +  j \adfmod{172}$ for $x < 172$,
$x \mapsto (x - 172 + 4 j \adfmod{16}) + 172$ for $172 \le x < 188$,
$188 \mapsto 188$,
$0 \le j < 172$
 for the first ten blocks,
$0 \le j < 43$
 for the last block.
\ADFvfyParStart{(189, ((10, 172, ((172, 1), (16, 4), (1, 1))), (1, 43, ((172, 1), (16, 4), (1, 1)))), ((1, 172), (17, 1)))} 

\adfDgap
\noindent{\boldmath $ 1^{176} 25^{1} $}~
With the point set $\{0, 1, \dots, 200\}$ partitioned into
 residue classes modulo $176$ for $\{0, 1, \dots, 175\}$, and
 $\{176, 177, \dots, 200\}$,
 the design is generated from

\adfLgap {\adfBfont
$\{0, 124, 140, 112, 161\}$,
$\{109, 53, 180, 51, 106\}$,
$\{173, 103, 10, 125, 84\}$,\adfsplit
$\{137, 45, 0, 171, 99\}$,
$\{0, 1, 30, 97, 101\}$,
$\{0, 7, 32, 117, 143\}$,\adfsplit
$\{0, 9, 23, 69, 193\}$,
$\{0, 11, 31, 73, 198\}$,
$\{0, 6, 63, 158, 178\}$,\adfsplit
$\{0, 10, 108, 159, 181\}$,
$\{0, 8, 43, 90, 184\}$,
$\{0, 44, 88, 132, 200\}$

}
\adfLgap \noindent by the mapping:
$x \mapsto x +  j \adfmod{176}$ for $x < 176$,
$x \mapsto (x +  j \adfmod{16}) + 176$ for $176 \le x < 192$,
$x \mapsto (x +  j \adfmod{8}) + 192$ for $192 \le x < 200$,
$200 \mapsto 200$,
$0 \le j < 176$
 for the first 11 blocks,
$0 \le j < 44$
 for the last block.
\ADFvfyParStart{(201, ((11, 176, ((176, 1), (16, 1), (8, 1), (1, 1))), (1, 44, ((176, 1), (16, 1), (8, 1), (1, 1)))), ((1, 176), (25, 1)))} 

\adfDgap
\noindent{\boldmath $ 1^{176} 45^{1} $}~
With the point set $\{0, 1, \dots, 220\}$ partitioned into
 residue classes modulo $176$ for $\{0, 1, \dots, 175\}$, and
 $\{176, 177, \dots, 220\}$,
 the design is generated from

\adfLgap {\adfBfont
$\{187, 128, 49, 139, 167\}$,
$\{111, 115, 92, 38, 193\}$,
$\{9, 186, 175, 161, 51\}$,\adfsplit
$\{91, 152, 112, 195, 158\}$,
$\{209, 121, 151, 143, 148\}$,
$\{0, 1, 143, 145, 208\}$,\adfsplit
$\{0, 18, 69, 141, 183\}$,
$\{0, 15, 70, 117, 188\}$,
$\{0, 9, 71, 96, 194\}$,\adfsplit
$\{0, 7, 48, 91, 108\}$,
$\{0, 20, 57, 113, 181\}$,
$\{0, 12, 112, 138, 184\}$,\adfsplit
$\{0, 13, 49, 78, 94\}$,
$\{0, 44, 88, 132, 220\}$

}
\adfLgap \noindent by the mapping:
$x \mapsto x +  j \adfmod{176}$ for $x < 176$,
$x \mapsto (x +  j \adfmod{44}) + 176$ for $176 \le x < 220$,
$220 \mapsto 220$,
$0 \le j < 176$
 for the first 13 blocks,
$0 \le j < 44$
 for the last block.
\ADFvfyParStart{(221, ((13, 176, ((176, 1), (44, 1), (1, 1))), (1, 44, ((176, 1), (44, 1), (1, 1)))), ((1, 176), (45, 1)))} 

\adfDgap
\noindent{\boldmath $ 1^{184} 21^{1} $}~
With the point set $\{0, 1, \dots, 204\}$ partitioned into
 residue classes modulo $184$ for $\{0, 1, \dots, 183\}$, and
 $\{184, 185, \dots, 204\}$,
 the design is generated from

\adfLgap {\adfBfont
$\{36, 2, 116, 109, 165\}$,
$\{24, 102, 124, 125, 28\}$,
$\{68, 197, 131, 0, 157\}$,\adfsplit
$\{71, 52, 119, 157, 88\}$,
$\{0, 3, 15, 50, 199\}$,
$\{0, 13, 42, 103, 200\}$,\adfsplit
$\{0, 2, 18, 120, 145\}$,
$\{0, 6, 20, 132, 160\}$,
$\{0, 5, 37, 45, 130\}$,\adfsplit
$\{0, 9, 60, 71, 186\}$,
$\{0, 10, 43, 119, 196\}$,
$\{0, 46, 92, 138, 204\}$

}
\adfLgap \noindent by the mapping:
$x \mapsto x +  j \adfmod{184}$ for $x < 184$,
$x \mapsto (x - 184 + 2 j \adfmod{16}) + 184$ for $184 \le x < 200$,
$x \mapsto (x +  j \adfmod{4}) + 200$ for $200 \le x < 204$,
$204 \mapsto 204$,
$0 \le j < 184$
 for the first 11 blocks,
$0 \le j < 46$
 for the last block.
\ADFvfyParStart{(205, ((11, 184, ((184, 1), (16, 2), (4, 1), (1, 1))), (1, 46, ((184, 1), (16, 2), (4, 1), (1, 1)))), ((1, 184), (21, 1)))} 

\adfDgap
\noindent{\boldmath $ 1^{184} 41^{1} $}~
With the point set $\{0, 1, \dots, 224\}$ partitioned into
 residue classes modulo $184$ for $\{0, 1, \dots, 183\}$, and
 $\{184, 185, \dots, 224\}$,
 the design is generated from

\adfLgap {\adfBfont
$\{200, 80, 93, 174, 155\}$,
$\{71, 185, 177, 148, 14\}$,
$\{70, 96, 61, 217, 58\}$,\adfsplit
$\{20, 121, 208, 24, 85\}$,
$\{219, 136, 1, 26, 85\}$,
$\{214, 29, 115, 182, 122\}$,\adfsplit
$\{0, 20, 63, 131, 193\}$,
$\{0, 14, 41, 140, 202\}$,
$\{0, 1, 16, 96, 129\}$,\adfsplit
$\{0, 10, 76, 115, 206\}$,
$\{0, 6, 17, 54, 156\}$,
$\{0, 2, 24, 32, 144\}$,\adfsplit
$\{0, 5, 52, 166, 201\}$,
$\{0, 46, 92, 138, 224\}$

}
\adfLgap \noindent by the mapping:
$x \mapsto x +  j \adfmod{184}$ for $x < 184$,
$x \mapsto (x - 184 + 5 j \adfmod{40}) + 184$ for $184 \le x < 224$,
$224 \mapsto 224$,
$0 \le j < 184$
 for the first 13 blocks,
$0 \le j < 46$
 for the last block.
\ADFvfyParStart{(225, ((13, 184, ((184, 1), (40, 5), (1, 1))), (1, 46, ((184, 1), (40, 5), (1, 1)))), ((1, 184), (41, 1)))} 

\adfDgap
\noindent{\boldmath $ 1^{188} 29^{1} $}~
With the point set $\{0, 1, \dots, 216\}$ partitioned into
 residue classes modulo $188$ for $\{0, 1, \dots, 187\}$, and
 $\{188, 189, \dots, 216\}$,
 the design is generated from

\adfLgap {\adfBfont
$\{209, 108, 5, 82, 71\}$,
$\{109, 196, 182, 148, 119\}$,
$\{124, 79, 170, 110, 66\}$,\adfsplit
$\{204, 56, 107, 106, 125\}$,
$\{0, 2, 5, 27, 191\}$,
$\{0, 6, 21, 95, 194\}$,\adfsplit
$\{0, 17, 59, 82, 193\}$,
$\{0, 30, 79, 117, 192\}$,
$\{0, 4, 24, 102, 156\}$,\adfsplit
$\{0, 8, 61, 70, 113\}$,
$\{0, 12, 67, 108, 124\}$,
$\{0, 7, 35, 107, 155\}$,\adfsplit
$\{0, 47, 94, 141, 216\}$

}
\adfLgap \noindent by the mapping:
$x \mapsto x +  j \adfmod{188}$ for $x < 188$,
$x \mapsto (x - 188 + 7 j \adfmod{28}) + 188$ for $188 \le x < 216$,
$216 \mapsto 216$,
$0 \le j < 188$
 for the first 12 blocks,
$0 \le j < 47$
 for the last block.
\ADFvfyParStart{(217, ((12, 188, ((188, 1), (28, 7), (1, 1))), (1, 47, ((188, 1), (28, 7), (1, 1)))), ((1, 188), (29, 1)))} 

\adfDgap
\noindent{\boldmath $ 1^{192} 17^{1} $}~
With the point set $\{0, 1, \dots, 208\}$ partitioned into
 residue classes modulo $192$ for $\{0, 1, \dots, 191\}$, and
 $\{192, 193, \dots, 208\}$,
 the design is generated from

\adfLgap {\adfBfont
$\{43, 188, 94, 101, 134\}$,
$\{189, 115, 101, 119, 39\}$,
$\{116, 135, 188, 166, 99\}$,\adfsplit
$\{129, 88, 134, 190, 202\}$,
$\{0, 1, 3, 12, 129\}$,
$\{0, 28, 60, 99, 137\}$,\adfsplit
$\{0, 16, 43, 111, 135\}$,
$\{0, 15, 52, 162, 202\}$,
$\{0, 8, 21, 86, 197\}$,\adfsplit
$\{0, 6, 26, 85, 169\}$,
$\{0, 10, 35, 158, 193\}$,
$\{0, 48, 96, 144, 208\}$

}
\adfLgap \noindent by the mapping:
$x \mapsto x +  j \adfmod{192}$ for $x < 192$,
$x \mapsto (x +  j \adfmod{16}) + 192$ for $192 \le x < 208$,
$208 \mapsto 208$,
$0 \le j < 192$
 for the first 11 blocks,
$0 \le j < 48$
 for the last block.
\ADFvfyParStart{(209, ((11, 192, ((192, 1), (16, 1), (1, 1))), (1, 48, ((192, 1), (16, 1), (1, 1)))), ((1, 192), (17, 1)))} 

\adfDgap
\noindent{\boldmath $ 1^{192} 37^{1} $}~
With the point set $\{0, 1, \dots, 228\}$ partitioned into
 residue classes modulo $192$ for $\{0, 1, \dots, 191\}$, and
 $\{192, 193, \dots, 228\}$,
 the design is generated from

\adfLgap {\adfBfont
$\{128, 120, 2, 201, 90\}$,
$\{182, 67, 222, 58, 72\}$,
$\{152, 61, 7, 156, 85\}$,\adfsplit
$\{41, 62, 206, 168, 74\}$,
$\{72, 83, 44, 128, 28\}$,
$\{219, 29, 54, 2, 3\}$,\adfsplit
$\{0, 35, 105, 151, 218\}$,
$\{0, 15, 32, 72, 90\}$,
$\{0, 2, 36, 85, 200\}$,\adfsplit
$\{0, 3, 22, 64, 133\}$,
$\{0, 7, 80, 179, 202\}$,
$\{0, 10, 63, 113, 196\}$,\adfsplit
$\{0, 6, 29, 161, 214\}$,
$\{0, 48, 96, 144, 228\}$

}
\adfLgap \noindent by the mapping:
$x \mapsto x +  j \adfmod{192}$ for $x < 192$,
$x \mapsto (x +  j \adfmod{24}) + 192$ for $192 \le x < 216$,
$x \mapsto (x +  j \adfmod{12}) + 216$ for $216 \le x < 228$,
$228 \mapsto 228$,
$0 \le j < 192$
 for the first 13 blocks,
$0 \le j < 48$
 for the last block.
\ADFvfyParStart{(229, ((13, 192, ((192, 1), (24, 1), (12, 1), (1, 1))), (1, 48, ((192, 1), (24, 1), (12, 1), (1, 1)))), ((1, 192), (37, 1)))} 

\adfDgap
\noindent{\boldmath $ 1^{192} 57^{1} $}~
With the point set $\{0, 1, \dots, 248\}$ partitioned into
 residue classes modulo $192$ for $\{0, 1, \dots, 191\}$, and
 $\{192, 193, \dots, 248\}$,
 the design is generated from

\adfLgap {\adfBfont
$\{60, 186, 233, 46, 43\}$,
$\{165, 107, 65, 211, 50\}$,
$\{93, 113, 193, 120, 77\}$,\adfsplit
$\{27, 237, 101, 83, 182\}$,
$\{57, 238, 78, 82, 90\}$,
$\{174, 246, 61, 143, 114\}$,\adfsplit
$\{80, 124, 57, 200, 119\}$,
$\{0, 26, 61, 146, 227\}$,
$\{0, 1, 11, 76, 246\}$,\adfsplit
$\{0, 6, 38, 128, 147\}$,
$\{0, 9, 68, 98, 199\}$,
$\{0, 40, 95, 145, 202\}$,\adfsplit
$\{0, 24, 78, 112, 205\}$,
$\{0, 22, 63, 91, 214\}$,
$\{0, 2, 73, 86, 220\}$,\adfsplit
$\{0, 48, 96, 144, 248\}$

}
\adfLgap \noindent by the mapping:
$x \mapsto x +  j \adfmod{192}$ for $x < 192$,
$x \mapsto (x +  j \adfmod{32}) + 192$ for $192 \le x < 224$,
$x \mapsto (x +  j \adfmod{16}) + 224$ for $224 \le x < 240$,
$x \mapsto (x +  j \adfmod{8}) + 240$ for $240 \le x < 248$,
$248 \mapsto 248$,
$0 \le j < 192$
 for the first 15 blocks,
$0 \le j < 48$
 for the last block.
\ADFvfyParStart{(249, ((15, 192, ((192, 1), (32, 1), (16, 1), (8, 1), (1, 1))), (1, 48, ((192, 1), (32, 1), (16, 1), (8, 1), (1, 1)))), ((1, 192), (57, 1)))} 

\adfDgap
\noindent{\boldmath $ 1^{196} 25^{1} $}~
With the point set $\{0, 1, \dots, 220\}$ partitioned into
 residue classes modulo $196$ for $\{0, 1, \dots, 195\}$, and
 $\{196, 197, \dots, 220\}$,
 the design is generated from

\adfLgap {\adfBfont
$\{53, 122, 200, 96, 131\}$,
$\{147, 142, 120, 179, 139\}$,
$\{202, 25, 102, 32, 183\}$,\adfsplit
$\{130, 204, 147, 13, 76\}$,
$\{16, 47, 144, 57, 168\}$,
$\{0, 1, 83, 106, 201\}$,\adfsplit
$\{0, 11, 53, 146, 199\}$,
$\{0, 21, 55, 94, 197\}$,
$\{0, 6, 57, 72, 86\}$,\adfsplit
$\{0, 12, 25, 58, 132\}$,
$\{0, 2, 20, 67, 168\}$,
$\{0, 4, 60, 96, 112\}$,\adfsplit
$\{0, 49, 98, 147, 220\}$

}
\adfLgap \noindent by the mapping:
$x \mapsto x +  j \adfmod{196}$ for $x < 196$,
$x \mapsto (x - 196 + 6 j \adfmod{24}) + 196$ for $196 \le x < 220$,
$220 \mapsto 220$,
$0 \le j < 196$
 for the first 12 blocks,
$0 \le j < 49$
 for the last block.
\ADFvfyParStart{(221, ((12, 196, ((196, 1), (24, 6), (1, 1))), (1, 49, ((196, 1), (24, 6), (1, 1)))), ((1, 196), (25, 1)))} 

\adfDgap
\noindent{\boldmath $ 1^{196} 45^{1} $}~
With the point set $\{0, 1, \dots, 240\}$ partitioned into
 residue classes modulo $196$ for $\{0, 1, \dots, 195\}$, and
 $\{196, 197, \dots, 240\}$,
 the design is generated from

\adfLgap {\adfBfont
$\{76, 239, 147, 146, 161\}$,
$\{191, 84, 112, 138, 149\}$,
$\{7, 233, 12, 174, 165\}$,\adfsplit
$\{224, 124, 145, 94, 151\}$,
$\{182, 230, 157, 27, 60\}$,
$\{175, 116, 85, 203, 81\}$,\adfsplit
$\{161, 219, 57, 12, 89\}$,
$\{0, 2, 20, 82, 138\}$,
$\{0, 7, 24, 68, 91\}$,\adfsplit
$\{0, 3, 22, 123, 203\}$,
$\{0, 8, 108, 121, 221\}$,
$\{0, 16, 55, 103, 210\}$,\adfsplit
$\{0, 12, 64, 127, 216\}$,
$\{0, 10, 46, 156, 205\}$,
$\{0, 49, 98, 147, 240\}$

}
\adfLgap \noindent by the mapping:
$x \mapsto x +  j \adfmod{196}$ for $x < 196$,
$x \mapsto (x +  j \adfmod{28}) + 196$ for $196 \le x < 224$,
$x \mapsto (x + 4 j \adfmod{16}) + 224$ for $224 \le x < 240$,
$240 \mapsto 240$,
$0 \le j < 196$
 for the first 14 blocks,
$0 \le j < 49$
 for the last block.
\ADFvfyParStart{(241, ((14, 196, ((196, 1), (28, 1), (16, 4), (1, 1))), (1, 49, ((196, 1), (28, 1), (16, 4), (1, 1)))), ((1, 196), (45, 1)))} 

\adfDgap
\noindent{\boldmath $ 2^{32} 14^{1} $}~
With the point set $\{0, 1, \dots, 77\}$ partitioned into
 residue classes modulo $32$ for $\{0, 1, \dots, 63\}$, and
 $\{64, 65, \dots, 77\}$,
 the design is generated from

\adfLgap {\adfBfont
$\{0, 11, 20, 53, 65\}$,
$\{0, 1, 7, 62, 76\}$,
$\{0, 4, 30, 43, 69\}$,\adfsplit
$\{0, 23, 59, 63, 64\}$,
$\{0, 12, 49, 57, 66\}$,
$\{0, 6, 16, 24, 41\}$,\adfsplit
$\{0, 5, 21, 31, 51\}$,
$\{0, 14, 42, 61, 73\}$,
$\{0, 15, 27, 29, 75\}$

}
\adfLgap \noindent by the mapping:
$x \mapsto x + 2 j \adfmod{64}$ for $x < 64$,
$x \mapsto (x +  j \adfmod{8}) + 64$ for $64 \le x < 72$,
$x \mapsto (x +  j \adfmod{4}) + 72$ for $72 \le x < 76$,
$x \mapsto (x +  j \adfmod{2}) + 76$ for $x \ge 76$,
$0 \le j < 32$.
\ADFvfyParStart{(78, ((9, 32, ((64, 2), (8, 1), (4, 1), (2, 1)))), ((2, 32), (14, 1)))} 

\adfDgap
\noindent{\boldmath $ 2^{40} 6^{1} $}~
With the point set $\{0, 1, \dots, 85\}$ partitioned into
 residue classes modulo $40$ for $\{0, 1, \dots, 79\}$, and
 $\{80, 81, \dots, 85\}$,
 the design is generated from

\adfLgap {\adfBfont
$\{51, 31, 63, 50, 27\}$,
$\{4, 7, 69, 34, 71\}$,
$\{0, 2, 7, 73, 80\}$,\adfsplit
$\{0, 6, 49, 75, 82\}$,
$\{0, 20, 41, 51, 79\}$,
$\{0, 11, 45, 62, 81\}$,\adfsplit
$\{0, 8, 17, 23, 64\}$,
$\{0, 19, 27, 52, 77\}$,
$\{0, 4, 14, 46, 58\}$

}
\adfLgap \noindent by the mapping:
$x \mapsto x + 2 j \adfmod{80}$ for $x < 80$,
$x \mapsto (x - 80 + 3 j \adfmod{6}) + 80$ for $x \ge 80$,
$0 \le j < 40$.
\ADFvfyParStart{(86, ((9, 40, ((80, 2), (6, 3)))), ((2, 40), (6, 1)))} 

\adfDgap
\noindent{\boldmath $ 2^{48} 18^{1} $}~
With the point set $\{0, 1, \dots, 113\}$ partitioned into
 residue classes modulo $48$ for $\{0, 1, \dots, 95\}$, and
 $\{96, 97, \dots, 113\}$,
 the design is generated from

\adfLgap {\adfBfont
$\{111, 83, 60, 53, 38\}$,
$\{15, 79, 50, 48, 84\}$,
$\{80, 3, 98, 19, 57\}$,\adfsplit
$\{0, 5, 32, 58, 79\}$,
$\{0, 1, 3, 9, 108\}$,
$\{0, 6, 39, 82, 112\}$,\adfsplit
$\{0, 8, 65, 80, 93\}$,
$\{0, 41, 51, 75, 87\}$,
$\{0, 7, 66, 83, 101\}$,\adfsplit
$\{1, 5, 57, 83, 107\}$,
$\{0, 11, 12, 52, 97\}$,
$\{0, 10, 59, 78, 105\}$,\adfsplit
$\{0, 4, 46, 71, 98\}$

}
\adfLgap \noindent by the mapping:
$x \mapsto x + 2 j \adfmod{96}$ for $x < 96$,
$x \mapsto (x +  j \adfmod{12}) + 96$ for $96 \le x < 108$,
$x \mapsto (x +  j \adfmod{6}) + 108$ for $x \ge 108$,
$0 \le j < 48$.
\ADFvfyParStart{(114, ((13, 48, ((96, 2), (12, 1), (6, 1)))), ((2, 48), (18, 1)))} 

\adfDgap
\noindent{\boldmath $ 2^{52} 14^{1} $}~
With the point set $\{0, 1, \dots, 117\}$ partitioned into
 residue classes modulo $52$ for $\{0, 1, \dots, 103\}$, and
 $\{104, 105, \dots, 117\}$,
 the design is generated from

\adfLgap {\adfBfont
$\{55, 54, 97, 0, 15\}$,
$\{26, 87, 56, 109, 29\}$,
$\{112, 30, 83, 65, 20\}$,\adfsplit
$\{16, 83, 11, 27, 95\}$,
$\{0, 2, 9, 103, 108\}$,
$\{0, 5, 39, 86, 104\}$,\adfsplit
$\{0, 13, 17, 46, 93\}$,
$\{0, 25, 51, 70, 106\}$,
$\{0, 21, 81, 87, 95\}$,\adfsplit
$\{0, 27, 29, 42, 107\}$,
$\{0, 14, 33, 83, 110\}$,
$\{0, 4, 41, 72, 96\}$,\adfsplit
$\{0, 6, 26, 66, 82\}$

}
\adfLgap \noindent by the mapping:
$x \mapsto x + 2 j \adfmod{104}$ for $x < 104$,
$x \mapsto (x - 104 + 7 j \adfmod{14}) + 104$ for $x \ge 104$,
$0 \le j < 52$.
\ADFvfyParStart{(118, ((13, 52, ((104, 2), (14, 7)))), ((2, 52), (14, 1)))} 

\adfDgap
\noindent{\boldmath $ 2^{56} 10^{1} $}~
With the point set $\{0, 1, \dots, 121\}$ partitioned into
 residue classes modulo $56$ for $\{0, 1, \dots, 111\}$, and
 $\{112, 113, \dots, 121\}$,
 the design is generated from

\adfLgap {\adfBfont
$\{120, 105, 88, 79, 102\}$,
$\{20, 107, 10, 52, 16\}$,
$\{20, 78, 116, 105, 45\}$,\adfsplit
$\{67, 72, 116, 26, 31\}$,
$\{0, 1, 2, 92, 116\}$,
$\{0, 15, 37, 49, 113\}$,\adfsplit
$\{0, 8, 26, 61, 77\}$,
$\{0, 9, 29, 48, 82\}$,
$\{0, 7, 44, 72, 96\}$,\adfsplit
$\{0, 12, 43, 45, 74\}$,
$\{0, 13, 19, 63, 101\}$,
$\{0, 11, 57, 65, 105\}$,\adfsplit
$\{0, 67, 95, 99, 109\}$

}
\adfLgap \noindent by the mapping:
$x \mapsto x + 2 j \adfmod{112}$ for $x < 112$,
$x \mapsto (x +  j \adfmod{8}) + 112$ for $112 \le x < 120$,
$x \mapsto (x +  j \adfmod{2}) + 120$ for $x \ge 120$,
$0 \le j < 56$.
\ADFvfyParStart{(122, ((13, 56, ((112, 2), (8, 1), (2, 1)))), ((2, 56), (10, 1)))} 

\adfDgap
\noindent{\boldmath $ 2^{60} 6^{1} $}~
With the point set $\{0, 1, \dots, 125\}$ partitioned into
 residue classes modulo $60$ for $\{0, 1, \dots, 119\}$, and
 $\{120, 121, \dots, 125\}$,
 the design is generated from

\adfLgap {\adfBfont
$\{15, 81, 4, 83, 10\}$,
$\{39, 123, 56, 22, 109\}$,
$\{80, 73, 102, 77, 8\}$,\adfsplit
$\{20, 102, 124, 98, 105\}$,
$\{62, 54, 52, 34, 43\}$,
$\{0, 15, 55, 93, 122\}$,\adfsplit
$\{0, 12, 36, 63, 68\}$,
$\{0, 14, 30, 59, 80\}$,
$\{0, 1, 21, 37, 74\}$,\adfsplit
$\{0, 25, 33, 39, 57\}$,
$\{0, 23, 49, 62, 97\}$,
$\{0, 19, 31, 41, 75\}$,\adfsplit
$\{0, 13, 43, 44, 105\}$

}
\adfLgap \noindent by the mapping:
$x \mapsto x + 2 j \adfmod{120}$ for $x < 120$,
$x \mapsto (x +  j \adfmod{6}) + 120$ for $x \ge 120$,
$0 \le j < 60$.
\ADFvfyParStart{(126, ((13, 60, ((120, 2), (6, 1)))), ((2, 60), (6, 1)))} 

\adfDgap
\noindent{\boldmath $ 3^{20} 11^{1} $}~
With the point set $\{0, 1, \dots, 70\}$ partitioned into
 residue classes modulo $20$ for $\{0, 1, \dots, 59\}$, and
 $\{60, 61, \dots, 70\}$,
 the design is generated from

\adfLgap {\adfBfont
$\{0, 3, 11, 38, 61\}$,
$\{0, 1, 5, 7, 44\}$,
$\{0, 13, 29, 55, 69\}$,\adfsplit
$\{0, 2, 10, 37, 51\}$,
$\{0, 9, 31, 59, 63\}$,
$\{0, 4, 43, 46, 68\}$,\adfsplit
$\{0, 6, 34, 53, 67\}$,
$\{0, 15, 30, 45, 70\}$,
$\{0, 12, 24, 36, 48\}$,\adfsplit
$\{1, 13, 25, 37, 49\}$

}
\adfLgap \noindent by the mapping:
$x \mapsto x + 2 j \adfmod{60}$ for $x < 60$,
$x \mapsto (x +  j \adfmod{10}) + 60$ for $60 \le x < 70$,
$70 \mapsto 70$,
$0 \le j < 30$
 for the first seven blocks,
$0 \le j < 15$
 for the next block,
$0 \le j < 6$
 for the last two blocks.
\ADFvfyParStart{(71, ((7, 30, ((60, 2), (10, 1), (1, 1))), (1, 15, ((60, 2), (10, 1), (1, 1))), (2, 6, ((60, 2), (10, 1), (1, 1)))), ((3, 20), (11, 1)))} 

\adfDgap
\noindent{\boldmath $ 3^{28} 7^{1} $}~
With the point set $\{0, 1, \dots, 90\}$ partitioned into
 residue classes modulo $28$ for $\{0, 1, \dots, 83\}$, and
 $\{84, 85, \dots, 90\}$,
 the design is generated from

\adfLgap {\adfBfont
$\{0, 60, 9, 82, 50\}$,
$\{9, 15, 19, 49, 58\}$,
$\{0, 1, 8, 59, 78\}$,\adfsplit
$\{0, 19, 57, 71, 79\}$,
$\{0, 16, 53, 55, 90\}$,
$\{0, 3, 38, 64, 69\}$,\adfsplit
$\{0, 12, 27, 30, 86\}$,
$\{0, 47, 67, 83, 88\}$,
$\{0, 4, 17, 29, 40\}$,\adfsplit
$\{0, 21, 42, 63, 84\}$

}
\adfLgap \noindent by the mapping:
$x \mapsto x + 2 j \adfmod{84}$ for $x < 84$,
$x \mapsto (x +  j \adfmod{7}) + 84$ for $x \ge 84$,
$0 \le j < 42$
 for the first nine blocks,
$0 \le j < 21$
 for the last block.
\ADFvfyParStart{(91, ((9, 42, ((84, 2), (7, 1))), (1, 21, ((84, 2), (7, 1)))), ((3, 28), (7, 1)))} 

\adfDgap
\noindent{\boldmath $ 3^{32} 11^{1} $}~
With the point set $\{0, 1, \dots, 106\}$ partitioned into
 residue classes modulo $32$ for $\{0, 1, \dots, 95\}$, and
 $\{96, 97, \dots, 106\}$,
 the design is generated from

\adfLgap {\adfBfont
$\{99, 75, 72, 57, 14\}$,
$\{8, 101, 2, 55, 21\}$,
$\{23, 86, 91, 0, 35\}$,\adfsplit
$\{0, 1, 2, 27, 98\}$,
$\{0, 14, 31, 77, 97\}$,
$\{0, 18, 87, 89, 100\}$,\adfsplit
$\{0, 37, 41, 57, 79\}$,
$\{0, 7, 59, 67, 73\}$,
$\{0, 4, 15, 26, 60\}$,\adfsplit
$\{0, 8, 20, 29, 50\}$,
$\{0, 16, 55, 65, 68\}$,
$\{0, 24, 48, 72, 106\}$,\adfsplit
$\{1, 25, 49, 73, 106\}$

}
\adfLgap \noindent by the mapping:
$x \mapsto x + 2 j \adfmod{96}$ for $x < 96$,
$x \mapsto (x - 96 + 5 j \adfmod{10}) + 96$ for $96 \le x < 106$,
$106 \mapsto 106$,
$0 \le j < 48$
 for the first 11 blocks,
$0 \le j < 12$
 for the last two blocks.
\ADFvfyParStart{(107, ((11, 48, ((96, 2), (10, 5), (1, 1))), (2, 12, ((96, 2), (10, 5), (1, 1)))), ((3, 32), (11, 1)))} 

\adfDgap
\noindent{\boldmath $ 3^{36} 15^{1} $}~
With the point set $\{0, 1, \dots, 122\}$ partitioned into
 residue classes modulo $36$ for $\{0, 1, \dots, 107\}$, and
 $\{108, 109, \dots, 122\}$,
 the design is generated from

\adfLgap {\adfBfont
$\{104, 7, 20, 0, 75\}$,
$\{13, 24, 37, 14, 7\}$,
$\{100, 78, 116, 32, 81\}$,\adfsplit
$\{54, 42, 70, 107, 85\}$,
$\{0, 1, 6, 39, 120\}$,
$\{0, 8, 38, 56, 85\}$,\adfsplit
$\{0, 14, 59, 87, 111\}$,
$\{0, 2, 63, 66, 119\}$,
$\{0, 5, 21, 25, 115\}$,\adfsplit
$\{0, 9, 57, 71, 83\}$,
$\{0, 26, 67, 76, 118\}$,
$\{0, 17, 35, 74, 93\}$,\adfsplit
$\{1, 3, 11, 67, 116\}$,
$\{0, 27, 54, 81, 122\}$

}
\adfLgap \noindent by the mapping:
$x \mapsto x + 2 j \adfmod{108}$ for $x < 108$,
$x \mapsto (x + 2 j \adfmod{12}) + 108$ for $108 \le x < 120$,
$x \mapsto (x +  j \adfmod{2}) + 120$ for $120 \le x < 122$,
$122 \mapsto 122$,
$0 \le j < 54$
 for the first 13 blocks,
$0 \le j < 27$
 for the last block.
\ADFvfyParStart{(123, ((13, 54, ((108, 2), (12, 2), (2, 1), (1, 1))), (1, 27, ((108, 2), (12, 2), (2, 1), (1, 1)))), ((3, 36), (15, 1)))} 

\adfDgap
\noindent{\boldmath $ 5^{24} 25^{1} $}~
With the point set $\{0, 1, \dots, 144\}$ partitioned into
 residue classes modulo $24$ for $\{0, 1, \dots, 119\}$, and
 $\{120, 121, \dots, 144\}$,
 the design is generated from

\adfLgap {\adfBfont
$\{2, 54, 94, 116, 13\}$,
$\{78, 105, 69, 112, 134\}$,
$\{109, 31, 121, 5, 96\}$,\adfsplit
$\{0, 8, 61, 71, 123\}$,
$\{0, 2, 21, 85, 97\}$,
$\{0, 5, 20, 38, 131\}$,\adfsplit
$\{0, 1, 32, 46, 120\}$,
$\{0, 3, 69, 73, 130\}$,
$\{0, 30, 60, 90, 144\}$

}
\adfLgap \noindent by the mapping:
$x \mapsto x +  j \adfmod{120}$ for $x < 120$,
$x \mapsto (x +  j \adfmod{24}) + 120$ for $120 \le x < 144$,
$144 \mapsto 144$,
$0 \le j < 120$
 for the first eight blocks,
$0 \le j < 30$
 for the last block.
\ADFvfyParStart{(145, ((8, 120, ((120, 1), (24, 1), (1, 1))), (1, 30, ((120, 1), (24, 1), (1, 1)))), ((5, 24), (25, 1)))} 

\adfDgap
\noindent{\boldmath $ 5^{28} 25^{1} $}~
With the point set $\{0, 1, \dots, 164\}$ partitioned into
 residue classes modulo $28$ for $\{0, 1, \dots, 139\}$, and
 $\{140, 141, \dots, 164\}$,
 the design is generated from

\adfLgap {\adfBfont
$\{41, 28, 158, 75, 14\}$,
$\{107, 154, 33, 108, 65\}$,
$\{50, 0, 145, 137, 46\}$,\adfsplit
$\{0, 2, 17, 39, 160\}$,
$\{0, 12, 38, 83, 154\}$,
$\{0, 6, 25, 58, 158\}$,\adfsplit
$\{0, 5, 16, 36, 116\}$,
$\{0, 7, 30, 51, 92\}$,
$\{0, 8, 18, 72, 81\}$,\adfsplit
$\{0, 35, 70, 105, 164\}$

}
\adfLgap \noindent by the mapping:
$x \mapsto x +  j \adfmod{140}$ for $x < 140$,
$x \mapsto (x +  j \adfmod{20}) + 140$ for $140 \le x < 160$,
$x \mapsto (x +  j \adfmod{4}) + 160$ for $160 \le x < 164$,
$164 \mapsto 164$,
$0 \le j < 140$
 for the first nine blocks,
$0 \le j < 35$
 for the last block.
\ADFvfyParStart{(165, ((9, 140, ((140, 1), (20, 1), (4, 1), (1, 1))), (1, 35, ((140, 1), (20, 1), (4, 1), (1, 1)))), ((5, 28), (25, 1)))} 

\adfDgap
\noindent{\boldmath $ 5^{32} 25^{1} $}~
With the point set $\{0, 1, \dots, 184\}$ partitioned into
 residue classes modulo $32$ for $\{0, 1, \dots, 159\}$, and
 $\{160, 161, \dots, 184\}$,
 the design is generated from

\adfLgap {\adfBfont
$\{35, 27, 136, 171, 2\}$,
$\{75, 90, 104, 22, 176\}$,
$\{69, 129, 50, 123, 107\}$,\adfsplit
$\{0, 1, 3, 10, 180\}$,
$\{0, 11, 31, 86, 130\}$,
$\{0, 12, 36, 83, 133\}$,\adfsplit
$\{0, 21, 58, 93, 171\}$,
$\{0, 28, 70, 104, 167\}$,
$\{0, 5, 23, 117, 165\}$,\adfsplit
$\{0, 4, 17, 69, 115\}$,
$\{0, 40, 80, 120, 184\}$

}
\adfLgap \noindent by the mapping:
$x \mapsto x +  j \adfmod{160}$ for $x < 160$,
$x \mapsto (x +  j \adfmod{20}) + 160$ for $160 \le x < 180$,
$x \mapsto (x +  j \adfmod{4}) + 180$ for $180 \le x < 184$,
$184 \mapsto 184$,
$0 \le j < 160$
 for the first ten blocks,
$0 \le j < 40$
 for the last block.
\ADFvfyParStart{(185, ((10, 160, ((160, 1), (20, 1), (4, 1), (1, 1))), (1, 40, ((160, 1), (20, 1), (4, 1), (1, 1)))), ((5, 32), (25, 1)))} 

\adfDgap
\noindent{\boldmath $ 5^{32} 45^{1} $}~
With the point set $\{0, 1, \dots, 204\}$ partitioned into
 residue classes modulo $32$ for $\{0, 1, \dots, 159\}$, and
 $\{160, 161, \dots, 204\}$,
 the design is generated from

\adfLgap {\adfBfont
$\{169, 28, 2, 16, 88\}$,
$\{186, 137, 94, 89, 124\}$,
$\{139, 173, 28, 56, 80\}$,\adfsplit
$\{95, 170, 32, 26, 116\}$,
$\{0, 1, 3, 10, 200\}$,
$\{0, 4, 15, 109, 180\}$,\adfsplit
$\{0, 8, 46, 62, 93\}$,
$\{0, 18, 41, 133, 188\}$,
$\{0, 19, 53, 135, 171\}$,\adfsplit
$\{0, 17, 56, 127, 166\}$,
$\{0, 20, 81, 123, 163\}$,
$\{0, 22, 58, 87, 186\}$,\adfsplit
$\{0, 40, 80, 120, 204\}$

}
\adfLgap \noindent by the mapping:
$x \mapsto x +  j \adfmod{160}$ for $x < 160$,
$x \mapsto (x +  j \adfmod{40}) + 160$ for $160 \le x < 200$,
$x \mapsto (x +  j \adfmod{4}) + 200$ for $200 \le x < 204$,
$204 \mapsto 204$,
$0 \le j < 160$
 for the first 12 blocks,
$0 \le j < 40$
 for the last block.
\ADFvfyParStart{(205, ((12, 160, ((160, 1), (40, 1), (4, 1), (1, 1))), (1, 40, ((160, 1), (40, 1), (4, 1), (1, 1)))), ((5, 32), (45, 1)))} 

\adfDgap
\noindent{\boldmath $ 5^{36} 45^{1} $}~
With the point set $\{0, 1, \dots, 224\}$ partitioned into
 residue classes modulo $36$ for $\{0, 1, \dots, 179\}$, and
 $\{180, 181, \dots, 224\}$,
 the design is generated from

\adfLgap {\adfBfont
$\{99, 71, 104, 57, 34\}$,
$\{9, 73, 88, 219, 164\}$,
$\{99, 130, 190, 16, 67\}$,\adfsplit
$\{75, 86, 214, 8, 92\}$,
$\{96, 109, 2, 28, 185\}$,
$\{0, 1, 4, 160, 162\}$,\adfsplit
$\{0, 35, 93, 141, 223\}$,
$\{0, 10, 59, 134, 203\}$,
$\{0, 8, 60, 100, 207\}$,\adfsplit
$\{0, 30, 71, 125, 186\}$,
$\{0, 9, 38, 136, 209\}$,
$\{0, 16, 43, 77, 185\}$,\adfsplit
$\{0, 7, 118, 130, 222\}$,
$\{0, 45, 90, 135, 180\}$

}
\adfLgap \noindent by the mapping:
$x \mapsto x +  j \adfmod{180}$ for $x < 180$,
$x \mapsto (x +  j \adfmod{45}) + 180$ for $x \ge 180$,
$0 \le j < 180$
 for the first 13 blocks,
$0 \le j < 45$
 for the last block.
\ADFvfyParStart{(225, ((13, 180, ((180, 1), (45, 1))), (1, 45, ((180, 1), (45, 1)))), ((5, 36), (45, 1)))} 

\adfDgap
\noindent{\boldmath $ 5^{40} 45^{1} $}~
With the point set $\{0, 1, \dots, 244\}$ partitioned into
 residue classes modulo $40$ for $\{0, 1, \dots, 199\}$, and
 $\{200, 201, \dots, 244\}$,
 the design is generated from

\adfLgap {\adfBfont
$\{100, 67, 95, 19, 199\}$,
$\{230, 112, 88, 70, 186\}$,
$\{98, 88, 61, 104, 210\}$,\adfsplit
$\{243, 153, 6, 75, 36\}$,
$\{75, 0, 94, 12, 23\}$,
$\{213, 125, 34, 36, 94\}$,\adfsplit
$\{0, 4, 17, 107, 205\}$,
$\{0, 1, 46, 55, 231\}$,
$\{0, 15, 47, 88, 221\}$,\adfsplit
$\{0, 8, 34, 70, 95\}$,
$\{0, 3, 38, 67, 210\}$,
$\{0, 7, 92, 151, 227\}$,\adfsplit
$\{0, 14, 65, 86, 209\}$,
$\{0, 22, 79, 156, 233\}$,
$\{0, 50, 100, 150, 244\}$

}
\adfLgap \noindent by the mapping:
$x \mapsto x +  j \adfmod{200}$ for $x < 200$,
$x \mapsto (x +  j \adfmod{40}) + 200$ for $200 \le x < 240$,
$x \mapsto (x +  j \adfmod{4}) + 240$ for $240 \le x < 244$,
$244 \mapsto 244$,
$0 \le j < 200$
 for the first 14 blocks,
$0 \le j < 50$
 for the last block.
\ADFvfyParStart{(245, ((14, 200, ((200, 1), (40, 1), (4, 1), (1, 1))), (1, 50, ((200, 1), (40, 1), (4, 1), (1, 1)))), ((5, 40), (45, 1)))} 

\adfDgap
\noindent{\boldmath $ 5^{44} 25^{1} $}~
With the point set $\{0, 1, \dots, 244\}$ partitioned into
 residue classes modulo $44$ for $\{0, 1, \dots, 219\}$, and
 $\{220, 221, \dots, 244\}$,
 the design is generated from

\adfLgap {\adfBfont
$\{48, 152, 234, 101, 49\}$,
$\{20, 181, 213, 144, 38\}$,
$\{216, 27, 93, 118, 135\}$,\adfsplit
$\{14, 148, 158, 184, 156\}$,
$\{142, 119, 231, 176, 17\}$,
$\{0, 3, 9, 14, 240\}$,\adfsplit
$\{0, 22, 62, 127, 162\}$,
$\{0, 13, 33, 54, 146\}$,
$\{0, 4, 16, 64, 153\}$,\adfsplit
$\{0, 15, 99, 145, 223\}$,
$\{0, 7, 63, 101, 148\}$,
$\{0, 29, 68, 138, 239\}$,\adfsplit
$\{0, 19, 49, 196, 236\}$,
$\{0, 55, 110, 165, 244\}$

}
\adfLgap \noindent by the mapping:
$x \mapsto x +  j \adfmod{220}$ for $x < 220$,
$x \mapsto (x +  j \adfmod{20}) + 220$ for $220 \le x < 240$,
$x \mapsto (x +  j \adfmod{4}) + 240$ for $240 \le x < 244$,
$244 \mapsto 244$,
$0 \le j < 220$
 for the first 13 blocks,
$0 \le j < 55$
 for the last block.
\ADFvfyParStart{(245, ((13, 220, ((220, 1), (20, 1), (4, 1), (1, 1))), (1, 55, ((220, 1), (20, 1), (4, 1), (1, 1)))), ((5, 44), (25, 1)))} 

\adfDgap
\noindent{\boldmath $ 5^{44} 45^{1} $}~
With the point set $\{0, 1, \dots, 264\}$ partitioned into
 residue classes modulo $44$ for $\{0, 1, \dots, 219\}$, and
 $\{220, 221, \dots, 264\}$,
 the design is generated from

\adfLgap {\adfBfont
$\{238, 107, 38, 100, 124\}$,
$\{144, 135, 226, 109, 104\}$,
$\{54, 151, 223, 214, 102\}$,\adfsplit
$\{4, 72, 260, 33, 1\}$,
$\{64, 164, 250, 202, 148\}$,
$\{258, 218, 11, 21, 135\}$,\adfsplit
$\{230, 134, 123, 53, 33\}$,
$\{0, 12, 65, 117, 145\}$,
$\{0, 8, 164, 178, 233\}$,\adfsplit
$\{0, 30, 89, 146, 252\}$,
$\{0, 15, 37, 73, 235\}$,
$\{0, 4, 111, 129, 240\}$,\adfsplit
$\{0, 33, 67, 118, 159\}$,
$\{0, 1, 77, 79, 122\}$,
$\{0, 6, 25, 72, 199\}$,\adfsplit
$\{0, 55, 110, 165, 264\}$

}
\adfLgap \noindent by the mapping:
$x \mapsto x +  j \adfmod{220}$ for $x < 220$,
$x \mapsto (x +  j \adfmod{44}) + 220$ for $220 \le x < 264$,
$264 \mapsto 264$,
$0 \le j < 220$
 for the first 15 blocks,
$0 \le j < 55$
 for the last block.
\ADFvfyParStart{(265, ((15, 220, ((220, 1), (44, 1), (1, 1))), (1, 55, ((220, 1), (44, 1), (1, 1)))), ((5, 44), (45, 1)))} 

\adfDgap
\noindent{\boldmath $ 6^{16} 10^{1} $}~
With the point set $\{0, 1, \dots, 105\}$ partitioned into
 residue classes modulo $16$ for $\{0, 1, \dots, 95\}$, and
 $\{96, 97, \dots, 105\}$,
 the design is generated from

\adfLgap {\adfBfont
$\{105, 17, 83, 38, 28\}$,
$\{20, 99, 63, 48, 94\}$,
$\{95, 58, 60, 11, 84\}$,\adfsplit
$\{0, 1, 4, 7, 40\}$,
$\{0, 9, 18, 52, 103\}$,
$\{0, 29, 31, 51, 58\}$,\adfsplit
$\{0, 12, 66, 83, 91\}$,
$\{0, 5, 6, 33, 88\}$,
$\{0, 19, 61, 76, 96\}$,\adfsplit
$\{1, 11, 35, 61, 102\}$,
$\{0, 21, 59, 73, 77\}$

}
\adfLgap \noindent by the mapping:
$x \mapsto x + 2 j \adfmod{96}$ for $x < 96$,
$x \mapsto (x +  j \adfmod{8}) + 96$ for $96 \le x < 104$,
$x \mapsto (x +  j \adfmod{2}) + 104$ for $x \ge 104$,
$0 \le j < 48$.
\ADFvfyParStart{(106, ((11, 48, ((96, 2), (8, 1), (2, 1)))), ((6, 16), (10, 1)))} 

\adfDgap
\noindent{\boldmath $ 8^{8} 12^{1} $}~
With the point set $\{0, 1, \dots, 75\}$ partitioned into
 residue classes modulo $8$ for $\{0, 1, \dots, 63\}$, and
 $\{64, 65, \dots, 75\}$,
 the design is generated from

\adfLgap {\adfBfont
$\{49, 48, 72, 18, 31\}$,
$\{55, 48, 50, 12, 69\}$,
$\{20, 26, 29, 40, 68\}$,\adfsplit
$\{0, 4, 19, 29, 41\}$

}
\adfLgap \noindent by the mapping:
$x \mapsto x +  j \adfmod{64}$ for $x < 64$,
$x \mapsto (x +  j \adfmod{8}) + 64$ for $64 \le x < 72$,
$x \mapsto (x +  j \adfmod{4}) + 72$ for $x \ge 72$,
$0 \le j < 64$.
\ADFvfyParStart{(76, ((4, 64, ((64, 1), (8, 1), (4, 1)))), ((8, 8), (12, 1)))} 

\adfDgap
\noindent{\boldmath $ 8^{10} 16^{1} $}~
With the point set $\{0, 1, \dots, 95\}$ partitioned into
 residue classes modulo $10$ for $\{0, 1, \dots, 79\}$, and
 $\{80, 81, \dots, 95\}$,
 the design is generated from

\adfLgap {\adfBfont
$\{0, 1, 4, 6, 23\}$,
$\{0, 8, 42, 51, 90\}$,
$\{0, 11, 26, 44, 89\}$,\adfsplit
$\{0, 7, 21, 56, 91\}$,
$\{0, 12, 25, 53, 81\}$,
$\{0, 16, 32, 48, 64\}$

}
\adfLgap \noindent by the mapping:
$x \mapsto x +  j \adfmod{80}$ for $x < 80$,
$x \mapsto (x +  j \adfmod{16}) + 80$ for $x \ge 80$,
$0 \le j < 80$
 for the first five blocks,
$0 \le j < 16$
 for the last block.
\ADFvfyParStart{(96, ((5, 80, ((80, 1), (16, 1))), (1, 16, ((80, 1), (16, 1)))), ((8, 10), (16, 1)))} 

\adfDgap
\noindent{\boldmath $ 8^{10} 20^{1} $}~
With the point set $\{0, 1, \dots, 99\}$ partitioned into
 residue classes modulo $10$ for $\{0, 1, \dots, 79\}$, and
 $\{80, 81, \dots, 99\}$,
 the design is generated from

\adfLgap {\adfBfont
$\{90, 78, 79, 50, 76\}$,
$\{86, 50, 6, 55, 17\}$,
$\{63, 46, 11, 40, 29\}$,\adfsplit
$\{0, 4, 12, 79, 96\}$,
$\{0, 7, 14, 33, 84\}$,
$\{0, 18, 56, 77, 88\}$,\adfsplit
$\{0, 9, 25, 34, 86\}$,
$\{0, 22, 53, 65, 98\}$,
$\{0, 15, 37, 39, 82\}$,\adfsplit
$\{0, 13, 57, 61, 95\}$,
$\{0, 27, 35, 41, 93\}$,
$\{0, 16, 32, 48, 64\}$

}
\adfLgap \noindent by the mapping:
$x \mapsto x + 2 j \adfmod{80}$ for $x < 80$,
$x \mapsto (x +  j \adfmod{20}) + 80$ for $x \ge 80$,
$0 \le j < 40$
 for the first 11 blocks,
$0 \le j < 8$
 for the last block.
\ADFvfyParStart{(100, ((11, 40, ((80, 2), (20, 1))), (1, 8, ((80, 2), (20, 1)))), ((8, 10), (20, 1)))} 

\adfDgap
\noindent{\boldmath $ 8^{14} 28^{1} $}~
With the point set $\{0, 1, \dots, 139\}$ partitioned into
 residue classes modulo $14$ for $\{0, 1, \dots, 111\}$, and
 $\{112, 113, \dots, 139\}$,
 the design is generated from

\adfLgap {\adfBfont
$\{127, 27, 96, 108, 35\}$,
$\{1, 112, 16, 26, 102\}$,
$\{0, 1, 94, 110, 119\}$,\adfsplit
$\{0, 4, 68, 75, 116\}$,
$\{0, 13, 45, 91, 130\}$,
$\{0, 9, 38, 58, 136\}$,\adfsplit
$\{0, 6, 30, 53, 135\}$,
$\{0, 5, 27, 60, 77\}$

}
\adfLgap \noindent by the mapping:
$x \mapsto x +  j \adfmod{112}$ for $x < 112$,
$x \mapsto (x +  j \adfmod{28}) + 112$ for $x \ge 112$,
$0 \le j < 112$.
\ADFvfyParStart{(140, ((8, 112, ((112, 1), (28, 1)))), ((8, 14), (28, 1)))} 

\adfDgap
\noindent{\boldmath $ 8^{15} 4^{1} $}~
With the point set $\{0, 1, \dots, 123\}$ partitioned into
 residue classes modulo $15$ for $\{0, 1, \dots, 119\}$, and
 $\{120, 121, 122, 123\}$,
 the design is generated from

\adfLgap {\adfBfont
$\{101, 53, 15, 51, 9\}$,
$\{23, 1, 112, 2, 120\}$,
$\{0, 4, 43, 61, 68\}$,\adfsplit
$\{0, 3, 17, 29, 100\}$,
$\{0, 8, 41, 54, 73\}$,
$\{0, 5, 16, 40, 67\}$

}
\adfLgap \noindent by the mapping:
$x \mapsto x +  j \adfmod{120}$ for $x < 120$,
$x \mapsto (x +  j \adfmod{4}) + 120$ for $x \ge 120$,
$0 \le j < 120$.
\ADFvfyParStart{(124, ((6, 120, ((120, 1), (4, 1)))), ((8, 15), (4, 1)))} 

\adfDgap
\noindent{\boldmath $ 8^{15} 16^{1} $}~
With the point set $\{0, 1, \dots, 135\}$ partitioned into
 residue classes modulo $15$ for $\{0, 1, \dots, 119\}$, and
 $\{120, 121, \dots, 135\}$,
 the design is generated from

\adfLgap {\adfBfont
$\{0, 1, 4, 6, 110\}$,
$\{0, 7, 39, 47, 56\}$,
$\{0, 12, 62, 89, 122\}$,\adfsplit
$\{0, 18, 53, 86, 128\}$,
$\{0, 26, 55, 92, 121\}$,
$\{0, 13, 36, 57, 82\}$,\adfsplit
$\{0, 19, 78, 98, 131\}$,
$\{0, 24, 48, 72, 96\}$

}
\adfLgap \noindent by the mapping:
$x \mapsto x +  j \adfmod{120}$ for $x < 120$,
$x \mapsto (x - 120 + 2 j \adfmod{16}) + 120$ for $x \ge 120$,
$0 \le j < 120$
 for the first seven blocks;
$x \mapsto x +  j \adfmod{120}$ for $x < 120$,
$x \mapsto (x - 120 + 4 j \adfmod{16}) + 120$ for $x \ge 120$,
$0 \le j < 24$
 for the last block.
\ADFvfyParStart{(136, ((7, 120, ((120, 1), (16, 2))), (1, 24, ((120, 1), (16, 4)))), ((8, 15), (16, 1)))} 

\adfDgap
\noindent{\boldmath $ 8^{15} 24^{1} $}~
With the point set $\{0, 1, \dots, 143\}$ partitioned into
 residue classes modulo $15$ for $\{0, 1, \dots, 119\}$, and
 $\{120, 121, \dots, 143\}$,
 the design is generated from

\adfLgap {\adfBfont
$\{136, 16, 113, 110, 36\}$,
$\{67, 140, 31, 99, 14\}$,
$\{0, 1, 6, 10, 102\}$,\adfsplit
$\{0, 8, 22, 70, 91\}$,
$\{0, 12, 59, 93, 128\}$,
$\{0, 16, 57, 82, 139\}$,\adfsplit
$\{0, 7, 40, 71, 134\}$,
$\{0, 2, 13, 78, 138\}$

}
\adfLgap \noindent by the mapping:
$x \mapsto x +  j \adfmod{120}$ for $x < 120$,
$x \mapsto (x +  j \adfmod{24}) + 120$ for $x \ge 120$,
$0 \le j < 120$.
\ADFvfyParStart{(144, ((8, 120, ((120, 1), (24, 1)))), ((8, 15), (24, 1)))} 

\adfDgap
\noindent{\boldmath $ 8^{15} 36^{1} $}~
With the point set $\{0, 1, \dots, 155\}$ partitioned into
 residue classes modulo $15$ for $\{0, 1, \dots, 119\}$, and
 $\{120, 121, \dots, 155\}$,
 the design is generated from

\adfLgap {\adfBfont
$\{80, 58, 155, 119, 54\}$,
$\{74, 24, 129, 66, 101\}$,
$\{0, 1, 3, 114, 154\}$,\adfsplit
$\{0, 14, 33, 54, 152\}$,
$\{0, 5, 36, 88, 125\}$,
$\{0, 11, 62, 82, 128\}$,\adfsplit
$\{0, 10, 63, 86, 136\}$,
$\{0, 17, 46, 64, 132\}$,
$\{0, 12, 25, 104, 131\}$,\adfsplit
$\{0, 24, 48, 72, 96\}$

}
\adfLgap \noindent by the mapping:
$x \mapsto x +  j \adfmod{120}$ for $x < 120$,
$x \mapsto (x +  j \adfmod{24}) + 120$ for $120 \le x < 144$,
$x \mapsto (x +  j \adfmod{12}) + 144$ for $x \ge 144$,
$0 \le j < 120$
 for the first nine blocks,
$0 \le j < 24$
 for the last block.
\ADFvfyParStart{(156, ((9, 120, ((120, 1), (24, 1), (12, 1))), (1, 24, ((120, 1), (24, 1), (12, 1)))), ((8, 15), (36, 1)))} 

\adfDgap
\noindent{\boldmath $ 8^{16} 20^{1} $}~
With the point set $\{0, 1, \dots, 147\}$ partitioned into
 residue classes modulo $16$ for $\{0, 1, \dots, 127\}$, and
 $\{128, 129, \dots, 147\}$,
 the design is generated from

\adfLgap {\adfBfont
$\{21, 31, 76, 89, 27\}$,
$\{21, 74, 105, 139, 54\}$,
$\{0, 1, 3, 110, 144\}$,\adfsplit
$\{0, 8, 71, 101, 131\}$,
$\{0, 9, 34, 90, 137\}$,
$\{0, 5, 17, 59, 105\}$,\adfsplit
$\{0, 11, 26, 50, 87\}$,
$\{0, 7, 92, 106, 132\}$

}
\adfLgap \noindent by the mapping:
$x \mapsto x +  j \adfmod{128}$ for $x < 128$,
$x \mapsto (x +  j \adfmod{16}) + 128$ for $128 \le x < 144$,
$x \mapsto (x +  j \adfmod{4}) + 144$ for $x \ge 144$,
$0 \le j < 128$.
\ADFvfyParStart{(148, ((8, 128, ((128, 1), (16, 1), (4, 1)))), ((8, 16), (20, 1)))} 

\adfDgap
\noindent{\boldmath $ 8^{17} 16^{1} $}~
With the point set $\{0, 1, \dots, 151\}$ partitioned into
 residue classes modulo $17$ for $\{0, 1, \dots, 135\}$, and
 $\{136, 137, \dots, 151\}$,
 the design is generated from

\adfLgap {\adfBfont
$\{82, 29, 94, 18, 126\}$,
$\{60, 24, 102, 87, 40\}$,
$\{0, 2, 5, 24, 45\}$,\adfsplit
$\{0, 8, 38, 56, 69\}$,
$\{0, 7, 57, 66, 140\}$,
$\{0, 1, 26, 55, 148\}$,\adfsplit
$\{0, 4, 103, 126, 137\}$,
$\{0, 6, 52, 101, 143\}$

}
\adfLgap \noindent by the mapping:
$x \mapsto x +  j \adfmod{136}$ for $x < 136$,
$x \mapsto (x - 136 + 2 j \adfmod{16}) + 136$ for $x \ge 136$,
$0 \le j < 136$.
\ADFvfyParStart{(152, ((8, 136, ((136, 1), (16, 2)))), ((8, 17), (16, 1)))} 

\adfDgap
\noindent{\boldmath $ 8^{18} 32^{1} $}~
With the point set $\{0, 1, \dots, 175\}$ partitioned into
 residue classes modulo $18$ for $\{0, 1, \dots, 143\}$, and
 $\{144, 145, \dots, 175\}$,
 the design is generated from

\adfLgap {\adfBfont
$\{60, 54, 100, 156, 135\}$,
$\{162, 11, 60, 77, 3\}$,
$\{45, 116, 139, 25, 155\}$,\adfsplit
$\{0, 1, 3, 27, 152\}$,
$\{0, 13, 51, 113, 160\}$,
$\{0, 10, 32, 96, 111\}$,\adfsplit
$\{0, 9, 21, 37, 76\}$,
$\{0, 4, 29, 88, 167\}$,
$\{0, 5, 19, 102, 155\}$,\adfsplit
$\{0, 7, 41, 52, 171\}$

}
\adfLgap \noindent by the mapping:
$x \mapsto x +  j \adfmod{144}$ for $x < 144$,
$x \mapsto (x - 144 + 2 j \adfmod{32}) + 144$ for $x \ge 144$,
$0 \le j < 144$.
\ADFvfyParStart{(176, ((10, 144, ((144, 1), (32, 2)))), ((8, 18), (32, 1)))} 

\adfDgap
\noindent{\boldmath $ 8^{20} 44^{1} $}~
With the point set $\{0, 1, \dots, 203\}$ partitioned into
 residue classes modulo $20$ for $\{0, 1, \dots, 159\}$, and
 $\{160, 161, \dots, 203\}$,
 the design is generated from

\adfLgap {\adfBfont
$\{98, 87, 68, 183, 84\}$,
$\{59, 81, 116, 162, 15\}$,
$\{22, 86, 53, 159, 77\}$,\adfsplit
$\{174, 126, 134, 152, 73\}$,
$\{43, 201, 85, 0, 154\}$,
$\{0, 1, 29, 113, 164\}$,\adfsplit
$\{0, 2, 67, 88, 180\}$,
$\{0, 25, 62, 133, 170\}$,
$\{0, 4, 50, 147, 177\}$,\adfsplit
$\{0, 12, 70, 104, 186\}$,
$\{0, 5, 41, 150, 199\}$,
$\{0, 7, 39, 122, 191\}$

}
\adfLgap \noindent by the mapping:
$x \mapsto x +  j \adfmod{160}$ for $x < 160$,
$x \mapsto (x +  j \adfmod{40}) + 160$ for $160 \le x < 200$,
$x \mapsto (x +  j \adfmod{4}) + 200$ for $x \ge 200$,
$0 \le j < 160$.
\ADFvfyParStart{(204, ((12, 160, ((160, 1), (40, 1), (4, 1)))), ((8, 20), (44, 1)))} 

\adfDgap
\noindent{\boldmath $ 8^{21} 20^{1} $}~
With the point set $\{0, 1, \dots, 187\}$ partitioned into
 residue classes modulo $21$ for $\{0, 1, \dots, 167\}$, and
 $\{168, 169, \dots, 187\}$,
 the design is generated from

\adfLgap {\adfBfont
$\{127, 171, 77, 36, 42\}$,
$\{145, 58, 180, 136, 11\}$,
$\{48, 0, 33, 29, 103\}$,\adfsplit
$\{0, 1, 3, 146, 172\}$,
$\{0, 5, 18, 67, 175\}$,
$\{0, 7, 38, 89, 184\}$,\adfsplit
$\{0, 11, 27, 122, 142\}$,
$\{0, 12, 64, 92, 109\}$,
$\{0, 10, 40, 54, 112\}$,\adfsplit
$\{0, 8, 32, 68, 107\}$

}
\adfLgap \noindent by the mapping:
$x \mapsto x +  j \adfmod{168}$ for $x < 168$,
$x \mapsto (x - 168 + 2 j \adfmod{16}) + 168$ for $168 \le x < 184$,
$x \mapsto (x +  j \adfmod{4}) + 184$ for $x \ge 184$,
$0 \le j < 168$.
\ADFvfyParStart{(188, ((10, 168, ((168, 1), (16, 2), (4, 1)))), ((8, 21), (20, 1)))} 

\adfDgap
\noindent{\boldmath $ 8^{21} 40^{1} $}~
With the point set $\{0, 1, \dots, 207\}$ partitioned into
 residue classes modulo $21$ for $\{0, 1, \dots, 167\}$, and
 $\{168, 169, \dots, 207\}$,
 the design is generated from

\adfLgap {\adfBfont
$\{178, 42, 76, 70, 122\}$,
$\{116, 191, 138, 85, 21\}$,
$\{88, 174, 58, 47, 45\}$,\adfsplit
$\{142, 189, 125, 109, 55\}$,
$\{0, 1, 103, 165, 192\}$,
$\{0, 7, 57, 86, 204\}$,\adfsplit
$\{0, 12, 48, 72, 109\}$,
$\{0, 8, 27, 47, 85\}$,
$\{0, 10, 25, 100, 181\}$,\adfsplit
$\{0, 18, 67, 112, 175\}$,
$\{0, 23, 92, 136, 194\}$,
$\{0, 5, 14, 40, 191\}$

}
\adfLgap \noindent by the mapping:
$x \mapsto x +  j \adfmod{168}$ for $x < 168$,
$x \mapsto (x - 168 + 3 j \adfmod{36}) + 168$ for $168 \le x < 204$,
$x \mapsto (x +  j \adfmod{4}) + 204$ for $x \ge 204$,
$0 \le j < 168$.
\ADFvfyParStart{(208, ((12, 168, ((168, 1), (36, 3), (4, 1)))), ((8, 21), (40, 1)))} 

\adfDgap
\noindent{\boldmath $ 8^{22} 16^{1} $}~
With the point set $\{0, 1, \dots, 191\}$ partitioned into
 residue classes modulo $22$ for $\{0, 1, \dots, 175\}$, and
 $\{176, 177, \dots, 191\}$,
 the design is generated from

\adfLgap {\adfBfont
$\{191, 130, 134, 128, 107\}$,
$\{62, 112, 77, 15, 44\}$,
$\{152, 57, 89, 37, 106\}$,\adfsplit
$\{0, 1, 8, 11, 104\}$,
$\{0, 5, 36, 75, 123\}$,
$\{0, 16, 57, 76, 121\}$,\adfsplit
$\{0, 14, 42, 98, 136\}$,
$\{0, 13, 37, 146, 179\}$,
$\{0, 25, 59, 150, 176\}$,\adfsplit
$\{0, 9, 99, 111, 187\}$

}
\adfLgap \noindent by the mapping:
$x \mapsto x +  j \adfmod{176}$ for $x < 176$,
$x \mapsto (x +  j \adfmod{16}) + 176$ for $x \ge 176$,
$0 \le j < 176$.
\ADFvfyParStart{(192, ((10, 176, ((176, 1), (16, 1)))), ((8, 22), (16, 1)))} 

\adfDgap
\noindent{\boldmath $ 8^{22} 36^{1} $}~
With the point set $\{0, 1, \dots, 211\}$ partitioned into
 residue classes modulo $22$ for $\{0, 1, \dots, 175\}$, and
 $\{176, 177, \dots, 211\}$,
 the design is generated from

\adfLgap {\adfBfont
$\{40, 88, 91, 38, 130\}$,
$\{114, 49, 105, 31, 205\}$,
$\{101, 198, 162, 127, 115\}$,\adfsplit
$\{58, 13, 205, 68, 30\}$,
$\{0, 1, 6, 31, 208\}$,
$\{0, 4, 15, 122, 181\}$,\adfsplit
$\{0, 29, 75, 127, 189\}$,
$\{0, 8, 32, 99, 112\}$,
$\{0, 23, 105, 139, 204\}$,\adfsplit
$\{0, 16, 36, 57, 133\}$,
$\{0, 7, 40, 113, 180\}$,
$\{0, 19, 81, 108, 176\}$

}
\adfLgap \noindent by the mapping:
$x \mapsto x +  j \adfmod{176}$ for $x < 176$,
$x \mapsto (x - 176 + 2 j \adfmod{32}) + 176$ for $176 \le x < 208$,
$x \mapsto (x +  j \adfmod{4}) + 208$ for $x \ge 208$,
$0 \le j < 176$.
\ADFvfyParStart{(212, ((12, 176, ((176, 1), (32, 2), (4, 1)))), ((8, 22), (36, 1)))} 

\adfDgap
\noindent{\boldmath $ 8^{23} 32^{1} $}~
With the point set $\{0, 1, \dots, 215\}$ partitioned into
 residue classes modulo $23$ for $\{0, 1, \dots, 183\}$, and
 $\{184, 185, \dots, 215\}$,
 the design is generated from

\adfLgap {\adfBfont
$\{120, 119, 172, 2, 29\}$,
$\{186, 180, 73, 182, 101\}$,
$\{24, 20, 209, 54, 98\}$,\adfsplit
$\{205, 142, 18, 152, 116\}$,
$\{0, 3, 20, 25, 190\}$,
$\{0, 7, 15, 39, 102\}$,\adfsplit
$\{0, 13, 29, 112, 149\}$,
$\{0, 9, 49, 80, 137\}$,
$\{0, 21, 76, 146, 204\}$,\adfsplit
$\{0, 6, 139, 172, 192\}$,
$\{0, 19, 61, 119, 187\}$,
$\{0, 11, 73, 141, 215\}$

}
\adfLgap \noindent by the mapping:
$x \mapsto x +  j \adfmod{184}$ for $x < 184$,
$x \mapsto (x - 184 + 4 j \adfmod{32}) + 184$ for $x \ge 184$,
$0 \le j < 184$.
\ADFvfyParStart{(216, ((12, 184, ((184, 1), (32, 4)))), ((8, 23), (32, 1)))} 

\adfDgap
\noindent{\boldmath $ 8^{24} 28^{1} $}~
With the point set $\{0, 1, \dots, 219\}$ partitioned into
 residue classes modulo $24$ for $\{0, 1, \dots, 191\}$, and
 $\{192, 193, \dots, 219\}$,
 the design is generated from

\adfLgap {\adfBfont
$\{187, 133, 147, 68, 180\}$,
$\{56, 62, 106, 44, 87\}$,
$\{119, 111, 25, 91, 182\}$,\adfsplit
$\{186, 3, 118, 134, 206\}$,
$\{71, 75, 174, 92, 175\}$,
$\{0, 2, 15, 57, 216\}$,\adfsplit
$\{0, 29, 74, 143, 198\}$,
$\{0, 5, 107, 158, 200\}$,
$\{0, 32, 70, 116, 205\}$,\adfsplit
$\{0, 3, 26, 56, 67\}$,
$\{0, 22, 58, 133, 204\}$,
$\{0, 10, 105, 165, 211\}$

}
\adfLgap \noindent by the mapping:
$x \mapsto x +  j \adfmod{192}$ for $x < 192$,
$x \mapsto (x +  j \adfmod{24}) + 192$ for $192 \le x < 216$,
$x \mapsto (x +  j \adfmod{4}) + 216$ for $x \ge 216$,
$0 \le j < 192$.
\ADFvfyParStart{(220, ((12, 192, ((192, 1), (24, 1), (4, 1)))), ((8, 24), (28, 1)))} 

\adfDgap
\noindent{\boldmath $ 8^{25} 4^{1} $}~
With the point set $\{0, 1, \dots, 203\}$ partitioned into
 residue classes modulo $25$ for $\{0, 1, \dots, 199\}$, and
 $\{200, 201, 202, 203\}$,
 the design is generated from

\adfLgap {\adfBfont
$\{61, 123, 191, 58, 177\}$,
$\{126, 154, 132, 152, 9\}$,
$\{127, 42, 131, 95, 173\}$,\adfsplit
$\{0, 1, 30, 35, 200\}$,
$\{0, 9, 27, 91, 107\}$,
$\{0, 13, 58, 79, 153\}$,\adfsplit
$\{0, 8, 19, 52, 160\}$,
$\{0, 15, 38, 101, 144\}$,
$\{0, 7, 17, 104, 176\}$,\adfsplit
$\{0, 12, 49, 88, 139\}$

}
\adfLgap \noindent by the mapping:
$x \mapsto x +  j \adfmod{200}$ for $x < 200$,
$x \mapsto (x +  j \adfmod{4}) + 200$ for $x \ge 200$,
$0 \le j < 200$.
\ADFvfyParStart{(204, ((10, 200, ((200, 1), (4, 1)))), ((8, 25), (4, 1)))} 

\adfDgap
\noindent{\boldmath $ 8^{25} 24^{1} $}~
With the point set $\{0, 1, \dots, 223\}$ partitioned into
 residue classes modulo $25$ for $\{0, 1, \dots, 199\}$, and
 $\{200, 201, \dots, 223\}$,
 the design is generated from

\adfLgap {\adfBfont
$\{67, 86, 68, 183, 121\}$,
$\{178, 59, 217, 68, 12\}$,
$\{168, 154, 37, 26, 63\}$,\adfsplit
$\{77, 163, 14, 183, 50\}$,
$\{159, 205, 181, 24, 183\}$,
$\{0, 3, 10, 161, 220\}$,\adfsplit
$\{0, 5, 21, 129, 141\}$,
$\{0, 6, 46, 79, 107\}$,
$\{0, 8, 23, 52, 112\}$,\adfsplit
$\{0, 4, 74, 122, 217\}$,
$\{0, 13, 30, 68, 200\}$,
$\{0, 32, 77, 134, 208\}$

}
\adfLgap \noindent by the mapping:
$x \mapsto x +  j \adfmod{200}$ for $x < 200$,
$x \mapsto (x +  j \adfmod{20}) + 200$ for $200 \le x < 220$,
$x \mapsto (x +  j \adfmod{4}) + 220$ for $x \ge 220$,
$0 \le j < 200$.
\ADFvfyParStart{(224, ((12, 200, ((200, 1), (20, 1), (4, 1)))), ((8, 25), (24, 1)))} 

\adfDgap
\noindent{\boldmath $ 8^{26} 20^{1} $}~
With the point set $\{0, 1, \dots, 227\}$ partitioned into
 residue classes modulo $26$ for $\{0, 1, \dots, 207\}$, and
 $\{208, 209, \dots, 227\}$,
 the design is generated from

\adfLgap {\adfBfont
$\{216, 58, 130, 112, 1\}$,
$\{222, 17, 139, 34, 5\}$,
$\{164, 68, 27, 115, 70\}$,\adfsplit
$\{201, 173, 31, 176, 195\}$,
$\{55, 203, 199, 78, 71\}$,
$\{0, 9, 30, 115, 224\}$,\adfsplit
$\{0, 5, 51, 82, 152\}$,
$\{0, 1, 40, 195, 210\}$,
$\{0, 10, 65, 181, 213\}$,\adfsplit
$\{0, 8, 67, 99, 135\}$,
$\{0, 11, 35, 69, 119\}$,
$\{0, 15, 48, 90, 110\}$

}
\adfLgap \noindent by the mapping:
$x \mapsto x +  j \adfmod{208}$ for $x < 208$,
$x \mapsto (x +  j \adfmod{16}) + 208$ for $208 \le x < 224$,
$x \mapsto (x +  j \adfmod{4}) + 224$ for $x \ge 224$,
$0 \le j < 208$.
\ADFvfyParStart{(228, ((12, 208, ((208, 1), (16, 1), (4, 1)))), ((8, 26), (20, 1)))} 

\adfDgap
\noindent{\boldmath $ 8^{27} 16^{1} $}~
With the point set $\{0, 1, \dots, 231\}$ partitioned into
 residue classes modulo $27$ for $\{0, 1, \dots, 215\}$, and
 $\{216, 217, \dots, 231\}$,
 the design is generated from

\adfLgap {\adfBfont
$\{133, 29, 73, 120, 195\}$,
$\{49, 198, 23, 126, 6\}$,
$\{127, 75, 109, 38, 43\}$,\adfsplit
$\{206, 212, 198, 177, 70\}$,
$\{136, 26, 187, 111, 220\}$,
$\{0, 1, 39, 138, 228\}$,\adfsplit
$\{0, 2, 48, 59, 63\}$,
$\{0, 16, 36, 118, 183\}$,
$\{0, 12, 40, 70, 123\}$,\adfsplit
$\{0, 3, 119, 209, 226\}$,
$\{0, 19, 42, 87, 143\}$,
$\{0, 9, 31, 95, 219\}$

}
\adfLgap \noindent by the mapping:
$x \mapsto x +  j \adfmod{216}$ for $x < 216$,
$x \mapsto (x +  j \adfmod{12}) + 216$ for $216 \le x < 228$,
$x \mapsto (x +  j \adfmod{4}) + 228$ for $x \ge 228$,
$0 \le j < 216$.
\ADFvfyParStart{(232, ((12, 216, ((216, 1), (12, 1), (4, 1)))), ((8, 27), (16, 1)))} 

\adfDgap
\noindent{\boldmath $ 8^{28} 12^{1} $}~
With the point set $\{0, 1, \dots, 235\}$ partitioned into
 residue classes modulo $28$ for $\{0, 1, \dots, 223\}$, and
 $\{224, 225, \dots, 235\}$,
 the design is generated from

\adfLgap {\adfBfont
$\{136, 105, 86, 206, 54\}$,
$\{235, 68, 221, 2, 115\}$,
$\{228, 146, 220, 144, 143\}$,\adfsplit
$\{48, 75, 36, 207, 153\}$,
$\{200, 86, 222, 11, 5\}$,
$\{0, 4, 126, 137, 231\}$,\adfsplit
$\{0, 10, 48, 103, 144\}$,
$\{0, 18, 42, 86, 181\}$,
$\{0, 8, 33, 67, 97\}$,\adfsplit
$\{0, 15, 52, 73, 109\}$,
$\{0, 16, 79, 124, 141\}$,
$\{0, 9, 23, 49, 69\}$

}
\adfLgap \noindent by the mapping:
$x \mapsto x +  j \adfmod{224}$ for $x < 224$,
$x \mapsto (x +  j \adfmod{8}) + 224$ for $224 \le x < 232$,
$x \mapsto (x +  j \adfmod{4}) + 232$ for $x \ge 232$,
$0 \le j < 224$.
\ADFvfyParStart{(236, ((12, 224, ((224, 1), (8, 1), (4, 1)))), ((8, 28), (12, 1)))} 

\adfDgap
\noindent{\boldmath $ 9^{8} 1^{1} $}~
With the point set $\{0, 1, \dots, 72\}$ partitioned into
 residue classes modulo $8$ for $\{0, 1, \dots, 71\}$, and
 $\{72\}$,
 the design is generated from

\adfLgap {\adfBfont
$\{0, 1, 5, 28, 35\}$,
$\{0, 2, 13, 19, 22\}$,
$\{0, 10, 25, 39, 51\}$,\adfsplit
$\{0, 18, 36, 54, 72\}$

}
\adfLgap \noindent by the mapping:
$x \mapsto x +  j \adfmod{72}$ for $x < 72$,
$72 \mapsto 72$,
$0 \le j < 72$
 for the first three blocks,
$0 \le j < 18$
 for the last block.
\ADFvfyParStart{(73, ((3, 72, ((72, 1), (1, 1))), (1, 18, ((72, 1), (1, 1)))), ((9, 8), (1, 1)))} 

\adfDgap
\noindent{\boldmath $ 9^{12} 13^{1} $}~
With the point set $\{0, 1, \dots, 120\}$ partitioned into
 residue classes modulo $12$ for $\{0, 1, \dots, 107\}$, and
 $\{108, 109, \dots, 120\}$,
 the design is generated from

\adfLgap {\adfBfont
$\{0, 1, 4, 6, 94\}$,
$\{0, 7, 37, 45, 68\}$,
$\{0, 13, 57, 79, 109\}$,\adfsplit
$\{0, 10, 62, 97, 117\}$,
$\{0, 9, 28, 67, 83\}$,
$\{0, 17, 49, 82, 111\}$,\adfsplit
$\{0, 27, 54, 81, 120\}$

}
\adfLgap \noindent by the mapping:
$x \mapsto x +  j \adfmod{108}$ for $x < 108$,
$x \mapsto (x +  j \adfmod{12}) + 108$ for $108 \le x < 120$,
$120 \mapsto 120$,
$0 \le j < 108$
 for the first six blocks,
$0 \le j < 27$
 for the last block.
\ADFvfyParStart{(121, ((6, 108, ((108, 1), (12, 1), (1, 1))), (1, 27, ((108, 1), (12, 1), (1, 1)))), ((9, 12), (13, 1)))} 

\adfDgap
\noindent{\boldmath $ 9^{16} 5^{1} $}~
With the point set $\{0, 1, \dots, 148\}$ partitioned into
 residue classes modulo $16$ for $\{0, 1, \dots, 143\}$, and
 $\{144, 145, 146, 147, 148\}$,
 the design is generated from

\adfLgap {\adfBfont
$\{77, 25, 52, 19, 82\}$,
$\{20, 43, 30, 34, 22\}$,
$\{0, 1, 50, 91, 144\}$,\adfsplit
$\{0, 17, 45, 83, 105\}$,
$\{0, 3, 43, 62, 77\}$,
$\{0, 18, 47, 73, 93\}$,\adfsplit
$\{0, 7, 31, 42, 107\}$,
$\{0, 36, 72, 108, 148\}$

}
\adfLgap \noindent by the mapping:
$x \mapsto x +  j \adfmod{144}$ for $x < 144$,
$x \mapsto (x +  j \adfmod{4}) + 144$ for $144 \le x < 148$,
$148 \mapsto 148$,
$0 \le j < 144$
 for the first seven blocks,
$0 \le j < 36$
 for the last block.
\ADFvfyParStart{(149, ((7, 144, ((144, 1), (4, 1), (1, 1))), (1, 36, ((144, 1), (4, 1), (1, 1)))), ((9, 16), (5, 1)))} 

\adfDgap
\noindent{\boldmath $ 9^{16} 25^{1} $}~
With the point set $\{0, 1, \dots, 168\}$ partitioned into
 residue classes modulo $16$ for $\{0, 1, \dots, 143\}$, and
 $\{144, 145, \dots, 168\}$,
 the design is generated from

\adfLgap {\adfBfont
$\{162, 118, 115, 88, 28\}$,
$\{86, 159, 75, 12, 131\}$,
$\{99, 26, 57, 79, 112\}$,\adfsplit
$\{0, 4, 9, 50, 129\}$,
$\{0, 7, 75, 85, 162\}$,
$\{0, 2, 14, 97, 118\}$,\adfsplit
$\{0, 6, 35, 43, 163\}$,
$\{0, 1, 18, 52, 154\}$,
$\{0, 23, 62, 100, 165\}$,\adfsplit
$\{0, 36, 72, 108, 168\}$

}
\adfLgap \noindent by the mapping:
$x \mapsto x +  j \adfmod{144}$ for $x < 144$,
$x \mapsto (x +  j \adfmod{24}) + 144$ for $144 \le x < 168$,
$168 \mapsto 168$,
$0 \le j < 144$
 for the first nine blocks,
$0 \le j < 36$
 for the last block.
\ADFvfyParStart{(169, ((9, 144, ((144, 1), (24, 1), (1, 1))), (1, 36, ((144, 1), (24, 1), (1, 1)))), ((9, 16), (25, 1)))} 

\adfDgap
\noindent{\boldmath $ 9^{20} 17^{1} $}~
With the point set $\{0, 1, \dots, 196\}$ partitioned into
 residue classes modulo $20$ for $\{0, 1, \dots, 179\}$, and
 $\{180, 181, \dots, 196\}$,
 the design is generated from

\adfLgap {\adfBfont
$\{51, 81, 35, 37, 166\}$,
$\{104, 187, 125, 15, 42\}$,
$\{27, 20, 88, 136, 62\}$,\adfsplit
$\{0, 1, 79, 82, 181\}$,
$\{0, 5, 15, 38, 180\}$,
$\{0, 6, 19, 37, 182\}$,\adfsplit
$\{0, 4, 43, 54, 127\}$,
$\{0, 8, 36, 113, 122\}$,
$\{0, 17, 69, 93, 125\}$,\adfsplit
$\{0, 12, 34, 59, 151\}$,
$\{0, 45, 90, 135, 196\}$

}
\adfLgap \noindent by the mapping:
$x \mapsto x +  j \adfmod{180}$ for $x < 180$,
$x \mapsto (x - 180 + 4 j \adfmod{16}) + 180$ for $180 \le x < 196$,
$196 \mapsto 196$,
$0 \le j < 180$
 for the first ten blocks,
$0 \le j < 45$
 for the last block.
\ADFvfyParStart{(197, ((10, 180, ((180, 1), (16, 4), (1, 1))), (1, 45, ((180, 1), (16, 4), (1, 1)))), ((9, 20), (17, 1)))} 

\adfDgap
\noindent{\boldmath $ 9^{20} 29^{1} $}~
With the point set $\{0, 1, \dots, 208\}$ partitioned into
 residue classes modulo $20$ for $\{0, 1, \dots, 179\}$, and
 $\{180, 181, \dots, 208\}$,
 the design is generated from

\adfLgap {\adfBfont
$\{180, 33, 46, 103, 152\}$,
$\{25, 150, 204, 8, 111\}$,
$\{114, 10, 22, 146, 64\}$,\adfsplit
$\{82, 68, 109, 205, 47\}$,
$\{0, 3, 25, 34, 186\}$,
$\{0, 10, 47, 73, 182\}$,\adfsplit
$\{0, 15, 33, 102, 181\}$,
$\{0, 4, 11, 79, 95\}$,
$\{0, 1, 59, 65, 67\}$,\adfsplit
$\{0, 30, 81, 127, 185\}$,
$\{0, 5, 28, 137, 156\}$,
$\{0, 45, 90, 135, 208\}$,\adfsplit
$\{0, 36, 72, 108, 144\}$

}
\adfLgap \noindent by the mapping:
$x \mapsto x +  j \adfmod{180}$ for $x < 180$,
$x \mapsto (x - 180 + 7 j \adfmod{28}) + 180$ for $180 \le x < 208$,
$208 \mapsto 208$,
$0 \le j < 180$
 for the first 11 blocks,
$0 \le j < 45$
 for the next block,
$0 \le j < 36$
 for the last block.
\ADFvfyParStart{(209, ((11, 180, ((180, 1), (28, 7), (1, 1))), (1, 45, ((180, 1), (28, 7), (1, 1))), (1, 36, ((180, 1), (28, 7), (1, 1)))), ((9, 20), (29, 1)))} 

\adfDgap
\noindent{\boldmath $ 9^{20} 37^{1} $}~
With the point set $\{0, 1, \dots, 216\}$ partitioned into
 residue classes modulo $20$ for $\{0, 1, \dots, 179\}$, and
 $\{180, 181, \dots, 216\}$,
 the design is generated from

\adfLgap {\adfBfont
$\{0, 174, 37, 52, 189\}$,
$\{53, 17, 114, 121, 168\}$,
$\{46, 113, 94, 2, 180\}$,\adfsplit
$\{0, 134, 193, 24, 101\}$,
$\{94, 85, 169, 83, 156\}$,
$\{0, 8, 57, 95, 129\}$,\adfsplit
$\{0, 1, 99, 177, 187\}$,
$\{0, 5, 130, 158, 213\}$,
$\{0, 12, 150, 166, 204\}$,\adfsplit
$\{0, 10, 31, 63, 207\}$,
$\{0, 23, 64, 89, 202\}$,
$\{0, 17, 56, 162, 201\}$,\adfsplit
$\{0, 45, 90, 135, 216\}$

}
\adfLgap \noindent by the mapping:
$x \mapsto x +  j \adfmod{180}$ for $x < 180$,
$x \mapsto (x +  j \adfmod{36}) + 180$ for $180 \le x < 216$,
$216 \mapsto 216$,
$0 \le j < 180$
 for the first 12 blocks,
$0 \le j < 45$
 for the last block.
\ADFvfyParStart{(217, ((12, 180, ((180, 1), (36, 1), (1, 1))), (1, 45, ((180, 1), (36, 1), (1, 1)))), ((9, 20), (37, 1)))} 

\adfDgap
\noindent{\boldmath $ 9^{20} 49^{1} $}~
With the point set $\{0, 1, \dots, 228\}$ partitioned into
 residue classes modulo $20$ for $\{0, 1, \dots, 179\}$, and
 $\{180, 181, \dots, 228\}$,
 the design is generated from

\adfLgap {\adfBfont
$\{66, 60, 114, 64, 165\}$,
$\{59, 219, 26, 82, 75\}$,
$\{118, 115, 215, 87, 21\}$,\adfsplit
$\{74, 193, 95, 28, 10\}$,
$\{69, 86, 190, 160, 108\}$,
$\{0, 5, 13, 76, 223\}$,\adfsplit
$\{0, 1, 27, 62, 227\}$,
$\{0, 10, 133, 148, 192\}$,
$\{0, 12, 41, 65, 209\}$,\adfsplit
$\{0, 14, 84, 121, 199\}$,
$\{0, 11, 69, 161, 195\}$,
$\{0, 38, 93, 136, 181\}$,\adfsplit
$\{0, 9, 34, 112, 211\}$,
$\{0, 45, 90, 135, 228\}$,
$\{0, 36, 72, 108, 144\}$

}
\adfLgap \noindent by the mapping:
$x \mapsto x +  j \adfmod{180}$ for $x < 180$,
$x \mapsto (x +  j \adfmod{36}) + 180$ for $180 \le x < 216$,
$x \mapsto (x +  j \adfmod{12}) + 216$ for $216 \le x < 228$,
$228 \mapsto 228$,
$0 \le j < 180$
 for the first 13 blocks,
$0 \le j < 45$
 for the next block,
$0 \le j < 36$
 for the last block.
\ADFvfyParStart{(229, ((13, 180, ((180, 1), (36, 1), (12, 1), (1, 1))), (1, 45, ((180, 1), (36, 1), (12, 1), (1, 1))), (1, 36, ((180, 1), (36, 1), (12, 1), (1, 1)))), ((9, 20), (49, 1)))} 

\adfDgap
\noindent{\boldmath $ 9^{28} 1^{1} $}~
With the point set $\{0, 1, \dots, 252\}$ partitioned into
 residue classes modulo $28$ for $\{0, 1, \dots, 251\}$, and
 $\{252\}$,
 the design is generated from

\adfLgap {\adfBfont
$\{126, 140, 248, 179, 153\}$,
$\{81, 244, 212, 152, 234\}$,
$\{204, 146, 58, 65, 100\}$,\adfsplit
$\{6, 234, 73, 37, 22\}$,
$\{168, 220, 135, 215, 211\}$,
$\{0, 1, 18, 151, 174\}$,\adfsplit
$\{0, 6, 54, 103, 186\}$,
$\{0, 8, 20, 118, 143\}$,
$\{0, 2, 61, 64, 138\}$,\adfsplit
$\{0, 11, 41, 111, 218\}$,
$\{0, 21, 50, 94, 187\}$,
$\{0, 19, 57, 147, 184\}$,\adfsplit
$\{0, 63, 126, 189, 252\}$

}
\adfLgap \noindent by the mapping:
$x \mapsto x +  j \adfmod{252}$ for $x < 252$,
$252 \mapsto 252$,
$0 \le j < 252$
 for the first 12 blocks,
$0 \le j < 63$
 for the last block.
\ADFvfyParStart{(253, ((12, 252, ((252, 1), (1, 1))), (1, 63, ((252, 1), (1, 1)))), ((9, 28), (1, 1)))} 

\adfDgap
\noindent{\boldmath $ 10^{10} 18^{1} $}~
With the point set $\{0, 1, \dots, 117\}$ partitioned into
 residue classes modulo $9$ for $\{0, 1, \dots, 89\}$,
 $\{90, 91, \dots, 99\}$, and
 $\{100, 101, \dots, 117\}$,
 the design is generated from

\adfLgap {\adfBfont
$\{23, 19, 29, 81, 16\}$,
$\{24, 20, 30, 82, 17\}$,
$\{53, 101, 36, 48, 99\}$,\adfsplit
$\{54, 102, 37, 49, 90\}$,
$\{109, 48, 4, 64, 15\}$,
$\{0, 8, 43, 90, 111\}$,\adfsplit
$\{0, 26, 59, 75, 94\}$,
$\{1, 15, 35, 61, 115\}$,
$\{0, 21, 22, 61, 113\}$,\adfsplit
$\{0, 1, 20, 76, 114\}$,
$\{0, 2, 50, 69, 100\}$,
$\{0, 23, 31, 66, 95\}$,\adfsplit
$\{0, 37, 79, 97, 106\}$,
$\{0, 29, 51, 53, 108\}$

}
\adfLgap \noindent by the mapping:
$x \mapsto x + 2 j \adfmod{90}$ for $x < 90$,
$x \mapsto (x + 2 j \adfmod{10}) + 90$ for $90 \le x < 100$,
$x \mapsto (x - 100 + 2 j \adfmod{18}) + 100$ for $x \ge 100$,
$0 \le j < 45$.
\ADFvfyParStart{(118, ((14, 45, ((90, 2), (10, 2), (18, 2)))), ((10, 9), (10, 1), (18, 1)))} 

\adfDgap
\noindent{\boldmath $ 10^{20} 38^{1} $}~
With the point set $\{0, 1, \dots, 237\}$ partitioned into
 residue classes modulo $19$ for $\{0, 1, \dots, 189\}$,
 $\{190, 191, \dots, 199\}$, and
 $\{200, 201, \dots, 237\}$,
 the design is generated from

\adfLgap {\adfBfont
$\{150, 141, 128, 72, 168\}$,
$\{108, 150, 9, 191, 134\}$,
$\{236, 60, 106, 188, 143\}$,\adfsplit
$\{71, 174, 206, 36, 8\}$,
$\{82, 155, 154, 187, 232\}$,
$\{118, 200, 149, 63, 74\}$,\adfsplit
$\{0, 2, 5, 142, 178\}$,
$\{0, 23, 77, 120, 205\}$,
$\{0, 8, 74, 198, 225\}$,\adfsplit
$\{0, 6, 21, 132, 201\}$,
$\{0, 7, 68, 156, 216\}$,
$\{0, 25, 92, 131, 210\}$,\adfsplit
$\{0, 4, 143, 199, 219\}$,
$\{0, 10, 30, 90, 119\}$

}
\adfLgap \noindent by the mapping:
$x \mapsto x +  j \adfmod{190}$ for $x < 190$,
$x \mapsto (x +  j \adfmod{10}) + 190$ for $190 \le x < 200$,
$x \mapsto (x - 200 +  j \adfmod{38}) + 200$ for $x \ge 200$,
$0 \le j < 190$.
\ADFvfyParStart{(238, ((14, 190, ((190, 1), (10, 1), (38, 1)))), ((10, 19), (10, 1), (38, 1)))} 

\adfDgap
\noindent{\boldmath $ 11^{20} 19^{1} $}~
With the point set $\{0, 1, \dots, 238\}$ partitioned into
 residue classes modulo $19$ for $\{0, 1, \dots, 208\}$,
 $\{209, 210, \dots, 219\}$, and
 $\{220, 221, \dots, 238\}$,
 the design is generated from

\adfLgap {\adfBfont
$\{200, 139, 116, 210, 225\}$,
$\{237, 31, 127, 141, 191\}$,
$\{94, 164, 203, 0, 35\}$,\adfsplit
$\{97, 113, 215, 101, 10\}$,
$\{111, 131, 162, 140, 167\}$,
$\{220, 89, 135, 122, 37\}$,\adfsplit
$\{0, 10, 25, 166, 210\}$,
$\{0, 7, 67, 104, 130\}$,
$\{0, 8, 40, 128, 222\}$,\adfsplit
$\{0, 21, 92, 175, 232\}$,
$\{0, 1, 18, 66, 108\}$,
$\{0, 11, 69, 93, 147\}$,\adfsplit
$\{0, 2, 30, 74, 77\}$

}
\adfLgap \noindent by the mapping:
$x \mapsto x +  j \adfmod{209}$ for $x < 209$,
$x \mapsto (x +  j \adfmod{11}) + 209$ for $209 \le x < 220$,
$x \mapsto (x - 220 +  j \adfmod{19}) + 220$ for $x \ge 220$,
$0 \le j < 209$.
\ADFvfyParStart{(239, ((13, 209, ((209, 1), (11, 1), (19, 1)))), ((11, 19), (11, 1), (19, 1)))} 

\adfDgap
\noindent{\boldmath $ 12^{5} 8^{1} $}~
With the point set $\{0, 1, \dots, 67\}$ partitioned into
 residue classes modulo $5$ for $\{0, 1, \dots, 59\}$, and
 $\{60, 61, \dots, 67\}$,
 the design is generated from

\adfLgap {\adfBfont
$\{0, 2, 49, 51, 64\}$,
$\{0, 1, 7, 33, 59\}$,
$\{0, 4, 38, 41, 57\}$,\adfsplit
$\{0, 18, 19, 26, 32\}$,
$\{0, 9, 13, 31, 62\}$,
$\{0, 3, 6, 17, 65\}$,\adfsplit
$\{0, 8, 22, 29, 60\}$,
$\{0, 11, 42, 58, 61\}$,
$\{1, 18, 22, 39, 55\}$,\adfsplit
$\{1, 2, 30, 53, 66\}$,
$\{0, 16, 39, 43, 67\}$,
$\{1, 15, 34, 43, 61\}$,\adfsplit
$\{0, 12, 24, 36, 48\}$

}
\adfLgap \noindent by the mapping:
$x \mapsto x + 4 j \adfmod{60}$ for $x < 60$,
$x \mapsto (x +  j \adfmod{5}) + 60$ for $60 \le x < 65$,
$x \mapsto (x - 65 +  j \adfmod{3}) + 65$ for $x \ge 65$,
$0 \le j < 15$
 for the first 12 blocks;
$x \mapsto x +  j \adfmod{60}$ for $x < 60$,
$x \mapsto (x +  j \adfmod{5}) + 60$ for $60 \le x < 65$,
$x \mapsto (x - 65 +  j \adfmod{3}) + 65$ for $x \ge 65$,
$0 \le j < 12$
 for the last block.
\ADFvfyParStart{(68, ((12, 15, ((60, 4), (5, 1), (3, 1))), (1, 12, ((60, 1), (5, 1), (3, 1)))), ((12, 5), (8, 1)))} 

\adfDgap
\noindent{\boldmath $ 12^{10} 8^{1} $}~
With the point set $\{0, 1, \dots, 127\}$ partitioned into
 residue classes modulo $10$ for $\{0, 1, \dots, 119\}$, and
 $\{120, 121, \dots, 127\}$,
 the design is generated from

\adfLgap {\adfBfont
$\{21, 125, 79, 46, 64\}$,
$\{105, 124, 66, 107, 112\}$,
$\{0, 8, 19, 36, 107\}$,\adfsplit
$\{0, 1, 23, 27, 86\}$,
$\{0, 3, 12, 54, 68\}$,
$\{0, 6, 37, 53, 82\}$,\adfsplit
$\{0, 24, 48, 72, 96\}$

}
\adfLgap \noindent by the mapping:
$x \mapsto x +  j \adfmod{120}$ for $x < 120$,
$x \mapsto (x +  j \adfmod{8}) + 120$ for $x \ge 120$,
$0 \le j < 120$
 for the first six blocks,
$0 \le j < 24$
 for the last block.
\ADFvfyParStart{(128, ((6, 120, ((120, 1), (8, 1))), (1, 24, ((120, 1), (8, 1)))), ((12, 10), (8, 1)))} 

\adfDgap
\noindent{\boldmath $ 12^{10} 16^{1} $}~
With the point set $\{0, 1, \dots, 135\}$ partitioned into
 residue classes modulo $10$ for $\{0, 1, \dots, 119\}$, and
 $\{120, 121, \dots, 135\}$,
 the design is generated from

\adfLgap {\adfBfont
$\{49, 118, 74, 134, 37\}$,
$\{69, 129, 56, 83, 65\}$,
$\{0, 5, 28, 73, 120\}$,\adfsplit
$\{0, 2, 36, 101, 125\}$,
$\{0, 3, 32, 38, 49\}$,
$\{0, 1, 43, 59, 67\}$,\adfsplit
$\{0, 7, 22, 63, 94\}$

}
\adfLgap \noindent by the mapping:
$x \mapsto x +  j \adfmod{120}$ for $x < 120$,
$x \mapsto (x - 120 + 2 j \adfmod{16}) + 120$ for $x \ge 120$,
$0 \le j < 120$.
\ADFvfyParStart{(136, ((7, 120, ((120, 1), (16, 2)))), ((12, 10), (16, 1)))} 

\adfDgap
\noindent{\boldmath $ 12^{10} 28^{1} $}~
With the point set $\{0, 1, \dots, 147\}$ partitioned into
 residue classes modulo $10$ for $\{0, 1, \dots, 119\}$, and
 $\{120, 121, \dots, 147\}$,
 the design is generated from

\adfLgap {\adfBfont
$\{119, 25, 77, 78, 84\}$,
$\{109, 138, 92, 90, 34\}$,
$\{0, 3, 25, 34, 144\}$,\adfsplit
$\{0, 5, 18, 92, 139\}$,
$\{0, 12, 69, 83, 138\}$,
$\{0, 8, 23, 99, 132\}$,\adfsplit
$\{0, 4, 36, 47, 135\}$,
$\{0, 16, 54, 81, 122\}$,
$\{0, 24, 48, 72, 96\}$

}
\adfLgap \noindent by the mapping:
$x \mapsto x +  j \adfmod{120}$ for $x < 120$,
$x \mapsto (x +  j \adfmod{24}) + 120$ for $120 \le x < 144$,
$x \mapsto (x +  j \adfmod{4}) + 144$ for $x \ge 144$,
$0 \le j < 120$
 for the first eight blocks,
$0 \le j < 24$
 for the last block.
\ADFvfyParStart{(148, ((8, 120, ((120, 1), (24, 1), (4, 1))), (1, 24, ((120, 1), (24, 1), (4, 1)))), ((12, 10), (28, 1)))} 

\adfDgap
\noindent{\boldmath $ 12^{11} 20^{1} $}~
With the point set $\{0, 1, \dots, 151\}$ partitioned into
 residue classes modulo $11$ for $\{0, 1, \dots, 131\}$, and
 $\{132, 133, \dots, 151\}$,
 the design is generated from

\adfLgap {\adfBfont
$\{17, 144, 63, 128, 86\}$,
$\{34, 151, 0, 113, 107\}$,
$\{0, 1, 51, 58, 132\}$,\adfsplit
$\{0, 5, 14, 119, 135\}$,
$\{0, 15, 45, 62, 133\}$,
$\{0, 3, 38, 64, 92\}$,\adfsplit
$\{0, 2, 12, 95, 103\}$,
$\{0, 4, 24, 84, 100\}$

}
\adfLgap \noindent by the mapping:
$x \mapsto x +  j \adfmod{132}$ for $x < 132$,
$x \mapsto (x - 132 + 5 j \adfmod{20}) + 132$ for $x \ge 132$,
$0 \le j < 132$.
\ADFvfyParStart{(152, ((8, 132, ((132, 1), (20, 5)))), ((12, 11), (20, 1)))} 

\adfDgap
\noindent{\boldmath $ 12^{12} 4^{1} $}~
With the point set $\{0, 1, \dots, 147\}$ partitioned into
 residue classes modulo $12$ for $\{0, 1, \dots, 143\}$, and
 $\{144, 145, 146, 147\}$,
 the design is generated from

\adfLgap {\adfBfont
$\{145, 26, 49, 15, 12\}$,
$\{104, 62, 16, 77, 97\}$,
$\{0, 1, 55, 65, 71\}$,\adfsplit
$\{0, 4, 25, 33, 91\}$,
$\{0, 9, 28, 50, 68\}$,
$\{0, 2, 32, 49, 101\}$,\adfsplit
$\{0, 5, 31, 44, 82\}$

}
\adfLgap \noindent by the mapping:
$x \mapsto x +  j \adfmod{144}$ for $x < 144$,
$x \mapsto (x +  j \adfmod{4}) + 144$ for $x \ge 144$,
$0 \le j < 144$.
\ADFvfyParStart{(148, ((7, 144, ((144, 1), (4, 1)))), ((12, 12), (4, 1)))} 

\adfDgap
\noindent{\boldmath $ 12^{12} 24^{1} $}~
With the point set $\{0, 1, \dots, 167\}$ partitioned into
 residue classes modulo $12$ for $\{0, 1, \dots, 143\}$, and
 $\{144, 145, \dots, 167\}$,
 the design is generated from

\adfLgap {\adfBfont
$\{89, 110, 157, 118, 87\}$,
$\{154, 46, 61, 122, 66\}$,
$\{113, 33, 47, 162, 79\}$,\adfsplit
$\{0, 1, 101, 134, 151\}$,
$\{0, 19, 71, 109, 157\}$,
$\{0, 17, 70, 119, 147\}$,\adfsplit
$\{0, 6, 13, 22, 63\}$,
$\{0, 4, 30, 69, 97\}$,
$\{0, 3, 40, 58, 85\}$

}
\adfLgap \noindent by the mapping:
$x \mapsto x +  j \adfmod{144}$ for $x < 144$,
$x \mapsto (x +  j \adfmod{24}) + 144$ for $x \ge 144$,
$0 \le j < 144$.
\ADFvfyParStart{(168, ((9, 144, ((144, 1), (24, 1)))), ((12, 12), (24, 1)))} 

\adfDgap
\noindent{\boldmath $ 12^{13} 8^{1} $}~
With the point set $\{0, 1, \dots, 163\}$ partitioned into
 residue classes modulo $13$ for $\{0, 1, \dots, 155\}$, and
 $\{156, 157, \dots, 163\}$,
 the design is generated from

\adfLgap {\adfBfont
$\{17, 151, 107, 23, 126\}$,
$\{105, 87, 90, 94, 151\}$,
$\{0, 1, 34, 43, 157\}$,\adfsplit
$\{0, 23, 54, 81, 156\}$,
$\{0, 2, 40, 76, 108\}$,
$\{0, 5, 21, 35, 94\}$,\adfsplit
$\{0, 10, 55, 79, 96\}$,
$\{0, 8, 20, 71, 127\}$

}
\adfLgap \noindent by the mapping:
$x \mapsto x +  j \adfmod{156}$ for $x < 156$,
$x \mapsto (x - 156 + 2 j \adfmod{8}) + 156$ for $x \ge 156$,
$0 \le j < 156$.
\ADFvfyParStart{(164, ((8, 156, ((156, 1), (8, 2)))), ((12, 13), (8, 1)))} 

\adfDgap
\noindent{\boldmath $ 12^{13} 28^{1} $}~
With the point set $\{0, 1, \dots, 183\}$ partitioned into
 residue classes modulo $13$ for $\{0, 1, \dots, 155\}$, and
 $\{156, 157, \dots, 183\}$,
 the design is generated from

\adfLgap {\adfBfont
$\{56, 119, 9, 42, 110\}$,
$\{16, 152, 65, 60, 166\}$,
$\{73, 164, 30, 104, 62\}$,\adfsplit
$\{0, 1, 128, 135, 170\}$,
$\{0, 6, 25, 111, 180\}$,
$\{0, 2, 12, 96, 120\}$,\adfsplit
$\{0, 4, 27, 61, 165\}$,
$\{0, 17, 35, 76, 106\}$,
$\{0, 3, 103, 119, 175\}$,\adfsplit
$\{0, 8, 83, 98, 171\}$

}
\adfLgap \noindent by the mapping:
$x \mapsto x +  j \adfmod{156}$ for $x < 156$,
$x \mapsto (x - 156 + 2 j \adfmod{24}) + 156$ for $156 \le x < 180$,
$x \mapsto (x +  j \adfmod{4}) + 180$ for $x \ge 180$,
$0 \le j < 156$.
\ADFvfyParStart{(184, ((10, 156, ((156, 1), (24, 2), (4, 1)))), ((12, 13), (28, 1)))} 

\adfDgap
\noindent{\boldmath $ 12^{14} 32^{1} $}~
With the point set $\{0, 1, \dots, 199\}$ partitioned into
 residue classes modulo $14$ for $\{0, 1, \dots, 167\}$, and
 $\{168, 169, \dots, 199\}$,
 the design is generated from

\adfLgap {\adfBfont
$\{189, 166, 52, 157, 107\}$,
$\{148, 4, 140, 115, 100\}$,
$\{84, 95, 189, 111, 20\}$,\adfsplit
$\{115, 117, 41, 138, 80\}$,
$\{0, 1, 31, 117, 180\}$,
$\{0, 5, 41, 67, 189\}$,\adfsplit
$\{0, 12, 34, 80, 99\}$,
$\{0, 6, 44, 89, 188\}$,
$\{0, 18, 61, 108, 175\}$,\adfsplit
$\{0, 20, 49, 115, 196\}$,
$\{0, 3, 7, 158, 197\}$

}
\adfLgap \noindent by the mapping:
$x \mapsto x +  j \adfmod{168}$ for $x < 168$,
$x \mapsto (x +  j \adfmod{24}) + 168$ for $168 \le x < 192$,
$x \mapsto (x +  j \adfmod{8}) + 192$ for $x \ge 192$,
$0 \le j < 168$.
\ADFvfyParStart{(200, ((11, 168, ((168, 1), (24, 1), (8, 1)))), ((12, 14), (32, 1)))} 

\adfDgap
\noindent{\boldmath $ 12^{15} 8^{1} $}~
With the point set $\{0, 1, \dots, 187\}$ partitioned into
 residue classes modulo $15$ for $\{0, 1, \dots, 179\}$, and
 $\{180, 181, \dots, 187\}$,
 the design is generated from

\adfLgap {\adfBfont
$\{175, 18, 106, 11, 128\}$,
$\{67, 66, 149, 153, 94\}$,
$\{0, 2, 5, 11, 181\}$,\adfsplit
$\{0, 13, 62, 127, 180\}$,
$\{0, 19, 44, 100, 151\}$,
$\{0, 8, 18, 39, 148\}$,\adfsplit
$\{0, 12, 26, 64, 138\}$,
$\{0, 24, 57, 103, 146\}$,
$\{0, 17, 37, 78, 113\}$,\adfsplit
$\{0, 36, 72, 108, 144\}$

}
\adfLgap \noindent by the mapping:
$x \mapsto x +  j \adfmod{180}$ for $x < 180$,
$x \mapsto (x - 180 + 2 j \adfmod{8}) + 180$ for $x \ge 180$,
$0 \le j < 180$
 for the first nine blocks,
$0 \le j < 36$
 for the last block.
\ADFvfyParStart{(188, ((9, 180, ((180, 1), (8, 2))), (1, 36, ((180, 1), (8, 2)))), ((12, 15), (8, 1)))} 

\adfDgap
\noindent{\boldmath $ 12^{15} 16^{1} $}~
With the point set $\{0, 1, \dots, 195\}$ partitioned into
 residue classes modulo $15$ for $\{0, 1, \dots, 179\}$, and
 $\{180, 181, \dots, 195\}$,
 the design is generated from

\adfLgap {\adfBfont
$\{59, 19, 60, 172, 189\}$,
$\{193, 133, 171, 140, 54\}$,
$\{9, 133, 18, 15, 141\}$,\adfsplit
$\{72, 150, 179, 91, 187\}$,
$\{0, 10, 109, 137, 184\}$,
$\{0, 5, 47, 111, 134\}$,\adfsplit
$\{0, 11, 33, 77, 91\}$,
$\{0, 4, 20, 145, 163\}$,
$\{0, 2, 26, 98, 130\}$,\adfsplit
$\{0, 12, 25, 61, 95\}$

}
\adfLgap \noindent by the mapping:
$x \mapsto x +  j \adfmod{180}$ for $x < 180$,
$x \mapsto (x +  j \adfmod{12}) + 180$ for $180 \le x < 192$,
$x \mapsto (x +  j \adfmod{4}) + 192$ for $x \ge 192$,
$0 \le j < 180$.
\ADFvfyParStart{(196, ((10, 180, ((180, 1), (12, 1), (4, 1)))), ((12, 15), (16, 1)))} 

\adfDgap
\noindent{\boldmath $ 12^{15} 28^{1} $}~
With the point set $\{0, 1, \dots, 207\}$ partitioned into
 residue classes modulo $15$ for $\{0, 1, \dots, 179\}$, and
 $\{180, 181, \dots, 207\}$,
 the design is generated from

\adfLgap {\adfBfont
$\{203, 2, 113, 132, 103\}$,
$\{8, 15, 188, 122, 161\}$,
$\{170, 161, 200, 43, 48\}$,\adfsplit
$\{29, 205, 116, 67, 46\}$,
$\{0, 1, 3, 98, 183\}$,
$\{0, 13, 54, 115, 185\}$,\adfsplit
$\{0, 23, 74, 117, 180\}$,
$\{0, 6, 31, 140, 154\}$,
$\{0, 4, 22, 125, 160\}$,\adfsplit
$\{0, 8, 64, 76, 92\}$,
$\{0, 11, 44, 91, 143\}$,
$\{0, 36, 72, 108, 144\}$

}
\adfLgap \noindent by the mapping:
$x \mapsto x +  j \adfmod{180}$ for $x < 180$,
$x \mapsto (x - 180 + 7 j \adfmod{28}) + 180$ for $x \ge 180$,
$0 \le j < 180$
 for the first 11 blocks,
$0 \le j < 36$
 for the last block.
\ADFvfyParStart{(208, ((11, 180, ((180, 1), (28, 7))), (1, 36, ((180, 1), (28, 7)))), ((12, 15), (28, 1)))} 

\adfDgap
\noindent{\boldmath $ 12^{15} 36^{1} $}~
With the point set $\{0, 1, \dots, 215\}$ partitioned into
 residue classes modulo $15$ for $\{0, 1, \dots, 179\}$, and
 $\{180, 181, \dots, 215\}$,
 the design is generated from

\adfLgap {\adfBfont
$\{197, 72, 69, 70, 47\}$,
$\{160, 173, 78, 127, 71\}$,
$\{210, 65, 105, 129, 36\}$,\adfsplit
$\{141, 68, 200, 60, 136\}$,
$\{8, 174, 181, 42, 73\}$,
$\{0, 11, 62, 83, 125\}$,\adfsplit
$\{0, 17, 36, 54, 74\}$,
$\{0, 6, 58, 67, 188\}$,
$\{0, 10, 80, 164, 215\}$,\adfsplit
$\{0, 4, 43, 152, 198\}$,
$\{0, 27, 77, 121, 183\}$,
$\{0, 12, 47, 139, 196\}$

}
\adfLgap \noindent by the mapping:
$x \mapsto x +  j \adfmod{180}$ for $x < 180$,
$x \mapsto (x +  j \adfmod{36}) + 180$ for $x \ge 180$,
$0 \le j < 180$.
\ADFvfyParStart{(216, ((12, 180, ((180, 1), (36, 1)))), ((12, 15), (36, 1)))} 

\adfDgap
\noindent{\boldmath $ 12^{15} 48^{1} $}~
With the point set $\{0, 1, \dots, 227\}$ partitioned into
 residue classes modulo $15$ for $\{0, 1, \dots, 179\}$, and
 $\{180, 181, \dots, 227\}$,
 the design is generated from

\adfLgap {\adfBfont
$\{86, 149, 8, 196, 54\}$,
$\{164, 96, 77, 148, 221\}$,
$\{79, 123, 56, 190, 39\}$,\adfsplit
$\{157, 219, 132, 153, 160\}$,
$\{124, 137, 115, 5, 202\}$,
$\{11, 38, 92, 222, 58\}$,\adfsplit
$\{0, 1, 104, 122, 185\}$,
$\{0, 29, 64, 143, 192\}$,
$\{0, 11, 49, 73, 138\}$,\adfsplit
$\{0, 5, 55, 129, 181\}$,
$\{0, 6, 86, 98, 180\}$,
$\{0, 14, 57, 83, 215\}$,\adfsplit
$\{0, 2, 33, 172, 205\}$,
$\{0, 36, 72, 108, 144\}$

}
\adfLgap \noindent by the mapping:
$x \mapsto x +  j \adfmod{180}$ for $x < 180$,
$x \mapsto (x +  j \adfmod{36}) + 180$ for $180 \le x < 216$,
$x \mapsto (x +  j \adfmod{12}) + 216$ for $x \ge 216$,
$0 \le j < 180$
 for the first 13 blocks,
$0 \le j < 36$
 for the last block.
\ADFvfyParStart{(228, ((13, 180, ((180, 1), (36, 1), (12, 1))), (1, 36, ((180, 1), (36, 1), (12, 1)))), ((12, 15), (48, 1)))} 

\adfDgap
\noindent{\boldmath $ 12^{16} 20^{1} $}~
With the point set $\{0, 1, \dots, 211\}$ partitioned into
 residue classes modulo $16$ for $\{0, 1, \dots, 191\}$, and
 $\{192, 193, \dots, 211\}$,
 the design is generated from

\adfLgap {\adfBfont
$\{193, 106, 109, 79, 37\}$,
$\{7, 209, 109, 96, 134\}$,
$\{45, 34, 171, 163, 88\}$,\adfsplit
$\{204, 35, 87, 182, 66\}$,
$\{0, 1, 34, 36, 207\}$,
$\{0, 37, 87, 136, 200\}$,\adfsplit
$\{0, 12, 40, 59, 151\}$,
$\{0, 7, 17, 77, 148\}$,
$\{0, 5, 67, 106, 124\}$,\adfsplit
$\{0, 4, 24, 82, 170\}$,
$\{0, 6, 15, 29, 113\}$

}
\adfLgap \noindent by the mapping:
$x \mapsto x +  j \adfmod{192}$ for $x < 192$,
$x \mapsto (x +  j \adfmod{16}) + 192$ for $192 \le x < 208$,
$x \mapsto (x +  j \adfmod{4}) + 208$ for $x \ge 208$,
$0 \le j < 192$.
\ADFvfyParStart{(212, ((11, 192, ((192, 1), (16, 1), (4, 1)))), ((12, 16), (20, 1)))} 

\adfDgap
\noindent{\boldmath $ 12^{16} 40^{1} $}~
With the point set $\{0, 1, \dots, 231\}$ partitioned into
 residue classes modulo $16$ for $\{0, 1, \dots, 191\}$, and
 $\{192, 193, \dots, 231\}$,
 the design is generated from

\adfLgap {\adfBfont
$\{11, 53, 203, 121, 42\}$,
$\{2, 228, 36, 150, 131\}$,
$\{95, 103, 90, 164, 83\}$,\adfsplit
$\{63, 206, 101, 3, 91\}$,
$\{153, 96, 215, 180, 139\}$,
$\{152, 87, 153, 223, 15\}$,\adfsplit
$\{0, 6, 153, 162, 229\}$,
$\{0, 17, 109, 134, 202\}$,
$\{0, 15, 37, 77, 136\}$,\adfsplit
$\{0, 2, 103, 107, 220\}$,
$\{0, 18, 53, 159, 205\}$,
$\{0, 3, 49, 70, 197\}$,\adfsplit
$\{0, 23, 47, 73, 163\}$

}
\adfLgap \noindent by the mapping:
$x \mapsto x +  j \adfmod{192}$ for $x < 192$,
$x \mapsto (x +  j \adfmod{32}) + 192$ for $192 \le x < 224$,
$x \mapsto (x +  j \adfmod{8}) + 224$ for $x \ge 224$,
$0 \le j < 192$.
\ADFvfyParStart{(232, ((13, 192, ((192, 1), (32, 1), (8, 1)))), ((12, 16), (40, 1)))} 

\adfDgap
\noindent{\boldmath $ 12^{17} 4^{1} $}~
With the point set $\{0, 1, \dots, 207\}$ partitioned into
 residue classes modulo $17$ for $\{0, 1, \dots, 203\}$, and
 $\{204, 205, 206, 207\}$,
 the design is generated from

\adfLgap {\adfBfont
$\{136, 64, 203, 52, 12\}$,
$\{81, 42, 4, 11, 167\}$,
$\{152, 107, 106, 103, 133\}$,\adfsplit
$\{200, 59, 204, 150, 5\}$,
$\{0, 15, 35, 71, 161\}$,
$\{0, 11, 25, 98, 160\}$,\adfsplit
$\{0, 16, 37, 112, 144\}$,
$\{0, 2, 10, 111, 140\}$,
$\{0, 6, 28, 89, 122\}$,\adfsplit
$\{0, 5, 23, 47, 104\}$

}
\adfLgap \noindent by the mapping:
$x \mapsto x +  j \adfmod{204}$ for $x < 204$,
$x \mapsto (x +  j \adfmod{4}) + 204$ for $x \ge 204$,
$0 \le j < 204$.
\ADFvfyParStart{(208, ((10, 204, ((204, 1), (4, 1)))), ((12, 17), (4, 1)))} 

\adfDgap
\noindent{\boldmath $ 12^{17} 24^{1} $}~
With the point set $\{0, 1, \dots, 227\}$ partitioned into
 residue classes modulo $17$ for $\{0, 1, \dots, 203\}$, and
 $\{204, 205, \dots, 227\}$,
 the design is generated from

\adfLgap {\adfBfont
$\{188, 218, 134, 75, 37\}$,
$\{58, 221, 176, 62, 192\}$,
$\{93, 13, 133, 165, 102\}$,\adfsplit
$\{149, 66, 113, 188, 87\}$,
$\{91, 179, 62, 18, 67\}$,
$\{0, 13, 48, 94, 108\}$,\adfsplit
$\{0, 2, 12, 79, 140\}$,
$\{0, 1, 19, 56, 210\}$,
$\{0, 3, 23, 182, 204\}$,\adfsplit
$\{0, 11, 146, 174, 211\}$,
$\{0, 8, 50, 65, 219\}$,
$\{0, 6, 33, 104, 111\}$

}
\adfLgap \noindent by the mapping:
$x \mapsto x +  j \adfmod{204}$ for $x < 204$,
$x \mapsto (x - 204 + 2 j \adfmod{24}) + 204$ for $x \ge 204$,
$0 \le j < 204$.
\ADFvfyParStart{(228, ((12, 204, ((204, 1), (24, 2)))), ((12, 17), (24, 1)))} 

\adfDgap
\noindent{\boldmath $ 12^{17} 44^{1} $}~
With the point set $\{0, 1, \dots, 247\}$ partitioned into
 residue classes modulo $17$ for $\{0, 1, \dots, 203\}$, and
 $\{204, 205, \dots, 247\}$,
 the design is generated from

\adfLgap {\adfBfont
$\{119, 149, 226, 160, 141\}$,
$\{208, 52, 201, 121, 163\}$,
$\{137, 237, 133, 162, 84\}$,\adfsplit
$\{187, 133, 229, 69, 75\}$,
$\{82, 98, 19, 1, 203\}$,
$\{148, 225, 19, 180, 89\}$,\adfsplit
$\{0, 5, 45, 76, 207\}$,
$\{0, 9, 66, 139, 240\}$,
$\{0, 13, 27, 50, 241\}$,\adfsplit
$\{0, 3, 15, 39, 87\}$,
$\{0, 1, 61, 89, 96\}$,
$\{0, 44, 90, 142, 206\}$,\adfsplit
$\{0, 21, 47, 103, 230\}$,
$\{0, 10, 104, 137, 215\}$

}
\adfLgap \noindent by the mapping:
$x \mapsto x +  j \adfmod{204}$ for $x < 204$,
$x \mapsto (x - 204 + 3 j \adfmod{36}) + 204$ for $204 \le x < 240$,
$x \mapsto (x + 2 j \adfmod{8}) + 240$ for $x \ge 240$,
$0 \le j < 204$.
\ADFvfyParStart{(248, ((14, 204, ((204, 1), (36, 3), (8, 2)))), ((12, 17), (44, 1)))} 

\adfDgap
\noindent{\boldmath $ 12^{18} 8^{1} $}~
With the point set $\{0, 1, \dots, 223\}$ partitioned into
 residue classes modulo $18$ for $\{0, 1, \dots, 215\}$, and
 $\{216, 217, \dots, 223\}$,
 the design is generated from

\adfLgap {\adfBfont
$\{46, 40, 182, 140, 125\}$,
$\{222, 76, 37, 62, 74\}$,
$\{73, 76, 175, 131, 84\}$,\adfsplit
$\{62, 89, 58, 63, 108\}$,
$\{0, 7, 66, 206, 221\}$,
$\{0, 13, 81, 119, 141\}$,\adfsplit
$\{0, 23, 51, 118, 183\}$,
$\{0, 9, 61, 96, 139\}$,
$\{0, 21, 53, 124, 154\}$,\adfsplit
$\{0, 29, 63, 111, 175\}$,
$\{0, 16, 40, 89, 109\}$

}
\adfLgap \noindent by the mapping:
$x \mapsto x +  j \adfmod{216}$ for $x < 216$,
$x \mapsto (x +  j \adfmod{8}) + 216$ for $x \ge 216$,
$0 \le j < 216$.
\ADFvfyParStart{(224, ((11, 216, ((216, 1), (8, 1)))), ((12, 18), (8, 1)))} 

\adfDgap
\noindent{\boldmath $ 12^{18} 28^{1} $}~
With the point set $\{0, 1, \dots, 243\}$ partitioned into
 residue classes modulo $18$ for $\{0, 1, \dots, 215\}$, and
 $\{216, 217, \dots, 243\}$,
 the design is generated from

\adfLgap {\adfBfont
$\{243, 11, 196, 81, 2\}$,
$\{57, 72, 83, 208, 206\}$,
$\{98, 140, 133, 117, 186\}$,\adfsplit
$\{64, 171, 39, 181, 26\}$,
$\{177, 115, 55, 82, 219\}$,
$\{0, 1, 4, 188, 196\}$,\adfsplit
$\{0, 17, 66, 114, 158\}$,
$\{0, 37, 76, 133, 176\}$,
$\{0, 5, 78, 165, 226\}$,\adfsplit
$\{0, 6, 47, 110, 222\}$,
$\{0, 30, 85, 135, 219\}$,
$\{0, 14, 59, 127, 225\}$,\adfsplit
$\{0, 12, 130, 164, 227\}$

}
\adfLgap \noindent by the mapping:
$x \mapsto x +  j \adfmod{216}$ for $x < 216$,
$x \mapsto (x +  j \adfmod{24}) + 216$ for $216 \le x < 240$,
$x \mapsto (x +  j \adfmod{4}) + 240$ for $x \ge 240$,
$0 \le j < 216$.
\ADFvfyParStart{(244, ((13, 216, ((216, 1), (24, 1), (4, 1)))), ((12, 18), (28, 1)))} 

\adfDgap
\noindent{\boldmath $ 12^{18} 48^{1} $}~
With the point set $\{0, 1, \dots, 263\}$ partitioned into
 residue classes modulo $18$ for $\{0, 1, \dots, 215\}$, and
 $\{216, 217, \dots, 263\}$,
 the design is generated from

\adfLgap {\adfBfont
$\{139, 181, 20, 168, 116\}$,
$\{165, 247, 43, 44, 83\}$,
$\{115, 257, 202, 17, 31\}$,\adfsplit
$\{67, 158, 216, 75, 72\}$,
$\{109, 229, 28, 102, 143\}$,
$\{246, 137, 86, 165, 118\}$,\adfsplit
$\{0, 30, 63, 143, 245\}$,
$\{0, 12, 100, 117, 235\}$,
$\{0, 6, 112, 179, 261\}$,\adfsplit
$\{0, 10, 21, 71, 178\}$,
$\{0, 2, 22, 78, 192\}$,
$\{0, 27, 85, 154, 220\}$,\adfsplit
$\{0, 15, 64, 139, 228\}$,
$\{0, 4, 70, 163, 230\}$,
$\{0, 9, 44, 200, 224\}$

}
\adfLgap \noindent by the mapping:
$x \mapsto x +  j \adfmod{216}$ for $x < 216$,
$x \mapsto (x - 216 + 2 j \adfmod{48}) + 216$ for $x \ge 216$,
$0 \le j < 216$.
\ADFvfyParStart{(264, ((15, 216, ((216, 1), (48, 2)))), ((12, 18), (48, 1)))} 

\adfDgap
\noindent{\boldmath $ 12^{19} 32^{1} $}~
With the point set $\{0, 1, \dots, 259\}$ partitioned into
 residue classes modulo $19$ for $\{0, 1, \dots, 227\}$, and
 $\{228, 229, \dots, 259\}$,
 the design is generated from

\adfLgap {\adfBfont
$\{185, 139, 30, 231, 216\}$,
$\{245, 50, 213, 48, 35\}$,
$\{198, 115, 190, 54, 151\}$,\adfsplit
$\{259, 99, 25, 70, 204\}$,
$\{184, 241, 99, 113, 122\}$,
$\{83, 32, 162, 161, 172\}$,\adfsplit
$\{0, 5, 27, 142, 232\}$,
$\{0, 17, 58, 175, 234\}$,
$\{0, 21, 55, 162, 230\}$,\adfsplit
$\{0, 26, 59, 185, 228\}$,
$\{0, 20, 64, 101, 168\}$,
$\{0, 6, 24, 54, 106\}$,\adfsplit
$\{0, 3, 28, 35, 138\}$,
$\{0, 4, 16, 72, 112\}$

}
\adfLgap \noindent by the mapping:
$x \mapsto x +  j \adfmod{228}$ for $x < 228$,
$x \mapsto (x - 228 + 8 j \adfmod{32}) + 228$ for $x \ge 228$,
$0 \le j < 228$.
\ADFvfyParStart{(260, ((14, 228, ((228, 1), (32, 8)))), ((12, 19), (32, 1)))} 

\adfDgap
\noindent{\boldmath $ 12^{19} 52^{1} $}~
With the point set $\{0, 1, \dots, 279\}$ partitioned into
 residue classes modulo $19$ for $\{0, 1, \dots, 227\}$, and
 $\{228, 229, \dots, 279\}$,
 the design is generated from

\adfLgap {\adfBfont
$\{250, 205, 227, 108, 200\}$,
$\{273, 92, 45, 175, 2\}$,
$\{220, 21, 239, 115, 12\}$,\adfsplit
$\{261, 7, 144, 207, 118\}$,
$\{82, 5, 3, 132, 240\}$,
$\{111, 272, 98, 31, 77\}$,\adfsplit
$\{0, 1, 7, 212, 257\}$,
$\{0, 3, 187, 217, 236\}$,
$\{0, 10, 61, 163, 276\}$,\adfsplit
$\{0, 12, 36, 71, 156\}$,
$\{0, 15, 60, 93, 147\}$,
$\{0, 4, 66, 179, 258\}$,\adfsplit
$\{0, 37, 106, 158, 230\}$,
$\{0, 8, 39, 64, 112\}$,
$\{0, 32, 74, 142, 275\}$,\adfsplit
$\{0, 18, 58, 146, 255\}$

}
\adfLgap \noindent by the mapping:
$x \mapsto x +  j \adfmod{228}$ for $x < 228$,
$x \mapsto (x - 228 + 4 j \adfmod{48}) + 228$ for $228 \le x < 276$,
$x \mapsto (x +  j \adfmod{4}) + 276$ for $x \ge 276$,
$0 \le j < 228$.
\ADFvfyParStart{(280, ((16, 228, ((228, 1), (48, 4), (4, 1)))), ((12, 19), (52, 1)))} 

\adfDgap
\noindent{\boldmath $ 12^{20} 8^{1} $}~
With the point set $\{0, 1, \dots, 247\}$ partitioned into
 residue classes modulo $20$ for $\{0, 1, \dots, 239\}$, and
 $\{240, 241, \dots, 247\}$,
 the design is generated from

\adfLgap {\adfBfont
$\{33, 245, 148, 94, 64\}$,
$\{7, 223, 176, 191, 134\}$,
$\{160, 219, 15, 56, 6\}$,\adfsplit
$\{16, 140, 28, 26, 3\}$,
$\{81, 145, 92, 7, 213\}$,
$\{0, 45, 110, 162, 240\}$,\adfsplit
$\{0, 16, 33, 91, 174\}$,
$\{0, 38, 87, 131, 177\}$,
$\{0, 7, 35, 129, 168\}$,\adfsplit
$\{0, 1, 4, 171, 222\}$,
$\{0, 6, 14, 43, 148\}$,
$\{0, 5, 26, 81, 178\}$,\adfsplit
$\{0, 48, 96, 144, 192\}$

}
\adfLgap \noindent by the mapping:
$x \mapsto x +  j \adfmod{240}$ for $x < 240$,
$x \mapsto (x +  j \adfmod{8}) + 240$ for $x \ge 240$,
$0 \le j < 240$
 for the first 12 blocks,
$0 \le j < 48$
 for the last block.
\ADFvfyParStart{(248, ((12, 240, ((240, 1), (8, 1))), (1, 48, ((240, 1), (8, 1)))), ((12, 20), (8, 1)))} 

\adfDgap
\noindent{\boldmath $ 12^{21} 20^{1} $}~
With the point set $\{0, 1, \dots, 271\}$ partitioned into
 residue classes modulo $21$ for $\{0, 1, \dots, 251\}$, and
 $\{252, 253, \dots, 271\}$,
 the design is generated from

\adfLgap {\adfBfont
$\{113, 94, 74, 217, 245\}$,
$\{215, 76, 108, 242, 143\}$,
$\{61, 249, 149, 258, 203\}$,\adfsplit
$\{61, 53, 240, 197, 38\}$,
$\{247, 241, 217, 47, 250\}$,
$\{100, 153, 23, 151, 83\}$,\adfsplit
$\{0, 1, 79, 90, 265\}$,
$\{0, 34, 71, 169, 264\}$,
$\{0, 7, 55, 102, 193\}$,\adfsplit
$\{0, 16, 38, 112, 141\}$,
$\{0, 5, 31, 41, 196\}$,
$\{0, 12, 25, 69, 131\}$,\adfsplit
$\{0, 40, 85, 177, 252\}$,
$\{0, 4, 18, 176, 262\}$

}
\adfLgap \noindent by the mapping:
$x \mapsto x +  j \adfmod{252}$ for $x < 252$,
$x \mapsto (x +  j \adfmod{12}) + 252$ for $252 \le x < 264$,
$x \mapsto (x + 2 j \adfmod{8}) + 264$ for $x \ge 264$,
$0 \le j < 252$.
\ADFvfyParStart{(272, ((14, 252, ((252, 1), (12, 1), (8, 2)))), ((12, 21), (20, 1)))} 

\adfDgap
\noindent{\boldmath $ 12^{21} 40^{1} $}~
With the point set $\{0, 1, \dots, 291\}$ partitioned into
 residue classes modulo $21$ for $\{0, 1, \dots, 251\}$, and
 $\{252, 253, \dots, 291\}$,
 the design is generated from

\adfLgap {\adfBfont
$\{142, 36, 282, 31, 124\}$,
$\{251, 168, 59, 58, 102\}$,
$\{3, 257, 163, 74, 214\}$,\adfsplit
$\{7, 177, 4, 56, 40\}$,
$\{22, 16, 161, 9, 200\}$,
$\{193, 283, 18, 94, 63\}$,\adfsplit
$\{156, 183, 280, 247, 191\}$,
$\{0, 2, 11, 25, 288\}$,
$\{0, 10, 34, 158, 190\}$,\adfsplit
$\{0, 19, 57, 127, 174\}$,
$\{0, 12, 67, 224, 283\}$,
$\{0, 30, 120, 166, 272\}$,\adfsplit
$\{0, 22, 75, 226, 267\}$,
$\{0, 15, 73, 102, 187\}$,
$\{0, 4, 133, 202, 262\}$,\adfsplit
$\{0, 17, 37, 171, 280\}$

}
\adfLgap \noindent by the mapping:
$x \mapsto x +  j \adfmod{252}$ for $x < 252$,
$x \mapsto (x +  j \adfmod{36}) + 252$ for $252 \le x < 288$,
$x \mapsto (x +  j \adfmod{4}) + 288$ for $x \ge 288$,
$0 \le j < 252$.
\ADFvfyParStart{(292, ((16, 252, ((252, 1), (36, 1), (4, 1)))), ((12, 21), (40, 1)))} 

\adfDgap
\noindent{\boldmath $ 12^{21} 60^{1} $}~
With the point set $\{0, 1, \dots, 311\}$ partitioned into
 residue classes modulo $21$ for $\{0, 1, \dots, 251\}$, and
 $\{252, 253, \dots, 311\}$,
 the design is generated from

\adfLgap {\adfBfont
$\{91, 300, 131, 149, 61\}$,
$\{135, 174, 260, 231, 203\}$,
$\{134, 291, 180, 213, 89\}$,\adfsplit
$\{84, 239, 52, 293, 101\}$,
$\{0, 119, 7, 120, 108\}$,
$\{53, 253, 30, 140, 178\}$,\adfsplit
$\{95, 154, 230, 262, 45\}$,
$\{269, 243, 94, 107, 189\}$,
$\{0, 2, 37, 43, 303\}$,\adfsplit
$\{0, 3, 51, 55, 75\}$,
$\{0, 9, 36, 83, 190\}$,
$\{0, 25, 85, 208, 271\}$,\adfsplit
$\{0, 16, 93, 146, 265\}$,
$\{0, 10, 90, 121, 274\}$,
$\{0, 19, 153, 179, 286\}$,\adfsplit
$\{0, 8, 22, 86, 276\}$,
$\{0, 5, 61, 163, 306\}$,
$\{0, 15, 152, 186, 300\}$

}
\adfLgap \noindent by the mapping:
$x \mapsto x +  j \adfmod{252}$ for $x < 252$,
$x \mapsto (x +  j \adfmod{36}) + 252$ for $252 \le x < 288$,
$x \mapsto (x + 2 j \adfmod{24}) + 288$ for $x \ge 288$,
$0 \le j < 252$.
\ADFvfyParStart{(312, ((18, 252, ((252, 1), (36, 1), (24, 2)))), ((12, 21), (60, 1)))} 

\adfDgap
\noindent{\boldmath $ 12^{22} 4^{1} $}~
With the point set $\{0, 1, \dots, 267\}$ partitioned into
 residue classes modulo $22$ for $\{0, 1, \dots, 263\}$, and
 $\{264, 265, 266, 267\}$,
 the design is generated from

\adfLgap {\adfBfont
$\{210, 229, 3, 186, 158\}$,
$\{8, 3, 149, 242, 134\}$,
$\{122, 200, 20, 223, 164\}$,\adfsplit
$\{187, 235, 101, 95, 238\}$,
$\{108, 76, 155, 126, 50\}$,
$\{0, 1, 63, 70, 264\}$,\adfsplit
$\{0, 2, 11, 106, 254\}$,
$\{0, 8, 49, 83, 174\}$,
$\{0, 4, 35, 100, 117\}$,\adfsplit
$\{0, 27, 72, 112, 149\}$,
$\{0, 13, 33, 107, 210\}$,
$\{0, 14, 53, 150, 218\}$,\adfsplit
$\{0, 16, 80, 135, 191\}$

}
\adfLgap \noindent by the mapping:
$x \mapsto x +  j \adfmod{264}$ for $x < 264$,
$x \mapsto (x +  j \adfmod{4}) + 264$ for $x \ge 264$,
$0 \le j < 264$.
\ADFvfyParStart{(268, ((13, 264, ((264, 1), (4, 1)))), ((12, 22), (4, 1)))} 

\adfDgap
\noindent{\boldmath $ 12^{22} 24^{1} $}~
With the point set $\{0, 1, \dots, 287\}$ partitioned into
 residue classes modulo $22$ for $\{0, 1, \dots, 263\}$, and
 $\{264, 265, \dots, 287\}$,
 the design is generated from

\adfLgap {\adfBfont
$\{55, 200, 198, 170, 123\}$,
$\{189, 100, 151, 214, 260\}$,
$\{14, 257, 196, 181, 185\}$,\adfsplit
$\{173, 258, 39, 282, 95\}$,
$\{215, 73, 123, 175, 209\}$,
$\{145, 228, 199, 88, 285\}$,\adfsplit
$\{87, 107, 74, 0, 205\}$,
$\{0, 1, 8, 49, 248\}$,
$\{0, 31, 84, 127, 200\}$,\adfsplit
$\{0, 5, 113, 125, 152\}$,
$\{0, 14, 173, 209, 280\}$,
$\{0, 10, 80, 174, 282\}$,\adfsplit
$\{0, 18, 37, 99, 222\}$,
$\{0, 9, 206, 238, 279\}$,
$\{0, 3, 106, 129, 271\}$

}
\adfLgap \noindent by the mapping:
$x \mapsto x +  j \adfmod{264}$ for $x < 264$,
$x \mapsto (x +  j \adfmod{24}) + 264$ for $x \ge 264$,
$0 \le j < 264$.
\ADFvfyParStart{(288, ((15, 264, ((264, 1), (24, 1)))), ((12, 22), (24, 1)))} 

\adfDgap
\noindent{\boldmath $ 12^{22} 44^{1} $}~
With the point set $\{0, 1, \dots, 307\}$ partitioned into
 residue classes modulo $22$ for $\{0, 1, \dots, 263\}$, and
 $\{264, 265, \dots, 307\}$,
 the design is generated from

\adfLgap {\adfBfont
$\{225, 105, 169, 186, 229\}$,
$\{145, 116, 30, 196, 223\}$,
$\{297, 250, 252, 1, 173\}$,\adfsplit
$\{241, 264, 151, 259, 240\}$,
$\{238, 231, 76, 137, 205\}$,
$\{78, 28, 266, 245, 145\}$,\adfsplit
$\{70, 269, 123, 11, 39\}$,
$\{0, 3, 8, 232, 275\}$,
$\{0, 6, 131, 201, 279\}$,\adfsplit
$\{0, 25, 99, 141, 277\}$,
$\{0, 10, 46, 58, 192\}$,
$\{0, 9, 96, 145, 169\}$,\adfsplit
$\{0, 20, 54, 111, 300\}$,
$\{0, 16, 37, 159, 294\}$,
$\{0, 14, 55, 226, 305\}$,\adfsplit
$\{0, 30, 75, 181, 289\}$,
$\{0, 11, 76, 138, 161\}$

}
\adfLgap \noindent by the mapping:
$x \mapsto x +  j \adfmod{264}$ for $x < 264$,
$x \mapsto (x +  j \adfmod{24}) + 264$ for $264 \le x < 288$,
$x \mapsto (x +  j \adfmod{12}) + 288$ for $288 \le x < 300$,
$x \mapsto (x - 300 +  j \adfmod{8}) + 300$ for $x \ge 300$,
$0 \le j < 264$.
\ADFvfyParStart{(308, ((17, 264, ((264, 1), (24, 1), (12, 1), (8, 1)))), ((12, 22), (44, 1)))} 

\adfDgap
\noindent{\boldmath $ 12^{22} 64^{1} $}~
With the point set $\{0, 1, \dots, 327\}$ partitioned into
 residue classes modulo $22$ for $\{0, 1, \dots, 263\}$, and
 $\{264, 265, \dots, 327\}$,
 the design is generated from

\adfLgap {\adfBfont
$\{216, 50, 184, 145, 131\}$,
$\{256, 319, 162, 149, 192\}$,
$\{248, 317, 84, 11, 67\}$,\adfsplit
$\{278, 92, 152, 153, 102\}$,
$\{177, 287, 171, 162, 173\}$,
$\{155, 202, 324, 85, 184\}$,\adfsplit
$\{291, 58, 27, 95, 6\}$,
$\{70, 37, 1, 185, 269\}$,
$\{50, 267, 140, 62, 4\}$,\adfsplit
$\{0, 3, 155, 259, 323\}$,
$\{0, 23, 146, 238, 301\}$,
$\{0, 34, 79, 142, 309\}$,\adfsplit
$\{0, 20, 77, 125, 197\}$,
$\{0, 24, 86, 127, 182\}$,
$\{0, 35, 111, 224, 280\}$,\adfsplit
$\{0, 25, 84, 126, 294\}$,
$\{0, 7, 131, 150, 290\}$,
$\{0, 16, 135, 226, 304\}$,\adfsplit
$\{0, 28, 93, 190, 296\}$

}
\adfLgap \noindent by the mapping:
$x \mapsto x +  j \adfmod{264}$ for $x < 264$,
$x \mapsto (x - 264 + 2 j \adfmod{48}) + 264$ for $264 \le x < 312$,
$x \mapsto (x +  j \adfmod{12}) + 312$ for $312 \le x < 324$,
$x \mapsto (x +  j \adfmod{4}) + 324$ for $x \ge 324$,
$0 \le j < 264$.
\ADFvfyParStart{(328, ((19, 264, ((264, 1), (48, 2), (12, 1), (4, 1)))), ((12, 22), (64, 1)))} 

\adfDgap
\noindent{\boldmath $ 12^{23} 28^{1} $}~
With the point set $\{0, 1, \dots, 303\}$ partitioned into
 residue classes modulo $23$ for $\{0, 1, \dots, 275\}$, and
 $\{276, 277, \dots, 303\}$,
 the design is generated from

\adfLgap {\adfBfont
$\{269, 101, 208, 67, 10\}$,
$\{49, 186, 114, 147, 207\}$,
$\{124, 51, 282, 214, 205\}$,\adfsplit
$\{194, 294, 19, 48, 241\}$,
$\{8, 274, 277, 85, 151\}$,
$\{300, 33, 70, 95, 76\}$,\adfsplit
$\{0, 1, 3, 14, 278\}$,
$\{0, 5, 119, 126, 276\}$,
$\{0, 15, 41, 246, 281\}$,\adfsplit
$\{0, 4, 22, 170, 194\}$,
$\{0, 27, 76, 116, 218\}$,
$\{0, 20, 48, 151, 245\}$,\adfsplit
$\{0, 8, 50, 209, 221\}$,
$\{0, 32, 70, 129, 181\}$,
$\{0, 16, 100, 136, 180\}$,\adfsplit
$\{0, 35, 88, 144, 212\}$

}
\adfLgap \noindent by the mapping:
$x \mapsto x +  j \adfmod{276}$ for $x < 276$,
$x \mapsto (x - 276 + 7 j \adfmod{28}) + 276$ for $x \ge 276$,
$0 \le j < 276$.
\ADFvfyParStart{(304, ((16, 276, ((276, 1), (28, 7)))), ((12, 23), (28, 1)))} 

\adfDgap
\noindent{\boldmath $ 12^{23} 48^{1} $}~
With the point set $\{0, 1, \dots, 323\}$ partitioned into
 residue classes modulo $23$ for $\{0, 1, \dots, 275\}$, and
 $\{276, 277, \dots, 323\}$,
 the design is generated from

\adfLgap {\adfBfont
$\{45, 3, 146, 197, 315\}$,
$\{73, 78, 276, 29, 255\}$,
$\{109, 99, 92, 90, 121\}$,\adfsplit
$\{268, 64, 19, 192, 102\}$,
$\{312, 56, 133, 203, 67\}$,
$\{156, 3, 281, 88, 273\}$,\adfsplit
$\{302, 23, 27, 98, 210\}$,
$\{6, 86, 21, 85, 242\}$,
$\{0, 3, 154, 263, 288\}$,\adfsplit
$\{0, 8, 36, 96, 228\}$,
$\{0, 35, 98, 141, 203\}$,
$\{0, 21, 53, 155, 179\}$,\adfsplit
$\{0, 20, 127, 222, 315\}$,
$\{0, 25, 162, 243, 323\}$,
$\{0, 14, 190, 229, 322\}$,\adfsplit
$\{0, 30, 87, 146, 310\}$,
$\{0, 18, 131, 242, 297\}$,
$\{0, 26, 198, 235, 285\}$

}
\adfLgap \noindent by the mapping:
$x \mapsto x +  j \adfmod{276}$ for $x < 276$,
$x \mapsto (x - 276 + 4 j \adfmod{48}) + 276$ for $x \ge 276$,
$0 \le j < 276$.
\ADFvfyParStart{(324, ((18, 276, ((276, 1), (48, 4)))), ((12, 23), (48, 1)))} 

\adfDgap
\noindent{\boldmath $ 12^{23} 68^{1} $}~
With the point set $\{0, 1, \dots, 343\}$ partitioned into
 residue classes modulo $23$ for $\{0, 1, \dots, 275\}$, and
 $\{276, 277, \dots, 343\}$,
 the design is generated from

\adfLgap {\adfBfont
$\{60, 341, 75, 237, 178\}$,
$\{17, 239, 64, 15, 307\}$,
$\{173, 85, 6, 216, 224\}$,\adfsplit
$\{87, 55, 116, 52, 282\}$,
$\{213, 269, 90, 299, 188\}$,
$\{275, 61, 303, 185, 255\}$,\adfsplit
$\{318, 259, 265, 191, 165\}$,
$\{233, 342, 147, 106, 252\}$,
$\{161, 248, 241, 325, 76\}$,\adfsplit
$\{306, 43, 197, 119, 64\}$,
$\{0, 5, 22, 147, 318\}$,
$\{0, 4, 57, 164, 317\}$,\adfsplit
$\{0, 12, 48, 108, 132\}$,
$\{0, 1, 34, 45, 72\}$,
$\{0, 14, 135, 248, 290\}$,\adfsplit
$\{0, 30, 95, 170, 325\}$,
$\{0, 10, 193, 209, 289\}$,
$\{0, 13, 63, 102, 309\}$,\adfsplit
$\{0, 18, 166, 203, 296\}$,
$\{0, 9, 40, 159, 276\}$

}
\adfLgap \noindent by the mapping:
$x \mapsto x +  j \adfmod{276}$ for $x < 276$,
$x \mapsto (x - 276 + 5 j \adfmod{60}) + 276$ for $276 \le x < 336$,
$x \mapsto (x + 2 j \adfmod{8}) + 336$ for $x \ge 336$,
$0 \le j < 276$.
\ADFvfyParStart{(344, ((20, 276, ((276, 1), (60, 5), (8, 2)))), ((12, 23), (68, 1)))} 

\adfDgap
\noindent{\boldmath $ 12^{24} 32^{1} $}~
With the point set $\{0, 1, \dots, 319\}$ partitioned into
 residue classes modulo $24$ for $\{0, 1, \dots, 287\}$, and
 $\{288, 289, \dots, 319\}$,
 the design is generated from

\adfLgap {\adfBfont
$\{309, 261, 170, 111, 185\}$,
$\{188, 317, 11, 222, 221\}$,
$\{300, 266, 61, 245, 127\}$,\adfsplit
$\{107, 127, 135, 240, 296\}$,
$\{266, 51, 319, 217, 92\}$,
$\{36, 166, 285, 196, 1\}$,\adfsplit
$\{27, 307, 137, 49, 245\}$,
$\{210, 48, 111, 108, 95\}$,
$\{120, 146, 65, 151, 134\}$,\adfsplit
$\{0, 2, 54, 252, 303\}$,
$\{0, 9, 94, 176, 308\}$,
$\{0, 6, 50, 137, 232\}$,\adfsplit
$\{0, 10, 42, 67, 223\}$,
$\{0, 19, 80, 103, 148\}$,
$\{0, 11, 51, 109, 152\}$,\adfsplit
$\{0, 18, 97, 134, 161\}$,
$\{0, 7, 53, 124, 153\}$

}
\adfLgap \noindent by the mapping:
$x \mapsto x +  j \adfmod{288}$ for $x < 288$,
$x \mapsto (x +  j \adfmod{24}) + 288$ for $288 \le x < 312$,
$x \mapsto (x +  j \adfmod{8}) + 312$ for $x \ge 312$,
$0 \le j < 288$.
\ADFvfyParStart{(320, ((17, 288, ((288, 1), (24, 1), (8, 1)))), ((12, 24), (32, 1)))} 

\adfDgap
\noindent{\boldmath $ 12^{24} 52^{1} $}~
With the point set $\{0, 1, \dots, 339\}$ partitioned into
 residue classes modulo $24$ for $\{0, 1, \dots, 287\}$, and
 $\{288, 289, \dots, 339\}$,
 the design is generated from

\adfLgap {\adfBfont
$\{110, 286, 312, 184, 33\}$,
$\{87, 174, 240, 31, 307\}$,
$\{76, 299, 67, 61, 174\}$,\adfsplit
$\{68, 335, 222, 191, 70\}$,
$\{96, 5, 165, 135, 51\}$,
$\{56, 111, 75, 197, 2\}$,\adfsplit
$\{331, 173, 120, 190, 172\}$,
$\{330, 2, 270, 227, 277\}$,
$\{74, 223, 174, 300, 182\}$,\adfsplit
$\{113, 314, 54, 139, 136\}$,
$\{0, 34, 140, 228, 326\}$,
$\{0, 22, 47, 193, 336\}$,\adfsplit
$\{0, 12, 28, 104, 185\}$,
$\{0, 4, 44, 163, 227\}$,
$\{0, 14, 71, 138, 170\}$,\adfsplit
$\{0, 29, 162, 220, 309\}$,
$\{0, 11, 62, 89, 172\}$,
$\{0, 10, 187, 208, 299\}$,\adfsplit
$\{0, 5, 38, 251, 300\}$

}
\adfLgap \noindent by the mapping:
$x \mapsto x +  j \adfmod{288}$ for $x < 288$,
$x \mapsto (x +  j \adfmod{32}) + 288$ for $288 \le x < 320$,
$x \mapsto (x +  j \adfmod{16}) + 320$ for $320 \le x < 336$,
$x \mapsto (x +  j \adfmod{4}) + 336$ for $x \ge 336$,
$0 \le j < 288$.
\ADFvfyParStart{(340, ((19, 288, ((288, 1), (32, 1), (16, 1), (4, 1)))), ((12, 24), (52, 1)))} 

\adfDgap
\noindent{\boldmath $ 12^{27} 4^{1} $}~
With the point set $\{0, 1, \dots, 327\}$ partitioned into
 residue classes modulo $27$ for $\{0, 1, \dots, 323\}$, and
 $\{324, 325, 326, 327\}$,
 the design is generated from

\adfLgap {\adfBfont
$\{210, 26, 110, 200, 214\}$,
$\{264, 306, 71, 322, 209\}$,
$\{189, 225, 33, 259, 113\}$,\adfsplit
$\{128, 192, 210, 203, 255\}$,
$\{169, 109, 160, 261, 286\}$,
$\{138, 36, 19, 158, 13\}$,\adfsplit
$\{144, 39, 26, 72, 313\}$,
$\{0, 3, 94, 137, 161\}$,
$\{0, 1, 30, 259, 324\}$,\adfsplit
$\{0, 15, 86, 215, 246\}$,
$\{0, 12, 44, 203, 225\}$,
$\{0, 5, 53, 115, 176\}$,\adfsplit
$\{0, 8, 47, 204, 283\}$,
$\{0, 2, 28, 87, 144\}$,
$\{0, 35, 103, 141, 210\}$,\adfsplit
$\{0, 19, 40, 96, 170\}$

}
\adfLgap \noindent by the mapping:
$x \mapsto x +  j \adfmod{324}$ for $x < 324$,
$x \mapsto (x +  j \adfmod{4}) + 324$ for $x \ge 324$,
$0 \le j < 324$.
\ADFvfyParStart{(328, ((16, 324, ((324, 1), (4, 1)))), ((12, 27), (4, 1)))} 

\adfDgap
\noindent{\boldmath $ 13^{8} 17^{1} $}~
With the point set $\{0, 1, \dots, 120\}$ partitioned into
 residue classes modulo $8$ for $\{0, 1, \dots, 103\}$, and
 $\{104, 105, \dots, 120\}$,
 the design is generated from

\adfLgap {\adfBfont
$\{19, 108, 60, 61, 7\}$,
$\{0, 2, 19, 100, 112\}$,
$\{0, 7, 29, 44, 105\}$,\adfsplit
$\{0, 3, 28, 33, 46\}$,
$\{0, 10, 31, 45, 65\}$,
$\{0, 9, 36, 47, 117\}$,\adfsplit
$\{0, 26, 52, 78, 120\}$

}
\adfLgap \noindent by the mapping:
$x \mapsto x +  j \adfmod{104}$ for $x < 104$,
$x \mapsto (x - 104 + 2 j \adfmod{16}) + 104$ for $104 \le x < 120$,
$120 \mapsto 120$,
$0 \le j < 104$
 for the first six blocks,
$0 \le j < 26$
 for the last block.
\ADFvfyParStart{(121, ((6, 104, ((104, 1), (16, 2), (1, 1))), (1, 26, ((104, 1), (16, 2), (1, 1)))), ((13, 8), (17, 1)))} 

\adfDgap
\noindent{\boldmath $ 13^{12} 1^{1} $}~
With the point set $\{0, 1, \dots, 156\}$ partitioned into
 residue classes modulo $12$ for $\{0, 1, \dots, 155\}$, and
 $\{156\}$,
 the design is generated from

\adfLgap {\adfBfont
$\{79, 33, 38, 29, 142\}$,
$\{21, 108, 41, 133, 102\}$,
$\{0, 2, 10, 29, 32\}$,\adfsplit
$\{0, 1, 14, 54, 71\}$,
$\{0, 11, 37, 88, 111\}$,
$\{0, 7, 28, 66, 101\}$,\adfsplit
$\{0, 15, 33, 49, 91\}$,
$\{0, 39, 78, 117, 156\}$

}
\adfLgap \noindent by the mapping:
$x \mapsto x +  j \adfmod{156}$ for $x < 156$,
$156 \mapsto 156$,
$0 \le j < 156$
 for the first seven blocks,
$0 \le j < 39$
 for the last block.
\ADFvfyParStart{(157, ((7, 156, ((156, 1), (1, 1))), (1, 39, ((156, 1), (1, 1)))), ((13, 12), (1, 1)))} 

\adfDgap
\noindent{\boldmath $ 13^{12} 21^{1} $}~
With the point set $\{0, 1, \dots, 176\}$ partitioned into
 residue classes modulo $12$ for $\{0, 1, \dots, 155\}$, and
 $\{156, 157, \dots, 176\}$,
 the design is generated from

\adfLgap {\adfBfont
$\{124, 152, 156, 101, 21\}$,
$\{47, 85, 147, 28, 21\}$,
$\{0, 1, 3, 146, 169\}$,\adfsplit
$\{0, 5, 22, 63, 168\}$,
$\{0, 21, 65, 107, 167\}$,
$\{0, 8, 55, 124, 161\}$,\adfsplit
$\{0, 6, 33, 67, 85\}$,
$\{0, 4, 20, 35, 110\}$,
$\{0, 9, 54, 83, 97\}$,\adfsplit
$\{0, 39, 78, 117, 176\}$

}
\adfLgap \noindent by the mapping:
$x \mapsto x +  j \adfmod{156}$ for $x < 156$,
$x \mapsto (x +  j \adfmod{12}) + 156$ for $156 \le x < 168$,
$x \mapsto (x + 2 j \adfmod{8}) + 168$ for $168 \le x < 176$,
$176 \mapsto 176$,
$0 \le j < 156$
 for the first nine blocks,
$0 \le j < 39$
 for the last block.
\ADFvfyParStart{(177, ((9, 156, ((156, 1), (12, 1), (8, 2), (1, 1))), (1, 39, ((156, 1), (12, 1), (8, 2), (1, 1)))), ((13, 12), (21, 1)))} 

\adfDgap
\noindent{\boldmath $ 13^{12} 41^{1} $}~
With the point set $\{0, 1, \dots, 196\}$ partitioned into
 residue classes modulo $12$ for $\{0, 1, \dots, 155\}$, and
 $\{156, 157, \dots, 196\}$,
 the design is generated from

\adfLgap {\adfBfont
$\{166, 115, 111, 47, 93\}$,
$\{76, 185, 66, 75, 47\}$,
$\{107, 27, 163, 0, 65\}$,\adfsplit
$\{114, 32, 168, 17, 19\}$,
$\{0, 3, 8, 102, 160\}$,
$\{0, 7, 21, 126, 192\}$,\adfsplit
$\{0, 26, 71, 112, 168\}$,
$\{0, 6, 40, 109, 177\}$,
$\{0, 23, 75, 100, 191\}$,\adfsplit
$\{0, 16, 33, 83, 161\}$,
$\{0, 11, 31, 66, 124\}$,
$\{0, 39, 78, 117, 196\}$

}
\adfLgap \noindent by the mapping:
$x \mapsto x +  j \adfmod{156}$ for $x < 156$,
$x \mapsto (x - 156 + 3 j \adfmod{36}) + 156$ for $156 \le x < 192$,
$x \mapsto (x +  j \adfmod{4}) + 192$ for $192 \le x < 196$,
$196 \mapsto 196$,
$0 \le j < 156$
 for the first 11 blocks,
$0 \le j < 39$
 for the last block.
\ADFvfyParStart{(197, ((11, 156, ((156, 1), (36, 3), (4, 1), (1, 1))), (1, 39, ((156, 1), (36, 3), (4, 1), (1, 1)))), ((13, 12), (41, 1)))} 

\adfDgap
\noindent{\boldmath $ 13^{16} 5^{1} $}~
With the point set $\{0, 1, \dots, 212\}$ partitioned into
 residue classes modulo $16$ for $\{0, 1, \dots, 207\}$, and
 $\{208, 209, 210, 211, 212\}$,
 the design is generated from

\adfLgap {\adfBfont
$\{92, 58, 2, 45, 89\}$,
$\{188, 103, 115, 104, 85\}$,
$\{82, 72, 0, 107, 14\}$,\adfsplit
$\{0, 2, 8, 127, 193\}$,
$\{0, 7, 46, 141, 208\}$,
$\{0, 28, 57, 120, 157\}$,\adfsplit
$\{0, 21, 59, 99, 153\}$,
$\{0, 20, 62, 137, 159\}$,
$\{0, 24, 50, 77, 110\}$,\adfsplit
$\{0, 4, 9, 45, 147\}$,
$\{0, 52, 104, 156, 212\}$

}
\adfLgap \noindent by the mapping:
$x \mapsto x +  j \adfmod{208}$ for $x < 208$,
$x \mapsto (x +  j \adfmod{4}) + 208$ for $208 \le x < 212$,
$212 \mapsto 212$,
$0 \le j < 208$
 for the first ten blocks,
$0 \le j < 52$
 for the last block.
\ADFvfyParStart{(213, ((10, 208, ((208, 1), (4, 1), (1, 1))), (1, 52, ((208, 1), (4, 1), (1, 1)))), ((13, 16), (5, 1)))} 

\adfDgap
\noindent{\boldmath $ 13^{16} 25^{1} $}~
With the point set $\{0, 1, \dots, 232\}$ partitioned into
 residue classes modulo $16$ for $\{0, 1, \dots, 207\}$, and
 $\{208, 209, \dots, 232\}$,
 the design is generated from

\adfLgap {\adfBfont
$\{207, 123, 149, 124, 212\}$,
$\{138, 91, 115, 28, 120\}$,
$\{217, 62, 31, 204, 53\}$,\adfsplit
$\{6, 39, 88, 178, 193\}$,
$\{0, 2, 67, 73, 217\}$,
$\{0, 3, 122, 132, 218\}$,\adfsplit
$\{0, 4, 12, 127, 170\}$,
$\{0, 13, 53, 130, 147\}$,
$\{0, 7, 34, 62, 106\}$,\adfsplit
$\{0, 19, 39, 107, 152\}$,
$\{0, 14, 51, 162, 219\}$,
$\{0, 11, 41, 149, 231\}$,\adfsplit
$\{0, 52, 104, 156, 232\}$

}
\adfLgap \noindent by the mapping:
$x \mapsto x +  j \adfmod{208}$ for $x < 208$,
$x \mapsto (x - 208 + 3 j \adfmod{24}) + 208$ for $208 \le x < 232$,
$232 \mapsto 232$,
$0 \le j < 208$
 for the first 12 blocks,
$0 \le j < 52$
 for the last block.
\ADFvfyParStart{(233, ((12, 208, ((208, 1), (24, 3), (1, 1))), (1, 52, ((208, 1), (24, 3), (1, 1)))), ((13, 16), (25, 1)))} 

\adfDgap
\noindent{\boldmath $ 13^{20} 1^{1} $}~
With the point set $\{0, 1, \dots, 260\}$ partitioned into
 residue classes modulo $20$ for $\{0, 1, \dots, 259\}$, and
 $\{260\}$,
 the design is generated from

\adfLgap {\adfBfont
$\{132, 30, 37, 91, 158\}$,
$\{150, 95, 129, 47, 125\}$,
$\{208, 65, 144, 100, 159\}$,\adfsplit
$\{255, 238, 167, 0, 124\}$,
$\{206, 144, 88, 125, 195\}$,
$\{0, 2, 68, 86, 99\}$,\adfsplit
$\{0, 12, 57, 110, 149\}$,
$\{0, 1, 9, 134, 184\}$,
$\{0, 3, 27, 90, 191\}$,\adfsplit
$\{0, 14, 47, 89, 105\}$,
$\{0, 6, 29, 112, 144\}$,
$\{0, 10, 38, 151, 224\}$,\adfsplit
$\{0, 65, 130, 195, 260\}$,
$\{0, 52, 104, 156, 208\}$

}
\adfLgap \noindent by the mapping:
$x \mapsto x +  j \adfmod{260}$ for $x < 260$,
$260 \mapsto 260$,
$0 \le j < 260$
 for the first 12 blocks,
$0 \le j < 65$
 for the next block,
$0 \le j < 52$
 for the last block.
\ADFvfyParStart{(261, ((12, 260, ((260, 1), (1, 1))), (1, 65, ((260, 1), (1, 1))), (1, 52, ((260, 1), (1, 1)))), ((13, 20), (1, 1)))} 

\adfDgap
\noindent{\boldmath $ 14^{8} 6^{1} $}~
With the point set $\{0, 1, \dots, 117\}$ partitioned into
 residue classes modulo $8$ for $\{0, 1, \dots, 111\}$, and
 $\{112, 113, \dots, 117\}$,
 the design is generated from

\adfLgap {\adfBfont
$\{55, 82, 62, 73, 114\}$,
$\{30, 64, 83, 1, 109\}$,
$\{106, 65, 53, 75, 76\}$,\adfsplit
$\{0, 1, 47, 108, 112\}$,
$\{0, 2, 65, 100, 109\}$,
$\{0, 7, 58, 76, 91\}$,\adfsplit
$\{0, 21, 35, 55, 97\}$,
$\{0, 3, 22, 28, 74\}$,
$\{0, 13, 50, 67, 73\}$,\adfsplit
$\{0, 25, 42, 68, 99\}$,
$\{0, 10, 37, 39, 116\}$

}
\adfLgap \noindent by the mapping:
$x \mapsto x + 2 j \adfmod{112}$ for $x < 112$,
$x \mapsto (x +  j \adfmod{4}) + 112$ for $112 \le x < 116$,
$x \mapsto (x +  j \adfmod{2}) + 116$ for $x \ge 116$,
$0 \le j < 56$.
\ADFvfyParStart{(118, ((11, 56, ((112, 2), (4, 1), (2, 1)))), ((14, 8), (6, 1)))} 

\adfDgap
\noindent{\boldmath $ 16^{6} 20^{1} $}~
With the point set $\{0, 1, \dots, 115\}$ partitioned into
 residue classes modulo $6$ for $\{0, 1, \dots, 95\}$, and
 $\{96, 97, \dots, 115\}$,
 the design is generated from

\adfLgap {\adfBfont
$\{0, 1, 3, 10, 112\}$,
$\{0, 4, 32, 47, 63\}$,
$\{0, 20, 46, 75, 96\}$,\adfsplit
$\{0, 11, 34, 51, 106\}$,
$\{0, 14, 39, 58, 107\}$,
$\{0, 5, 13, 74, 99\}$

}
\adfLgap \noindent by the mapping:
$x \mapsto x +  j \adfmod{96}$ for $x < 96$,
$x \mapsto (x +  j \adfmod{16}) + 96$ for $96 \le x < 112$,
$x \mapsto (x +  j \adfmod{4}) + 112$ for $x \ge 112$,
$0 \le j < 96$.
\ADFvfyParStart{(116, ((6, 96, ((96, 1), (16, 1), (4, 1)))), ((16, 6), (20, 1)))} 

\adfDgap
\noindent{\boldmath $ 16^{7} 12^{1} $}~
With the point set $\{0, 1, \dots, 123\}$ partitioned into
 residue classes modulo $7$ for $\{0, 1, \dots, 111\}$, and
 $\{112, 113, \dots, 123\}$,
 the design is generated from

\adfLgap {\adfBfont
$\{0, 1, 3, 102, 112\}$,
$\{0, 5, 23, 90, 114\}$,
$\{0, 9, 46, 71, 113\}$,\adfsplit
$\{0, 15, 34, 72, 88\}$,
$\{0, 4, 30, 36, 83\}$,
$\{0, 8, 20, 51, 68\}$

}
\adfLgap \noindent by the mapping:
$x \mapsto x +  j \adfmod{112}$ for $x < 112$,
$x \mapsto (x - 112 + 3 j \adfmod{12}) + 112$ for $x \ge 112$,
$0 \le j < 112$.
\ADFvfyParStart{(124, ((6, 112, ((112, 1), (12, 3)))), ((16, 7), (12, 1)))} 

\adfDgap
\noindent{\boldmath $ 16^{8} 4^{1} $}~
With the point set $\{0, 1, \dots, 131\}$ partitioned into
 residue classes modulo $8$ for $\{0, 1, \dots, 127\}$, and
 $\{128, 129, 130, 131\}$,
 the design is generated from

\adfLgap {\adfBfont
$\{0, 1, 3, 122, 128\}$,
$\{0, 4, 41, 46, 61\}$,
$\{0, 12, 31, 59, 89\}$,\adfsplit
$\{0, 10, 27, 53, 105\}$,
$\{0, 13, 35, 49, 103\}$,
$\{0, 11, 29, 73, 94\}$

}
\adfLgap \noindent by the mapping:
$x \mapsto x +  j \adfmod{128}$ for $x < 128$,
$x \mapsto (x +  j \adfmod{4}) + 128$ for $x \ge 128$,
$0 \le j < 128$.
\ADFvfyParStart{(132, ((6, 128, ((128, 1), (4, 1)))), ((16, 8), (4, 1)))} 

\adfDgap
\noindent{\boldmath $ 16^{9} 36^{1} $}~
With the point set $\{0, 1, \dots, 179\}$ partitioned into
 residue classes modulo $9$ for $\{0, 1, \dots, 143\}$, and
 $\{144, 145, \dots, 179\}$,
 the design is generated from

\adfLgap {\adfBfont
$\{99, 111, 164, 13, 125\}$,
$\{159, 109, 47, 96, 30\}$,
$\{106, 135, 2, 102, 151\}$,\adfsplit
$\{0, 1, 3, 123, 169\}$,
$\{0, 5, 39, 64, 74\}$,
$\{0, 6, 48, 103, 150\}$,\adfsplit
$\{0, 8, 28, 101, 170\}$,
$\{0, 19, 57, 113, 164\}$,
$\{0, 15, 67, 83, 167\}$,\adfsplit
$\{0, 7, 91, 114, 146\}$

}
\adfLgap \noindent by the mapping:
$x \mapsto x +  j \adfmod{144}$ for $x < 144$,
$x \mapsto (x +  j \adfmod{36}) + 144$ for $x \ge 144$,
$0 \le j < 144$.
\ADFvfyParStart{(180, ((10, 144, ((144, 1), (36, 1)))), ((16, 9), (36, 1)))} 

\adfDgap
\noindent{\boldmath $ 16^{10} 8^{1} $}~
With the point set $\{0, 1, \dots, 167\}$ partitioned into
 residue classes modulo $10$ for $\{0, 1, \dots, 159\}$, and
 $\{160, 161, \dots, 167\}$,
 the design is generated from

\adfLgap {\adfBfont
$\{27, 82, 164, 94, 28\}$,
$\{49, 117, 21, 150, 32\}$,
$\{0, 2, 38, 41, 165\}$,\adfsplit
$\{0, 7, 63, 72, 86\}$,
$\{0, 15, 34, 61, 77\}$,
$\{0, 4, 25, 51, 73\}$,\adfsplit
$\{0, 6, 24, 82, 131\}$,
$\{0, 5, 13, 57, 128\}$

}
\adfLgap \noindent by the mapping:
$x \mapsto x +  j \adfmod{160}$ for $x < 160$,
$x \mapsto (x +  j \adfmod{8}) + 160$ for $x \ge 160$,
$0 \le j < 160$.
\ADFvfyParStart{(168, ((8, 160, ((160, 1), (8, 1)))), ((16, 10), (8, 1)))} 

\adfDgap
\noindent{\boldmath $ 16^{10} 20^{1} $}~
With the point set $\{0, 1, \dots, 179\}$ partitioned into
 residue classes modulo $10$ for $\{0, 1, \dots, 159\}$, and
 $\{160, 161, \dots, 179\}$,
 the design is generated from

\adfLgap {\adfBfont
$\{142, 16, 95, 88, 124\}$,
$\{139, 154, 174, 22, 81\}$,
$\{0, 1, 3, 26, 155\}$,\adfsplit
$\{0, 4, 16, 93, 107\}$,
$\{0, 21, 45, 82, 168\}$,
$\{0, 17, 56, 122, 161\}$,\adfsplit
$\{0, 9, 51, 84, 95\}$,
$\{0, 13, 48, 111, 162\}$,
$\{0, 19, 46, 138, 176\}$,\adfsplit
$\{0, 32, 64, 96, 128\}$

}
\adfLgap \noindent by the mapping:
$x \mapsto x +  j \adfmod{160}$ for $x < 160$,
$x \mapsto (x +  j \adfmod{20}) + 160$ for $x \ge 160$,
$0 \le j < 160$
 for the first nine blocks,
$0 \le j < 32$
 for the last block.
\ADFvfyParStart{(180, ((9, 160, ((160, 1), (20, 1))), (1, 32, ((160, 1), (20, 1)))), ((16, 10), (20, 1)))} 

\adfDgap
\noindent{\boldmath $ 16^{10} 28^{1} $}~
With the point set $\{0, 1, \dots, 187\}$ partitioned into
 residue classes modulo $10$ for $\{0, 1, \dots, 159\}$, and
 $\{160, 161, \dots, 187\}$,
 the design is generated from

\adfLgap {\adfBfont
$\{37, 32, 74, 73, 177\}$,
$\{17, 10, 118, 165, 91\}$,
$\{153, 182, 100, 155, 88\}$,\adfsplit
$\{0, 3, 9, 22, 184\}$,
$\{0, 4, 29, 116, 132\}$,
$\{0, 11, 26, 88, 122\}$,\adfsplit
$\{0, 8, 54, 117, 168\}$,
$\{0, 18, 94, 139, 167\}$,
$\{0, 14, 31, 89, 113\}$,\adfsplit
$\{0, 23, 92, 127, 173\}$

}
\adfLgap \noindent by the mapping:
$x \mapsto x +  j \adfmod{160}$ for $x < 160$,
$x \mapsto (x +  j \adfmod{16}) + 160$ for $160 \le x < 176$,
$x \mapsto (x +  j \adfmod{8}) + 176$ for $176 \le x < 184$,
$x \mapsto (x +  j \adfmod{4}) + 184$ for $x \ge 184$,
$0 \le j < 160$.
\ADFvfyParStart{(188, ((10, 160, ((160, 1), (16, 1), (8, 1), (4, 1)))), ((16, 10), (28, 1)))} 

\adfDgap
\noindent{\boldmath $ 16^{10} 36^{1} $}~
With the point set $\{0, 1, \dots, 195\}$ partitioned into
 residue classes modulo $9$ for $\{0, 1, \dots, 143\}$,
 $\{144, 145, \dots, 159\}$, and
 $\{160, 161, \dots, 195\}$,
 the design is generated from

\adfLgap {\adfBfont
$\{157, 125, 56, 57, 160\}$,
$\{133, 110, 172, 40, 99\}$,
$\{6, 128, 46, 192, 14\}$,\adfsplit
$\{5, 129, 11, 151, 63\}$,
$\{0, 2, 5, 147, 164\}$,
$\{0, 12, 106, 153, 185\}$,\adfsplit
$\{0, 13, 28, 151, 161\}$,
$\{0, 24, 65, 107, 189\}$,
$\{0, 14, 109, 125, 194\}$,\adfsplit
$\{0, 10, 31, 98, 115\}$,
$\{0, 7, 55, 80, 174\}$,
$\{0, 4, 57, 101, 192\}$

}
\adfLgap \noindent by the mapping:
$x \mapsto x +  j \adfmod{144}$ for $x < 144$,
$x \mapsto (x +  j \adfmod{16}) + 144$ for $144 \le x < 160$,
$x \mapsto (x - 160 +  j \adfmod{36}) + 160$ for $x \ge 160$,
$0 \le j < 144$.
\ADFvfyParStart{(196, ((12, 144, ((144, 1), (16, 1), (36, 1)))), ((16, 9), (16, 1), (36, 1)))} 

\adfDgap
\noindent{\boldmath $ 16^{10} 40^{1} $}~
With the point set $\{0, 1, \dots, 199\}$ partitioned into
 residue classes modulo $10$ for $\{0, 1, \dots, 159\}$, and
 $\{160, 161, \dots, 199\}$,
 the design is generated from

\adfLgap {\adfBfont
$\{181, 39, 90, 12, 27\}$,
$\{116, 85, 83, 121, 37\}$,
$\{111, 140, 83, 36, 194\}$,\adfsplit
$\{27, 101, 76, 197, 10\}$,
$\{0, 1, 4, 138, 196\}$,
$\{0, 14, 67, 139, 164\}$,\adfsplit
$\{0, 6, 45, 89, 189\}$,
$\{0, 13, 37, 55, 172\}$,
$\{0, 9, 101, 117, 162\}$,\adfsplit
$\{0, 8, 62, 149, 168\}$,
$\{0, 7, 65, 126, 173\}$,
$\{0, 32, 64, 96, 128\}$

}
\adfLgap \noindent by the mapping:
$x \mapsto x +  j \adfmod{160}$ for $x < 160$,
$x \mapsto (x +  j \adfmod{40}) + 160$ for $x \ge 160$,
$0 \le j < 160$
 for the first 11 blocks,
$0 \le j < 32$
 for the last block.
\ADFvfyParStart{(200, ((11, 160, ((160, 1), (40, 1))), (1, 32, ((160, 1), (40, 1)))), ((16, 10), (40, 1)))} 

\adfDgap
\noindent{\boldmath $ 16^{11} 20^{1} $}~
With the point set $\{0, 1, \dots, 195\}$ partitioned into
 residue classes modulo $11$ for $\{0, 1, \dots, 175\}$, and
 $\{176, 177, \dots, 195\}$,
 the design is generated from

\adfLgap {\adfBfont
$\{71, 100, 165, 128, 142\}$,
$\{100, 51, 70, 193, 49\}$,
$\{63, 3, 155, 117, 76\}$,\adfsplit
$\{180, 22, 148, 112, 53\}$,
$\{0, 1, 7, 17, 113\}$,
$\{0, 26, 53, 98, 144\}$,\adfsplit
$\{0, 9, 56, 76, 124\}$,
$\{0, 4, 12, 87, 189\}$,
$\{0, 3, 136, 161, 187\}$,\adfsplit
$\{0, 5, 39, 74, 188\}$

}
\adfLgap \noindent by the mapping:
$x \mapsto x +  j \adfmod{176}$ for $x < 176$,
$x \mapsto (x +  j \adfmod{16}) + 176$ for $176 \le x < 192$,
$x \mapsto (x +  j \adfmod{4}) + 192$ for $x \ge 192$,
$0 \le j < 176$.
\ADFvfyParStart{(196, ((10, 176, ((176, 1), (16, 1), (4, 1)))), ((16, 11), (20, 1)))} 

\adfDgap
\noindent{\boldmath $ 16^{11} 40^{1} $}~
With the point set $\{0, 1, \dots, 215\}$ partitioned into
 residue classes modulo $11$ for $\{0, 1, \dots, 175\}$, and
 $\{176, 177, \dots, 215\}$,
 the design is generated from

\adfLgap {\adfBfont
$\{23, 35, 188, 26, 129\}$,
$\{92, 212, 129, 155, 90\}$,
$\{51, 200, 140, 69, 31\}$,\adfsplit
$\{10, 2, 193, 124, 171\}$,
$\{0, 1, 14, 31, 213\}$,
$\{0, 4, 46, 124, 188\}$,\adfsplit
$\{0, 29, 74, 115, 204\}$,
$\{0, 5, 53, 80, 112\}$,
$\{0, 6, 16, 97, 157\}$,\adfsplit
$\{0, 36, 93, 136, 193\}$,
$\{0, 21, 49, 72, 197\}$,
$\{0, 24, 92, 142, 207\}$

}
\adfLgap \noindent by the mapping:
$x \mapsto x +  j \adfmod{176}$ for $x < 176$,
$x \mapsto (x - 176 + 2 j \adfmod{32}) + 176$ for $176 \le x < 208$,
$x \mapsto (x +  j \adfmod{8}) + 208$ for $x \ge 208$,
$0 \le j < 176$.
\ADFvfyParStart{(216, ((12, 176, ((176, 1), (32, 2), (8, 1)))), ((16, 11), (40, 1)))} 

\adfDgap
\noindent{\boldmath $ 16^{12} 12^{1} $}~
With the point set $\{0, 1, \dots, 203\}$ partitioned into
 residue classes modulo $12$ for $\{0, 1, \dots, 191\}$, and
 $\{192, 193, \dots, 203\}$,
 the design is generated from

\adfLgap {\adfBfont
$\{166, 110, 88, 203, 139\}$,
$\{137, 136, 128, 134, 162\}$,
$\{137, 44, 83, 157, 93\}$,\adfsplit
$\{0, 4, 41, 46, 59\}$,
$\{0, 11, 32, 103, 117\}$,
$\{0, 7, 40, 105, 122\}$,\adfsplit
$\{0, 16, 63, 125, 195\}$,
$\{0, 38, 90, 135, 200\}$,
$\{0, 19, 50, 80, 123\}$,\adfsplit
$\{0, 15, 68, 91, 126\}$

}
\adfLgap \noindent by the mapping:
$x \mapsto x +  j \adfmod{192}$ for $x < 192$,
$x \mapsto (x +  j \adfmod{12}) + 192$ for $x \ge 192$,
$0 \le j < 192$.
\ADFvfyParStart{(204, ((10, 192, ((192, 1), (12, 1)))), ((16, 12), (12, 1)))} 

\adfDgap
\noindent{\boldmath $ 16^{12} 52^{1} $}~
With the point set $\{0, 1, \dots, 243\}$ partitioned into
 residue classes modulo $12$ for $\{0, 1, \dots, 191\}$, and
 $\{192, 193, \dots, 243\}$,
 the design is generated from

\adfLgap {\adfBfont
$\{243, 24, 37, 143, 114\}$,
$\{217, 169, 90, 168, 15\}$,
$\{162, 92, 96, 234, 65\}$,\adfsplit
$\{139, 80, 160, 98, 205\}$,
$\{131, 236, 7, 89, 123\}$,
$\{141, 138, 166, 173, 205\}$,\adfsplit
$\{0, 14, 44, 173, 194\}$,
$\{0, 22, 45, 149, 196\}$,
$\{0, 2, 123, 139, 227\}$,\adfsplit
$\{0, 5, 57, 182, 235\}$,
$\{0, 17, 64, 118, 236\}$,
$\{0, 20, 46, 131, 205\}$,\adfsplit
$\{0, 9, 58, 98, 109\}$,
$\{0, 6, 56, 93, 212\}$

}
\adfLgap \noindent by the mapping:
$x \mapsto x +  j \adfmod{192}$ for $x < 192$,
$x \mapsto (x +  j \adfmod{48}) + 192$ for $192 \le x < 240$,
$x \mapsto (x +  j \adfmod{4}) + 240$ for $x \ge 240$,
$0 \le j < 192$.
\ADFvfyParStart{(244, ((14, 192, ((192, 1), (48, 1), (4, 1)))), ((16, 12), (52, 1)))} 

\adfDgap
\noindent{\boldmath $ 16^{13} 4^{1} $}~
With the point set $\{0, 1, \dots, 211\}$ partitioned into
 residue classes modulo $13$ for $\{0, 1, \dots, 207\}$, and
 $\{208, 209, 210, 211\}$,
 the design is generated from

\adfLgap {\adfBfont
$\{208, 56, 62, 141, 175\}$,
$\{151, 67, 149, 49, 98\}$,
$\{25, 173, 72, 205, 136\}$,\adfsplit
$\{17, 50, 127, 195, 60\}$,
$\{0, 4, 29, 48, 166\}$,
$\{0, 9, 20, 112, 167\}$,\adfsplit
$\{0, 3, 27, 152, 173\}$,
$\{0, 5, 17, 132, 168\}$,
$\{0, 14, 72, 88, 142\}$,\adfsplit
$\{0, 1, 8, 23, 122\}$

}
\adfLgap \noindent by the mapping:
$x \mapsto x +  j \adfmod{208}$ for $x < 208$,
$x \mapsto (x +  j \adfmod{4}) + 208$ for $x \ge 208$,
$0 \le j < 208$.
\ADFvfyParStart{(212, ((10, 208, ((208, 1), (4, 1)))), ((16, 13), (4, 1)))} 

\adfDgap
\noindent{\boldmath $ 16^{13} 24^{1} $}~
With the point set $\{0, 1, \dots, 231\}$ partitioned into
 residue classes modulo $13$ for $\{0, 1, \dots, 207\}$, and
 $\{208, 209, \dots, 231\}$,
 the design is generated from

\adfLgap {\adfBfont
$\{210, 141, 104, 51, 138\}$,
$\{154, 18, 99, 66, 7\}$,
$\{217, 32, 121, 63, 142\}$,\adfsplit
$\{215, 41, 108, 42, 112\}$,
$\{77, 68, 199, 87, 41\}$,
$\{0, 2, 85, 99, 222\}$,\adfsplit
$\{0, 7, 22, 69, 226\}$,
$\{0, 17, 165, 183, 221\}$,
$\{0, 20, 49, 113, 154\}$,\adfsplit
$\{0, 5, 45, 68, 80\}$,
$\{0, 16, 44, 100, 151\}$,
$\{0, 6, 30, 38, 132\}$

}
\adfLgap \noindent by the mapping:
$x \mapsto x +  j \adfmod{208}$ for $x < 208$,
$x \mapsto (x - 208 + 3 j \adfmod{24}) + 208$ for $x \ge 208$,
$0 \le j < 208$.
\ADFvfyParStart{(232, ((12, 208, ((208, 1), (24, 3)))), ((16, 13), (24, 1)))} 

\adfDgap
\noindent{\boldmath $ 16^{13} 44^{1} $}~
With the point set $\{0, 1, \dots, 251\}$ partitioned into
 residue classes modulo $13$ for $\{0, 1, \dots, 207\}$, and
 $\{208, 209, \dots, 251\}$,
 the design is generated from

\adfLgap {\adfBfont
$\{230, 128, 170, 56, 195\}$,
$\{234, 86, 20, 173, 110\}$,
$\{162, 202, 208, 203, 54\}$,\adfsplit
$\{207, 78, 243, 32, 60\}$,
$\{247, 77, 70, 153, 183\}$,
$\{234, 32, 1, 89, 139\}$,\adfsplit
$\{149, 229, 105, 198, 76\}$,
$\{0, 5, 22, 103, 248\}$,
$\{0, 16, 36, 128, 160\}$,\adfsplit
$\{0, 4, 14, 51, 89\}$,
$\{0, 12, 27, 109, 225\}$,
$\{0, 8, 19, 53, 62\}$,\adfsplit
$\{0, 3, 134, 140, 221\}$,
$\{0, 2, 23, 58, 227\}$

}
\adfLgap \noindent by the mapping:
$x \mapsto x +  j \adfmod{208}$ for $x < 208$,
$x \mapsto (x - 208 + 2 j \adfmod{32}) + 208$ for $208 \le x < 240$,
$x \mapsto (x +  j \adfmod{8}) + 240$ for $240 \le x < 248$,
$x \mapsto (x +  j \adfmod{4}) + 248$ for $x \ge 248$,
$0 \le j < 208$.
\ADFvfyParStart{(252, ((14, 208, ((208, 1), (32, 2), (8, 1), (4, 1)))), ((16, 13), (44, 1)))} 

\adfDgap
\noindent{\boldmath $ 16^{14} 36^{1} $}~
With the point set $\{0, 1, \dots, 259\}$ partitioned into
 residue classes modulo $14$ for $\{0, 1, \dots, 223\}$, and
 $\{224, 225, \dots, 259\}$,
 the design is generated from

\adfLgap {\adfBfont
$\{92, 132, 226, 88, 80\}$,
$\{163, 78, 58, 233, 81\}$,
$\{99, 180, 229, 20, 80\}$,\adfsplit
$\{103, 246, 189, 55, 20\}$,
$\{47, 57, 122, 4, 26\}$,
$\{27, 60, 78, 149, 255\}$,\adfsplit
$\{0, 26, 76, 177, 241\}$,
$\{0, 7, 41, 162, 255\}$,
$\{0, 1, 108, 157, 224\}$,\adfsplit
$\{0, 9, 129, 209, 229\}$,
$\{0, 2, 13, 29, 59\}$,
$\{0, 17, 54, 131, 163\}$,\adfsplit
$\{0, 6, 45, 133, 158\}$,
$\{0, 5, 63, 99, 137\}$

}
\adfLgap \noindent by the mapping:
$x \mapsto x +  j \adfmod{224}$ for $x < 224$,
$x \mapsto (x +  j \adfmod{28}) + 224$ for $224 \le x < 252$,
$x \mapsto (x - 252 +  j \adfmod{8}) + 252$ for $x \ge 252$,
$0 \le j < 224$.
\ADFvfyParStart{(260, ((14, 224, ((224, 1), (28, 1), (8, 1)))), ((16, 14), (36, 1)))} 

\adfDgap
\noindent{\boldmath $ 16^{15} 8^{1} $}~
With the point set $\{0, 1, \dots, 247\}$ partitioned into
 residue classes modulo $15$ for $\{0, 1, \dots, 239\}$, and
 $\{240, 241, \dots, 247\}$,
 the design is generated from

\adfLgap {\adfBfont
$\{98, 145, 213, 161, 7\}$,
$\{240, 239, 6, 229, 2\}$,
$\{89, 158, 201, 137, 91\}$,\adfsplit
$\{225, 41, 70, 121, 192\}$,
$\{209, 177, 53, 227, 221\}$,
$\{0, 1, 54, 77, 245\}$,\adfsplit
$\{0, 25, 61, 119, 159\}$,
$\{0, 8, 49, 108, 145\}$,
$\{0, 26, 57, 99, 127\}$,\adfsplit
$\{0, 11, 35, 158, 213\}$,
$\{0, 19, 39, 126, 148\}$,
$\{0, 5, 14, 79, 157\}$

}
\adfLgap \noindent by the mapping:
$x \mapsto x +  j \adfmod{240}$ for $x < 240$,
$x \mapsto (x +  j \adfmod{8}) + 240$ for $x \ge 240$,
$0 \le j < 240$.
\ADFvfyParStart{(248, ((12, 240, ((240, 1), (8, 1)))), ((16, 15), (8, 1)))} 

\adfDgap
\noindent{\boldmath $ 16^{15} 28^{1} $}~
With the point set $\{0, 1, \dots, 267\}$ partitioned into
 residue classes modulo $15$ for $\{0, 1, \dots, 239\}$, and
 $\{240, 241, \dots, 267\}$,
 the design is generated from

\adfLgap {\adfBfont
$\{140, 0, 67, 63, 149\}$,
$\{61, 132, 163, 216, 80\}$,
$\{151, 219, 124, 130, 254\}$,\adfsplit
$\{28, 76, 154, 228, 192\}$,
$\{47, 222, 217, 266, 180\}$,
$\{81, 257, 103, 48, 147\}$,\adfsplit
$\{0, 1, 109, 111, 260\}$,
$\{0, 8, 64, 80, 123\}$,
$\{0, 10, 34, 128, 153\}$,\adfsplit
$\{0, 14, 32, 49, 193\}$,
$\{0, 7, 134, 227, 240\}$,
$\{0, 12, 171, 194, 241\}$,\adfsplit
$\{0, 3, 186, 214, 254\}$,
$\{0, 11, 50, 148, 189\}$

}
\adfLgap \noindent by the mapping:
$x \mapsto x +  j \adfmod{240}$ for $x < 240$,
$x \mapsto (x +  j \adfmod{16}) + 240$ for $240 \le x < 256$,
$x \mapsto (x +  j \adfmod{8}) + 256$ for $256 \le x < 264$,
$x \mapsto (x +  j \adfmod{4}) + 264$ for $x \ge 264$,
$0 \le j < 240$.
\ADFvfyParStart{(268, ((14, 240, ((240, 1), (16, 1), (8, 1), (4, 1)))), ((16, 15), (28, 1)))} 

\adfDgap
\noindent{\boldmath $ 16^{16} 20^{1} $}~
With the point set $\{0, 1, \dots, 275\}$ partitioned into
 residue classes modulo $16$ for $\{0, 1, \dots, 255\}$, and
 $\{256, 257, \dots, 275\}$,
 the design is generated from

\adfLgap {\adfBfont
$\{274, 87, 248, 194, 233\}$,
$\{269, 45, 57, 1, 190\}$,
$\{264, 231, 113, 210, 160\}$,\adfsplit
$\{264, 107, 67, 149, 94\}$,
$\{260, 18, 137, 214, 139\}$,
$\{41, 197, 216, 127, 65\}$,\adfsplit
$\{110, 136, 133, 144, 201\}$,
$\{0, 1, 38, 141, 158\}$,
$\{0, 14, 66, 88, 197\}$,\adfsplit
$\{0, 5, 41, 84, 171\}$,
$\{0, 4, 29, 35, 143\}$,
$\{0, 9, 72, 92, 102\}$,\adfsplit
$\{0, 28, 61, 106, 155\}$,
$\{0, 7, 53, 187, 205\}$

}
\adfLgap \noindent by the mapping:
$x \mapsto x +  j \adfmod{256}$ for $x < 256$,
$x \mapsto (x +  j \adfmod{16}) + 256$ for $256 \le x < 272$,
$x \mapsto (x +  j \adfmod{4}) + 272$ for $x \ge 272$,
$0 \le j < 256$.
\ADFvfyParStart{(276, ((14, 256, ((256, 1), (16, 1), (4, 1)))), ((16, 16), (20, 1)))} 

\adfDgap
\noindent{\boldmath $ 16^{16} 40^{1} $}~
With the point set $\{0, 1, \dots, 295\}$ partitioned into
 residue classes modulo $16$ for $\{0, 1, \dots, 255\}$, and
 $\{256, 257, \dots, 295\}$,
 the design is generated from

\adfLgap {\adfBfont
$\{294, 249, 2, 55, 84\}$,
$\{288, 186, 232, 127, 149\}$,
$\{272, 199, 90, 238, 125\}$,\adfsplit
$\{282, 211, 48, 207, 73\}$,
$\{277, 168, 102, 231, 145\}$,
$\{287, 227, 121, 95, 40\}$,\adfsplit
$\{263, 40, 6, 147, 226\}$,
$\{275, 185, 54, 152, 193\}$,
$\{0, 5, 47, 243, 264\}$,\adfsplit
$\{0, 6, 78, 146, 286\}$,
$\{0, 1, 20, 120, 205\}$,
$\{0, 15, 73, 103, 157\}$,\adfsplit
$\{0, 11, 49, 76, 211\}$,
$\{0, 10, 50, 67, 111\}$,
$\{0, 2, 14, 89, 166\}$,\adfsplit
$\{0, 3, 24, 31, 126\}$

}
\adfLgap \noindent by the mapping:
$x \mapsto x +  j \adfmod{256}$ for $x < 256$,
$x \mapsto (x +  j \adfmod{32}) + 256$ for $256 \le x < 288$,
$x \mapsto (x +  j \adfmod{8}) + 288$ for $x \ge 288$,
$0 \le j < 256$.
\ADFvfyParStart{(296, ((16, 256, ((256, 1), (32, 1), (8, 1)))), ((16, 16), (40, 1)))} 

\adfDgap
\noindent{\boldmath $ 16^{17} 12^{1} $}~
With the point set $\{0, 1, \dots, 283\}$ partitioned into
 residue classes modulo $17$ for $\{0, 1, \dots, 271\}$, and
 $\{272, 273, \dots, 283\}$,
 the design is generated from

\adfLgap {\adfBfont
$\{280, 124, 185, 215, 250\}$,
$\{278, 121, 82, 39, 75\}$,
$\{276, 51, 236, 82, 166\}$,\adfsplit
$\{132, 78, 5, 174, 215\}$,
$\{193, 243, 26, 49, 191\}$,
$\{90, 117, 180, 268, 109\}$,\adfsplit
$\{0, 1, 4, 15, 156\}$,
$\{0, 6, 16, 53, 164\}$,
$\{0, 21, 59, 81, 160\}$,\adfsplit
$\{0, 9, 67, 95, 143\}$,
$\{0, 25, 97, 123, 172\}$,
$\{0, 5, 29, 109, 208\}$,\adfsplit
$\{0, 18, 74, 140, 180\}$,
$\{0, 12, 32, 45, 89\}$

}
\adfLgap \noindent by the mapping:
$x \mapsto x +  j \adfmod{272}$ for $x < 272$,
$x \mapsto (x +  j \adfmod{8}) + 272$ for $272 \le x < 280$,
$x \mapsto (x +  j \adfmod{4}) + 280$ for $x \ge 280$,
$0 \le j < 272$.
\ADFvfyParStart{(284, ((14, 272, ((272, 1), (8, 1), (4, 1)))), ((16, 17), (12, 1)))} 

\adfDgap
\noindent{\boldmath $ 17^{8} 13^{1} $}~
With the point set $\{0, 1, \dots, 148\}$ partitioned into
 residue classes modulo $8$ for $\{0, 1, \dots, 135\}$, and
 $\{136, 137, \dots, 148\}$,
 the design is generated from

\adfLgap {\adfBfont
$\{16, 60, 71, 83, 34\}$,
$\{29, 36, 94, 136, 122\}$,
$\{0, 1, 84, 98, 137\}$,\adfsplit
$\{0, 2, 31, 35, 77\}$,
$\{0, 5, 20, 41, 111\}$,
$\{0, 10, 27, 89, 144\}$,\adfsplit
$\{0, 3, 9, 22, 85\}$,
$\{0, 34, 68, 102, 148\}$

}
\adfLgap \noindent by the mapping:
$x \mapsto x +  j \adfmod{136}$ for $x < 136$,
$x \mapsto (x +  j \adfmod{8}) + 136$ for $136 \le x < 144$,
$x \mapsto (x +  j \adfmod{4}) + 144$ for $144 \le x < 148$,
$148 \mapsto 148$,
$0 \le j < 136$
 for the first seven blocks,
$0 \le j < 34$
 for the last block.
\ADFvfyParStart{(149, ((7, 136, ((136, 1), (8, 1), (4, 1), (1, 1))), (1, 34, ((136, 1), (8, 1), (4, 1), (1, 1)))), ((17, 8), (13, 1)))} 

\adfDgap
\noindent{\boldmath $ 17^{8} 33^{1} $}~
With the point set $\{0, 1, \dots, 168\}$ partitioned into
 residue classes modulo $8$ for $\{0, 1, \dots, 135\}$, and
 $\{136, 137, \dots, 168\}$,
 the design is generated from

\adfLgap {\adfBfont
$\{160, 126, 98, 97, 123\}$,
$\{80, 75, 58, 45, 151\}$,
$\{93, 27, 132, 111, 141\}$,\adfsplit
$\{0, 2, 51, 61, 167\}$,
$\{0, 12, 45, 95, 160\}$,
$\{0, 4, 11, 82, 149\}$,\adfsplit
$\{0, 15, 89, 109, 142\}$,
$\{0, 6, 43, 79, 98\}$,
$\{0, 9, 76, 122, 166\}$,\adfsplit
$\{0, 34, 68, 102, 168\}$

}
\adfLgap \noindent by the mapping:
$x \mapsto x +  j \adfmod{136}$ for $x < 136$,
$x \mapsto (x - 136 + 4 j \adfmod{32}) + 136$ for $136 \le x < 168$,
$168 \mapsto 168$,
$0 \le j < 136$
 for the first nine blocks,
$0 \le j < 34$
 for the last block.
\ADFvfyParStart{(169, ((9, 136, ((136, 1), (32, 4), (1, 1))), (1, 34, ((136, 1), (32, 4), (1, 1)))), ((17, 8), (33, 1)))} 

\adfDgap
\noindent{\boldmath $ 17^{12} 9^{1} $}~
With the point set $\{0, 1, \dots, 212\}$ partitioned into
 residue classes modulo $12$ for $\{0, 1, \dots, 203\}$, and
 $\{204, 205, \dots, 212\}$,
 the design is generated from

\adfLgap {\adfBfont
$\{1, 60, 100, 21, 56\}$,
$\{135, 193, 4, 58, 198\}$,
$\{27, 28, 164, 25, 139\}$,\adfsplit
$\{0, 6, 13, 187, 204\}$,
$\{0, 21, 47, 154, 205\}$,
$\{0, 14, 45, 130, 148\}$,\adfsplit
$\{0, 16, 38, 138, 167\}$,
$\{0, 9, 41, 87, 98\}$,
$\{0, 19, 52, 80, 161\}$,\adfsplit
$\{0, 8, 42, 91, 118\}$,
$\{0, 51, 102, 153, 212\}$

}
\adfLgap \noindent by the mapping:
$x \mapsto x +  j \adfmod{204}$ for $x < 204$,
$x \mapsto (x - 204 + 2 j \adfmod{8}) + 204$ for $204 \le x < 212$,
$212 \mapsto 212$,
$0 \le j < 204$
 for the first ten blocks,
$0 \le j < 51$
 for the last block.
\ADFvfyParStart{(213, ((10, 204, ((204, 1), (8, 2), (1, 1))), (1, 51, ((204, 1), (8, 2), (1, 1)))), ((17, 12), (9, 1)))} 

\adfDgap
\noindent{\boldmath $ 17^{12} 29^{1} $}~
With the point set $\{0, 1, \dots, 232\}$ partitioned into
 residue classes modulo $12$ for $\{0, 1, \dots, 203\}$, and
 $\{204, 205, \dots, 232\}$,
 the design is generated from

\adfLgap {\adfBfont
$\{42, 170, 151, 184, 222\}$,
$\{110, 120, 30, 148, 33\}$,
$\{205, 105, 187, 40, 128\}$,\adfsplit
$\{66, 159, 132, 19, 58\}$,
$\{230, 193, 20, 83, 118\}$,
$\{0, 1, 6, 79, 164\}$,\adfsplit
$\{0, 15, 49, 107, 149\}$,
$\{0, 7, 44, 61, 219\}$,
$\{0, 13, 56, 81, 225\}$,\adfsplit
$\{0, 4, 26, 71, 224\}$,
$\{0, 2, 11, 32, 186\}$,
$\{0, 16, 69, 121, 206\}$,\adfsplit
$\{0, 51, 102, 153, 232\}$

}
\adfLgap \noindent by the mapping:
$x \mapsto x +  j \adfmod{204}$ for $x < 204$,
$x \mapsto (x - 204 + 2 j \adfmod{24}) + 204$ for $204 \le x < 228$,
$x \mapsto (x +  j \adfmod{4}) + 228$ for $228 \le x < 232$,
$232 \mapsto 232$,
$0 \le j < 204$
 for the first 12 blocks,
$0 \le j < 51$
 for the last block.
\ADFvfyParStart{(233, ((12, 204, ((204, 1), (24, 2), (4, 1), (1, 1))), (1, 51, ((204, 1), (24, 2), (4, 1), (1, 1)))), ((17, 12), (29, 1)))} 

\adfDgap
\noindent{\boldmath $ 17^{12} 49^{1} $}~
With the point set $\{0, 1, \dots, 252\}$ partitioned into
 residue classes modulo $12$ for $\{0, 1, \dots, 203\}$, and
 $\{204, 205, \dots, 252\}$,
 the design is generated from

\adfLgap {\adfBfont
$\{217, 113, 190, 93, 162\}$,
$\{114, 152, 243, 178, 147\}$,
$\{195, 95, 168, 89, 216\}$,\adfsplit
$\{92, 180, 237, 79, 18\}$,
$\{223, 122, 35, 192, 32\}$,
$\{36, 1, 164, 18, 40\}$,\adfsplit
$\{0, 7, 37, 119, 230\}$,
$\{0, 1, 53, 189, 221\}$,
$\{0, 9, 19, 142, 235\}$,\adfsplit
$\{0, 25, 75, 164, 222\}$,
$\{0, 14, 113, 172, 214\}$,
$\{0, 8, 29, 63, 86\}$,\adfsplit
$\{0, 2, 45, 111, 248\}$,
$\{0, 11, 94, 148, 224\}$,
$\{0, 51, 102, 153, 252\}$

}
\adfLgap \noindent by the mapping:
$x \mapsto x +  j \adfmod{204}$ for $x < 204$,
$x \mapsto (x - 204 + 4 j \adfmod{48}) + 204$ for $204 \le x < 252$,
$252 \mapsto 252$,
$0 \le j < 204$
 for the first 14 blocks,
$0 \le j < 51$
 for the last block.
\ADFvfyParStart{(253, ((14, 204, ((204, 1), (48, 4), (1, 1))), (1, 51, ((204, 1), (48, 4), (1, 1)))), ((17, 12), (49, 1)))} 

\adfDgap
\noindent{\boldmath $ 17^{16} 5^{1} $}~
With the point set $\{0, 1, \dots, 276\}$ partitioned into
 residue classes modulo $16$ for $\{0, 1, \dots, 271\}$, and
 $\{272, 273, 274, 275, 276\}$,
 the design is generated from

\adfLgap {\adfBfont
$\{205, 145, 41, 174, 122\}$,
$\{242, 172, 52, 144, 6\}$,
$\{91, 167, 30, 12, 65\}$,\adfsplit
$\{74, 40, 87, 205, 117\}$,
$\{145, 160, 187, 201, 138\}$,
$\{0, 9, 20, 110, 181\}$,\adfsplit
$\{0, 2, 87, 145, 272\}$,
$\{0, 6, 39, 51, 105\}$,
$\{0, 4, 44, 113, 130\}$,\adfsplit
$\{0, 10, 65, 84, 179\}$,
$\{0, 1, 25, 150, 200\}$,
$\{0, 3, 8, 124, 213\}$,\adfsplit
$\{0, 21, 78, 115, 153\}$,
$\{0, 68, 136, 204, 276\}$

}
\adfLgap \noindent by the mapping:
$x \mapsto x +  j \adfmod{272}$ for $x < 272$,
$x \mapsto (x +  j \adfmod{4}) + 272$ for $272 \le x < 276$,
$276 \mapsto 276$,
$0 \le j < 272$
 for the first 13 blocks,
$0 \le j < 68$
 for the last block.
\ADFvfyParStart{(277, ((13, 272, ((272, 1), (4, 1), (1, 1))), (1, 68, ((272, 1), (4, 1), (1, 1)))), ((17, 16), (5, 1)))} 

\adfDgap
\noindent{\boldmath $ 20^{8} 40^{1} $}~
With the point set $\{0, 1, \dots, 199\}$ partitioned into
 residue classes modulo $8$ for $\{0, 1, \dots, 159\}$, and
 $\{160, 161, \dots, 199\}$,
 the design is generated from

\adfLgap {\adfBfont
$\{103, 29, 18, 52, 78\}$,
$\{27, 46, 189, 128, 13\}$,
$\{109, 26, 168, 20, 11\}$,\adfsplit
$\{71, 147, 66, 8, 161\}$,
$\{0, 1, 44, 143, 196\}$,
$\{0, 12, 105, 125, 189\}$,\adfsplit
$\{0, 4, 91, 133, 198\}$,
$\{0, 2, 30, 52, 180\}$,
$\{0, 7, 36, 46, 167\}$,\adfsplit
$\{0, 3, 41, 95, 166\}$,
$\{0, 13, 70, 107, 199\}$

}
\adfLgap \noindent by the mapping:
$x \mapsto x +  j \adfmod{160}$ for $x < 160$,
$x \mapsto (x +  j \adfmod{40}) + 160$ for $x \ge 160$,
$0 \le j < 160$.
\ADFvfyParStart{(200, ((11, 160, ((160, 1), (40, 1)))), ((20, 8), (40, 1)))} 

\adfDgap
\noindent{\boldmath $ 20^{9} 40^{1} $}~
With the point set $\{0, 1, \dots, 219\}$ partitioned into
 residue classes modulo $9$ for $\{0, 1, \dots, 179\}$, and
 $\{180, 181, \dots, 219\}$,
 the design is generated from

\adfLgap {\adfBfont
$\{208, 7, 2, 168, 111\}$,
$\{64, 16, 207, 114, 107\}$,
$\{191, 57, 140, 115, 91\}$,\adfsplit
$\{74, 209, 82, 7, 76\}$,
$\{157, 170, 105, 183, 128\}$,
$\{0, 22, 68, 141, 187\}$,\adfsplit
$\{0, 20, 53, 84, 140\}$,
$\{0, 15, 70, 100, 121\}$,
$\{0, 10, 26, 88, 200\}$,\adfsplit
$\{0, 1, 4, 143, 218\}$,
$\{0, 11, 77, 163, 196\}$,
$\{0, 12, 47, 148, 182\}$

}
\adfLgap \noindent by the mapping:
$x \mapsto x +  j \adfmod{180}$ for $x < 180$,
$x \mapsto (x - 180 + 2 j \adfmod{40}) + 180$ for $x \ge 180$,
$0 \le j < 180$.
\ADFvfyParStart{(220, ((12, 180, ((180, 1), (40, 2)))), ((20, 9), (40, 1)))} 

\adfDgap
\noindent{\boldmath $ 20^{10} 36^{1} $}~
With the point set $\{0, 1, \dots, 235\}$ partitioned into
 residue classes modulo $9$ for $\{0, 1, \dots, 179\}$,
 $\{180, 181, \dots, 199\}$, and
 $\{200, 201, \dots, 235\}$,
 the design is generated from

\adfLgap {\adfBfont
$\{30, 181, 74, 177, 99\}$,
$\{85, 54, 70, 158, 119\}$,
$\{135, 134, 235, 40, 64\}$,\adfsplit
$\{197, 34, 32, 233, 87\}$,
$\{199, 70, 225, 59, 13\}$,
$\{0, 3, 7, 13, 181\}$,\adfsplit
$\{0, 5, 22, 197, 220\}$,
$\{0, 23, 137, 196, 233\}$,
$\{0, 14, 35, 93, 119\}$,\adfsplit
$\{0, 19, 60, 116, 213\}$,
$\{0, 8, 121, 151, 229\}$,
$\{0, 40, 91, 138, 200\}$,\adfsplit
$\{0, 20, 48, 148, 223\}$,
$\{0, 12, 74, 142, 226\}$

}
\adfLgap \noindent by the mapping:
$x \mapsto x +  j \adfmod{180}$ for $x < 180$,
$x \mapsto (x +  j \adfmod{20}) + 180$ for $180 \le x < 200$,
$x \mapsto (x - 200 +  j \adfmod{36}) + 200$ for $x \ge 200$,
$0 \le j < 180$.
\ADFvfyParStart{(236, ((14, 180, ((180, 1), (20, 1), (36, 1)))), ((20, 9), (20, 1), (36, 1)))} 

\adfDgap
\noindent{\boldmath $ 20^{10} 40^{1} $}~
With the point set $\{0, 1, \dots, 239\}$ partitioned into
 residue classes modulo $10$ for $\{0, 1, \dots, 199\}$, and
 $\{200, 201, \dots, 239\}$,
 the design is generated from

\adfLgap {\adfBfont
$\{156, 194, 172, 208, 11\}$,
$\{103, 104, 137, 168, 89\}$,
$\{186, 202, 130, 52, 89\}$,\adfsplit
$\{47, 173, 212, 122, 54\}$,
$\{175, 3, 90, 151, 102\}$,
$\{0, 2, 45, 108, 114\}$,\adfsplit
$\{0, 19, 42, 147, 228\}$,
$\{0, 18, 116, 141, 221\}$,
$\{0, 3, 47, 171, 218\}$,\adfsplit
$\{0, 11, 118, 154, 217\}$,
$\{0, 8, 62, 173, 202\}$,
$\{0, 5, 26, 109, 213\}$,\adfsplit
$\{0, 4, 13, 71, 235\}$

}
\adfLgap \noindent by the mapping:
$x \mapsto x +  j \adfmod{200}$ for $x < 200$,
$x \mapsto (x +  j \adfmod{40}) + 200$ for $x \ge 200$,
$0 \le j < 200$.
\ADFvfyParStart{(240, ((13, 200, ((200, 1), (40, 1)))), ((20, 10), (40, 1)))} 

\adfDgap
\noindent{\boldmath $ 20^{11} 40^{1} $}~
With the point set $\{0, 1, \dots, 259\}$ partitioned into
 residue classes modulo $11$ for $\{0, 1, \dots, 219\}$, and
 $\{220, 221, \dots, 259\}$,
 the design is generated from

\adfLgap {\adfBfont
$\{64, 183, 200, 118, 225\}$,
$\{138, 69, 180, 89, 76\}$,
$\{121, 242, 27, 78, 113\}$,\adfsplit
$\{166, 141, 18, 226, 48\}$,
$\{162, 91, 199, 31, 109\}$,
$\{136, 133, 228, 137, 105\}$,\adfsplit
$\{0, 2, 47, 164, 252\}$,
$\{0, 23, 57, 182, 226\}$,
$\{0, 6, 46, 70, 146\}$,\adfsplit
$\{0, 10, 26, 85, 243\}$,
$\{0, 15, 96, 137, 239\}$,
$\{0, 21, 128, 191, 237\}$,\adfsplit
$\{0, 5, 19, 152, 251\}$,
$\{0, 9, 36, 48, 115\}$

}
\adfLgap \noindent by the mapping:
$x \mapsto x +  j \adfmod{220}$ for $x < 220$,
$x \mapsto (x - 220 + 2 j \adfmod{40}) + 220$ for $x \ge 220$,
$0 \le j < 220$.
\ADFvfyParStart{(260, ((14, 220, ((220, 1), (40, 2)))), ((20, 11), (40, 1)))} 

\adfDgap
\noindent{\boldmath $ 21^{8} 9^{1} $}~
With the point set $\{0, 1, \dots, 176\}$ partitioned into
 residue classes modulo $8$ for $\{0, 1, \dots, 167\}$, and
 $\{168, 169, \dots, 176\}$,
 the design is generated from

\adfLgap {\adfBfont
$\{51, 90, 4, 158, 141\}$,
$\{150, 51, 56, 171, 106\}$,
$\{0, 1, 3, 21, 175\}$,\adfsplit
$\{0, 9, 22, 37, 71\}$,
$\{0, 6, 29, 65, 108\}$,
$\{0, 7, 19, 52, 77\}$,\adfsplit
$\{0, 11, 38, 92, 133\}$,
$\{0, 4, 30, 105, 115\}$,
$\{0, 42, 84, 126, 176\}$

}
\adfLgap \noindent by the mapping:
$x \mapsto x +  j \adfmod{168}$ for $x < 168$,
$x \mapsto (x +  j \adfmod{8}) + 168$ for $168 \le x < 176$,
$176 \mapsto 176$,
$0 \le j < 168$
 for the first eight blocks,
$0 \le j < 42$
 for the last block.
\ADFvfyParStart{(177, ((8, 168, ((168, 1), (8, 1), (1, 1))), (1, 42, ((168, 1), (8, 1), (1, 1)))), ((21, 8), (9, 1)))} 

\adfDgap
\noindent{\boldmath $ 21^{8} 29^{1} $}~
With the point set $\{0, 1, \dots, 196\}$ partitioned into
 residue classes modulo $8$ for $\{0, 1, \dots, 167\}$, and
 $\{168, 169, \dots, 196\}$,
 the design is generated from

\adfLgap {\adfBfont
$\{103, 33, 76, 91, 0\}$,
$\{94, 129, 155, 56, 77\}$,
$\{43, 93, 100, 56, 186\}$,\adfsplit
$\{0, 1, 3, 140, 149\}$,
$\{0, 4, 55, 66, 172\}$,
$\{0, 10, 46, 85, 185\}$,\adfsplit
$\{0, 18, 81, 115, 179\}$,
$\{0, 6, 47, 114, 193\}$,
$\{0, 14, 100, 123, 180\}$,\adfsplit
$\{0, 5, 30, 79, 183\}$,
$\{0, 42, 84, 126, 196\}$

}
\adfLgap \noindent by the mapping:
$x \mapsto x +  j \adfmod{168}$ for $x < 168$,
$x \mapsto (x +  j \adfmod{28}) + 168$ for $168 \le x < 196$,
$196 \mapsto 196$,
$0 \le j < 168$
 for the first ten blocks,
$0 \le j < 42$
 for the last block.
\ADFvfyParStart{(197, ((10, 168, ((168, 1), (28, 1), (1, 1))), (1, 42, ((168, 1), (28, 1), (1, 1)))), ((21, 8), (29, 1)))} 

\adfDgap
\noindent{\boldmath $ 21^{12} 17^{1} $}~
With the point set $\{0, 1, \dots, 268\}$ partitioned into
 residue classes modulo $12$ for $\{0, 1, \dots, 251\}$, and
 $\{252, 253, \dots, 268\}$,
 the design is generated from

\adfLgap {\adfBfont
$\{263, 44, 4, 183, 189\}$,
$\{245, 24, 242, 85, 124\}$,
$\{219, 212, 184, 17, 31\}$,\adfsplit
$\{214, 64, 102, 189, 219\}$,
$\{35, 31, 101, 234, 205\}$,
$\{0, 43, 94, 141, 264\}$,\adfsplit
$\{0, 11, 26, 177, 236\}$,
$\{0, 17, 54, 190, 262\}$,
$\{0, 8, 88, 146, 169\}$,\adfsplit
$\{0, 1, 19, 21, 124\}$,
$\{0, 13, 45, 89, 122\}$,
$\{0, 9, 55, 65, 258\}$,\adfsplit
$\{0, 22, 74, 115, 184\}$,
$\{0, 63, 126, 189, 268\}$

}
\adfLgap \noindent by the mapping:
$x \mapsto x +  j \adfmod{252}$ for $x < 252$,
$x \mapsto (x +  j \adfmod{12}) + 252$ for $252 \le x < 264$,
$x \mapsto (x +  j \adfmod{4}) + 264$ for $264 \le x < 268$,
$268 \mapsto 268$,
$0 \le j < 252$
 for the first 13 blocks,
$0 \le j < 63$
 for the last block.
\ADFvfyParStart{(269, ((13, 252, ((252, 1), (12, 1), (4, 1), (1, 1))), (1, 63, ((252, 1), (12, 1), (4, 1), (1, 1)))), ((21, 12), (17, 1)))} 

\adfDgap
\noindent{\boldmath $ 23^{8} 7^{1} $}~
With the point set $\{0, 1, \dots, 190\}$ partitioned into
 residue classes modulo $7$ for $\{0, 1, \dots, 160\}$,
 $\{161, 162, \dots, 183\}$, and
 $\{184, 185, \dots, 190\}$,
 the design is generated from

\adfLgap {\adfBfont
$\{26, 114, 126, 6, 185\}$,
$\{185, 85, 88, 165, 10\}$,
$\{56, 34, 19, 73, 74\}$,\adfsplit
$\{0, 2, 6, 11, 115\}$,
$\{0, 10, 33, 69, 95\}$,
$\{0, 19, 44, 116, 180\}$,\adfsplit
$\{0, 13, 93, 123, 166\}$,
$\{0, 29, 96, 130, 177\}$,
$\{0, 16, 87, 134, 163\}$,\adfsplit
$\{0, 8, 32, 111, 183\}$

}
\adfLgap \noindent by the mapping:
$x \mapsto x +  j \adfmod{161}$ for $x < 161$,
$x \mapsto (x +  j \adfmod{23}) + 161$ for $161 \le x < 184$,
$x \mapsto (x - 184 +  j \adfmod{7}) + 184$ for $x \ge 184$,
$0 \le j < 161$.
\ADFvfyParStart{(191, ((10, 161, ((161, 1), (23, 1), (7, 1)))), ((23, 7), (23, 1), (7, 1)))} 

\adfDgap
\noindent{\boldmath $ 24^{6} 20^{1} $}~
With the point set $\{0, 1, \dots, 163\}$ partitioned into
 residue classes modulo $6$ for $\{0, 1, \dots, 143\}$, and
 $\{144, 145, \dots, 163\}$,
 the design is generated from

\adfLgap {\adfBfont
$\{26, 82, 11, 156, 78\}$,
$\{112, 163, 105, 7, 50\}$,
$\{0, 1, 3, 38, 118\}$,\adfsplit
$\{0, 9, 25, 112, 122\}$,
$\{0, 5, 13, 63, 74\}$,
$\{0, 21, 49, 100, 148\}$,\adfsplit
$\{0, 14, 33, 124, 151\}$,
$\{0, 17, 40, 85, 157\}$

}
\adfLgap \noindent by the mapping:
$x \mapsto x +  j \adfmod{144}$ for $x < 144$,
$x \mapsto (x +  j \adfmod{16}) + 144$ for $144 \le x < 160$,
$x \mapsto (x +  j \adfmod{4}) + 160$ for $x \ge 160$,
$0 \le j < 144$.
\ADFvfyParStart{(164, ((8, 144, ((144, 1), (16, 1), (4, 1)))), ((24, 6), (20, 1)))} 

\adfDgap
\noindent{\boldmath $ 24^{7} 8^{1} $}~
With the point set $\{0, 1, \dots, 175\}$ partitioned into
 residue classes modulo $7$ for $\{0, 1, \dots, 167\}$, and
 $\{168, 169, \dots, 175\}$,
 the design is generated from

\adfLgap {\adfBfont
$\{135, 1, 159, 70, 81\}$,
$\{13, 63, 15, 54, 32\}$,
$\{159, 28, 29, 3, 114\}$,\adfsplit
$\{107, 162, 91, 87, 55\}$,
$\{127, 17, 12, 173, 104\}$,
$\{115, 161, 55, 88, 155\}$,\adfsplit
$\{90, 16, 24, 120, 133\}$,
$\{0, 3, 18, 47, 170\}$

}
\adfLgap \noindent by the mapping:
$x \mapsto x +  j \adfmod{168}$ for $x < 168$,
$x \mapsto (x +  j \adfmod{8}) + 168$ for $x \ge 168$,
$0 \le j < 168$.
\ADFvfyParStart{(176, ((8, 168, ((168, 1), (8, 1)))), ((24, 7), (8, 1)))} 

\adfDgap
\noindent{\boldmath $ 24^{7} 28^{1} $}~
With the point set $\{0, 1, \dots, 195\}$ partitioned into
 residue classes modulo $7$ for $\{0, 1, \dots, 167\}$, and
 $\{168, 169, \dots, 195\}$,
 the design is generated from

\adfLgap {\adfBfont
$\{25, 23, 54, 80, 140\}$,
$\{114, 132, 181, 136, 18\}$,
$\{48, 133, 93, 136, 127\}$,\adfsplit
$\{0, 1, 62, 67, 75\}$,
$\{0, 15, 39, 59, 171\}$,
$\{0, 10, 58, 131, 192\}$,\adfsplit
$\{0, 25, 90, 122, 180\}$,
$\{0, 12, 76, 99, 187\}$,
$\{0, 11, 38, 138, 176\}$,\adfsplit
$\{0, 16, 132, 151, 169\}$

}
\adfLgap \noindent by the mapping:
$x \mapsto x +  j \adfmod{168}$ for $x < 168$,
$x \mapsto (x +  j \adfmod{28}) + 168$ for $x \ge 168$,
$0 \le j < 168$.
\ADFvfyParStart{(196, ((10, 168, ((168, 1), (28, 1)))), ((24, 7), (28, 1)))} 

\adfDgap
\noindent{\boldmath $ 24^{8} 16^{1} $}~
With the point set $\{0, 1, \dots, 207\}$ partitioned into
 residue classes modulo $8$ for $\{0, 1, \dots, 191\}$, and
 $\{192, 193, \dots, 207\}$,
 the design is generated from

\adfLgap {\adfBfont
$\{148, 196, 26, 35, 23\}$,
$\{17, 163, 10, 108, 167\}$,
$\{32, 73, 53, 54, 163\}$,\adfsplit
$\{0, 2, 60, 65, 75\}$,
$\{0, 13, 43, 121, 158\}$,
$\{0, 28, 85, 154, 195\}$,\adfsplit
$\{0, 23, 52, 139, 207\}$,
$\{0, 25, 51, 143, 197\}$,
$\{0, 6, 17, 103, 165\}$,\adfsplit
$\{0, 14, 45, 138, 156\}$

}
\adfLgap \noindent by the mapping:
$x \mapsto x +  j \adfmod{192}$ for $x < 192$,
$x \mapsto (x +  j \adfmod{16}) + 192$ for $x \ge 192$,
$0 \le j < 192$.
\ADFvfyParStart{(208, ((10, 192, ((192, 1), (16, 1)))), ((24, 8), (16, 1)))} 

\adfDgap
\noindent{\boldmath $ 24^{8} 36^{1} $}~
With the point set $\{0, 1, \dots, 227\}$ partitioned into
 residue classes modulo $8$ for $\{0, 1, \dots, 191\}$, and
 $\{192, 193, \dots, 227\}$,
 the design is generated from

\adfLgap {\adfBfont
$\{132, 183, 69, 206, 14\}$,
$\{52, 205, 106, 112, 102\}$,
$\{82, 64, 131, 166, 93\}$,\adfsplit
$\{3, 169, 88, 30, 206\}$,
$\{187, 143, 37, 84, 184\}$,
$\{0, 5, 62, 71, 224\}$,\adfsplit
$\{0, 2, 95, 164, 222\}$,
$\{0, 1, 14, 172, 212\}$,
$\{0, 25, 77, 156, 206\}$,\adfsplit
$\{0, 15, 46, 83, 116\}$,
$\{0, 17, 127, 170, 204\}$,
$\{0, 7, 105, 180, 216\}$

}
\adfLgap \noindent by the mapping:
$x \mapsto x +  j \adfmod{192}$ for $x < 192$,
$x \mapsto (x +  j \adfmod{32}) + 192$ for $192 \le x < 224$,
$x \mapsto (x +  j \adfmod{4}) + 224$ for $x \ge 224$,
$0 \le j < 192$.
\ADFvfyParStart{(228, ((12, 192, ((192, 1), (32, 1), (4, 1)))), ((24, 8), (36, 1)))} 

\adfDgap
\noindent{\boldmath $ 24^{9} 4^{1} $}~
With the point set $\{0, 1, \dots, 219\}$ partitioned into
 residue classes modulo $9$ for $\{0, 1, \dots, 215\}$, and
 $\{216, 217, 218, 219\}$,
 the design is generated from

\adfLgap {\adfBfont
$\{216, 87, 44, 46, 61\}$,
$\{186, 30, 167, 179, 99\}$,
$\{96, 19, 83, 144, 138\}$,\adfsplit
$\{0, 1, 4, 146, 186\}$,
$\{0, 8, 86, 110, 124\}$,
$\{0, 10, 56, 115, 150\}$,\adfsplit
$\{0, 21, 50, 103, 154\}$,
$\{0, 11, 33, 58, 131\}$,
$\{0, 5, 28, 44, 93\}$,\adfsplit
$\{0, 20, 57, 109, 141\}$

}
\adfLgap \noindent by the mapping:
$x \mapsto x +  j \adfmod{216}$ for $x < 216$,
$x \mapsto (x +  j \adfmod{4}) + 216$ for $x \ge 216$,
$0 \le j < 216$.
\ADFvfyParStart{(220, ((10, 216, ((216, 1), (4, 1)))), ((24, 9), (4, 1)))} 

\adfDgap
\noindent{\boldmath $ 24^{9} 44^{1} $}~
With the point set $\{0, 1, \dots, 259\}$ partitioned into
 residue classes modulo $9$ for $\{0, 1, \dots, 215\}$, and
 $\{216, 217, \dots, 259\}$,
 the design is generated from

\adfLgap {\adfBfont
$\{71, 158, 242, 129, 109\}$,
$\{171, 218, 194, 28, 62\}$,
$\{231, 54, 47, 148, 158\}$,\adfsplit
$\{18, 58, 192, 124, 136\}$,
$\{136, 33, 184, 57, 182\}$,
$\{91, 168, 221, 97, 152\}$,\adfsplit
$\{0, 3, 11, 131, 157\}$,
$\{0, 1, 15, 52, 253\}$,
$\{0, 19, 41, 188, 259\}$,\adfsplit
$\{0, 30, 83, 116, 218\}$,
$\{0, 5, 44, 119, 229\}$,
$\{0, 25, 60, 184, 235\}$,\adfsplit
$\{0, 21, 64, 95, 248\}$,
$\{0, 4, 80, 203, 246\}$

}
\adfLgap \noindent by the mapping:
$x \mapsto x +  j \adfmod{216}$ for $x < 216$,
$x \mapsto (x +  j \adfmod{24}) + 216$ for $216 \le x < 240$,
$x \mapsto (x +  j \adfmod{12}) + 240$ for $240 \le x < 252$,
$x \mapsto (x - 252 +  j \adfmod{8}) + 252$ for $x \ge 252$,
$0 \le j < 216$.
\ADFvfyParStart{(260, ((14, 216, ((216, 1), (24, 1), (12, 1), (8, 1)))), ((24, 9), (44, 1)))} 

\adfDgap
\noindent{\boldmath $ 24^{10} 4^{1} $}~
With the point set $\{0, 1, \dots, 243\}$ partitioned into
 residue classes modulo $10$ for $\{0, 1, \dots, 239\}$, and
 $\{240, 241, 242, 243\}$,
 the design is generated from

\adfLgap {\adfBfont
$\{147, 152, 199, 120, 58\}$,
$\{162, 215, 231, 143, 158\}$,
$\{94, 223, 230, 7, 81\}$,\adfsplit
$\{70, 64, 179, 241, 61\}$,
$\{0, 1, 42, 174, 185\}$,
$\{0, 14, 39, 117, 203\}$,\adfsplit
$\{0, 21, 59, 82, 127\}$,
$\{0, 8, 36, 54, 155\}$,
$\{0, 2, 31, 126, 207\}$,\adfsplit
$\{0, 26, 75, 138, 182\}$,
$\{0, 12, 34, 77, 169\}$,
$\{0, 48, 96, 144, 192\}$

}
\adfLgap \noindent by the mapping:
$x \mapsto x +  j \adfmod{240}$ for $x < 240$,
$x \mapsto (x +  j \adfmod{4}) + 240$ for $x \ge 240$,
$0 \le j < 240$
 for the first 11 blocks,
$0 \le j < 48$
 for the last block.
\ADFvfyParStart{(244, ((11, 240, ((240, 1), (4, 1))), (1, 48, ((240, 1), (4, 1)))), ((24, 10), (4, 1)))} 

\adfDgap
\noindent{\boldmath $ 24^{10} 12^{1} $}~
With the point set $\{0, 1, \dots, 251\}$ partitioned into
 residue classes modulo $10$ for $\{0, 1, \dots, 239\}$, and
 $\{240, 241, \dots, 251\}$,
 the design is generated from

\adfLgap {\adfBfont
$\{200, 73, 167, 56, 141\}$,
$\{182, 127, 64, 145, 229\}$,
$\{216, 227, 174, 175, 129\}$,\adfsplit
$\{126, 180, 246, 55, 209\}$,
$\{233, 208, 229, 201, 210\}$,
$\{0, 5, 44, 108, 121\}$,\adfsplit
$\{0, 12, 43, 79, 178\}$,
$\{0, 3, 61, 95, 109\}$,
$\{0, 22, 91, 184, 242\}$,\adfsplit
$\{0, 27, 65, 191, 248\}$,
$\{0, 8, 24, 97, 112\}$,
$\{0, 6, 57, 139, 174\}$

}
\adfLgap \noindent by the mapping:
$x \mapsto x +  j \adfmod{240}$ for $x < 240$,
$x \mapsto (x +  j \adfmod{12}) + 240$ for $x \ge 240$,
$0 \le j < 240$.
\ADFvfyParStart{(252, ((12, 240, ((240, 1), (12, 1)))), ((24, 10), (12, 1)))} 

\adfDgap
\noindent{\boldmath $ 24^{10} 32^{1} $}~
With the point set $\{0, 1, \dots, 271\}$ partitioned into
 residue classes modulo $10$ for $\{0, 1, \dots, 239\}$, and
 $\{240, 241, \dots, 271\}$,
 the design is generated from

\adfLgap {\adfBfont
$\{50, 26, 152, 17, 24\}$,
$\{249, 210, 183, 195, 161\}$,
$\{43, 252, 64, 82, 218\}$,\adfsplit
$\{252, 57, 100, 89, 223\}$,
$\{188, 106, 191, 90, 125\}$,
$\{26, 224, 77, 264, 238\}$,\adfsplit
$\{0, 6, 75, 194, 271\}$,
$\{0, 5, 59, 88, 146\}$,
$\{0, 38, 127, 183, 254\}$,\adfsplit
$\{0, 25, 62, 143, 252\}$,
$\{0, 8, 53, 84, 249\}$,
$\{0, 23, 64, 131, 167\}$,\adfsplit
$\{0, 1, 48, 92, 163\}$,
$\{0, 4, 17, 72, 133\}$

}
\adfLgap \noindent by the mapping:
$x \mapsto x +  j \adfmod{240}$ for $x < 240$,
$x \mapsto (x +  j \adfmod{24}) + 240$ for $240 \le x < 264$,
$x \mapsto (x +  j \adfmod{8}) + 264$ for $x \ge 264$,
$0 \le j < 240$.
\ADFvfyParStart{(272, ((14, 240, ((240, 1), (24, 1), (8, 1)))), ((24, 10), (32, 1)))} 

\adfDgap
\noindent{\boldmath $ 24^{11} 20^{1} $}~
With the point set $\{0, 1, \dots, 283\}$ partitioned into
 residue classes modulo $11$ for $\{0, 1, \dots, 263\}$, and
 $\{264, 265, \dots, 283\}$,
 the design is generated from

\adfLgap {\adfBfont
$\{268, 176, 101, 175, 85\}$,
$\{158, 9, 282, 250, 167\}$,
$\{188, 121, 95, 85, 229\}$,\adfsplit
$\{13, 1, 30, 213, 263\}$,
$\{23, 189, 140, 227, 224\}$,
$\{5, 78, 136, 83, 262\}$,\adfsplit
$\{0, 4, 19, 221, 282\}$,
$\{0, 20, 48, 109, 160\}$,
$\{0, 24, 56, 96, 218\}$,\adfsplit
$\{0, 8, 105, 135, 274\}$,
$\{0, 13, 114, 182, 270\}$,
$\{0, 6, 27, 45, 145\}$,\adfsplit
$\{0, 34, 71, 157, 222\}$,
$\{0, 25, 79, 136, 195\}$

}
\adfLgap \noindent by the mapping:
$x \mapsto x +  j \adfmod{264}$ for $x < 264$,
$x \mapsto (x +  j \adfmod{12}) + 264$ for $264 \le x < 276$,
$x \mapsto (x - 276 +  j \adfmod{8}) + 276$ for $x \ge 276$,
$0 \le j < 264$.
\ADFvfyParStart{(284, ((14, 264, ((264, 1), (12, 1), (8, 1)))), ((24, 11), (20, 1)))} 

\adfDgap
\noindent{\boldmath $ 25^{8} 5^{1} $}~
With the point set $\{0, 1, \dots, 204\}$ partitioned into
 residue classes modulo $8$ for $\{0, 1, \dots, 199\}$, and
 $\{200, 201, 202, 203, 204\}$,
 the design is generated from

\adfLgap {\adfBfont
$\{112, 76, 6, 185, 50\}$,
$\{4, 15, 42, 97, 46\}$,
$\{0, 1, 3, 194, 200\}$,\adfsplit
$\{0, 13, 43, 66, 103\}$,
$\{0, 14, 33, 68, 85\}$,
$\{0, 20, 61, 83, 142\}$,\adfsplit
$\{0, 10, 25, 111, 123\}$,
$\{0, 18, 46, 75, 151\}$,
$\{0, 5, 39, 84, 131\}$,\adfsplit
$\{0, 50, 100, 150, 204\}$

}
\adfLgap \noindent by the mapping:
$x \mapsto x +  j \adfmod{200}$ for $x < 200$,
$x \mapsto (x +  j \adfmod{4}) + 200$ for $200 \le x < 204$,
$204 \mapsto 204$,
$0 \le j < 200$
 for the first nine blocks,
$0 \le j < 50$
 for the last block.
\ADFvfyParStart{(205, ((9, 200, ((200, 1), (4, 1), (1, 1))), (1, 50, ((200, 1), (4, 1), (1, 1)))), ((25, 8), (5, 1)))} 

\adfDgap
\noindent{\boldmath $ 25^{8} 45^{1} $}~
With the point set $\{0, 1, \dots, 244\}$ partitioned into
 residue classes modulo $8$ for $\{0, 1, \dots, 199\}$, and
 $\{200, 201, \dots, 244\}$,
 the design is generated from

\adfLgap {\adfBfont
$\{30, 202, 57, 120, 188\}$,
$\{43, 78, 232, 37, 199\}$,
$\{189, 212, 28, 81, 174\}$,\adfsplit
$\{104, 79, 180, 161, 74\}$,
$\{21, 31, 207, 19, 86\}$,
$\{0, 9, 22, 83, 230\}$,\adfsplit
$\{0, 17, 51, 98, 240\}$,
$\{0, 4, 33, 142, 200\}$,
$\{0, 20, 97, 125, 237\}$,\adfsplit
$\{0, 1, 37, 60, 86\}$,
$\{0, 7, 18, 129, 222\}$,
$\{0, 3, 46, 130, 209\}$,\adfsplit
$\{0, 14, 148, 169, 219\}$,
$\{0, 50, 100, 150, 244\}$

}
\adfLgap \noindent by the mapping:
$x \mapsto x +  j \adfmod{200}$ for $x < 200$,
$x \mapsto (x +  j \adfmod{40}) + 200$ for $200 \le x < 240$,
$x \mapsto (x +  j \adfmod{4}) + 240$ for $240 \le x < 244$,
$244 \mapsto 244$,
$0 \le j < 200$
 for the first 13 blocks,
$0 \le j < 50$
 for the last block.
\ADFvfyParStart{(245, ((13, 200, ((200, 1), (40, 1), (4, 1), (1, 1))), (1, 50, ((200, 1), (40, 1), (4, 1), (1, 1)))), ((25, 8), (45, 1)))} 

\adfDgap
\noindent{\boldmath $ 28^{6} 20^{1} $}~
With the point set $\{0, 1, \dots, 187\}$ partitioned into
 residue classes modulo $6$ for $\{0, 1, \dots, 167\}$, and
 $\{168, 169, \dots, 187\}$,
 the design is generated from

\adfLgap {\adfBfont
$\{64, 77, 36, 146, 67\}$,
$\{165, 100, 56, 17, 178\}$,
$\{0, 1, 33, 35, 52\}$,\adfsplit
$\{0, 9, 23, 49, 130\}$,
$\{0, 8, 37, 53, 64\}$,
$\{0, 5, 62, 105, 168\}$,\adfsplit
$\{0, 7, 80, 101, 172\}$,
$\{0, 15, 70, 92, 180\}$,
$\{0, 4, 97, 122, 187\}$

}
\adfLgap \noindent by the mapping:
$x \mapsto x +  j \adfmod{168}$ for $x < 168$,
$x \mapsto (x +  j \adfmod{12}) + 168$ for $168 \le x < 180$,
$x \mapsto (x - 180 +  j \adfmod{8}) + 180$ for $x \ge 180$,
$0 \le j < 168$.
\ADFvfyParStart{(188, ((9, 168, ((168, 1), (12, 1), (8, 1)))), ((28, 6), (20, 1)))} 

\adfDgap
\noindent{\boldmath $ 28^{6} 40^{1} $}~
With the point set $\{0, 1, \dots, 207\}$ partitioned into
 residue classes modulo $6$ for $\{0, 1, \dots, 167\}$, and
 $\{168, 169, \dots, 207\}$,
 the design is generated from

\adfLgap {\adfBfont
$\{207, 30, 39, 17, 1\}$,
$\{24, 91, 63, 110, 58\}$,
$\{193, 164, 22, 57, 133\}$,\adfsplit
$\{132, 58, 61, 188, 81\}$,
$\{0, 7, 75, 131, 198\}$,
$\{0, 1, 64, 81, 197\}$,\adfsplit
$\{0, 2, 10, 55, 175\}$,
$\{0, 14, 79, 141, 180\}$,
$\{0, 15, 40, 125, 174\}$,\adfsplit
$\{0, 4, 50, 99, 182\}$,
$\{0, 11, 32, 109, 181\}$

}
\adfLgap \noindent by the mapping:
$x \mapsto x +  j \adfmod{168}$ for $x < 168$,
$x \mapsto (x +  j \adfmod{28}) + 168$ for $168 \le x < 196$,
$x \mapsto (x - 196 +  j \adfmod{12}) + 196$ for $x \ge 196$,
$0 \le j < 168$.
\ADFvfyParStart{(208, ((11, 168, ((168, 1), (28, 1), (12, 1)))), ((28, 6), (40, 1)))} 

\adfDgap
\noindent{\boldmath $ 28^{7} 16^{1} $}~
With the point set $\{0, 1, \dots, 211\}$ partitioned into
 residue classes modulo $7$ for $\{0, 1, \dots, 195\}$, and
 $\{196, 197, \dots, 211\}$,
 the design is generated from

\adfLgap {\adfBfont
$\{206, 16, 82, 189, 183\}$,
$\{201, 68, 130, 27, 177\}$,
$\{204, 27, 133, 12, 94\}$,\adfsplit
$\{0, 1, 3, 186, 199\}$,
$\{0, 5, 27, 79, 165\}$,
$\{0, 4, 37, 54, 162\}$,\adfsplit
$\{0, 8, 20, 65, 143\}$,
$\{0, 9, 68, 92, 111\}$,
$\{0, 18, 44, 69, 99\}$,\adfsplit
$\{0, 16, 64, 96, 136\}$

}
\adfLgap \noindent by the mapping:
$x \mapsto x +  j \adfmod{196}$ for $x < 196$,
$x \mapsto (x - 196 + 4 j \adfmod{16}) + 196$ for $x \ge 196$,
$0 \le j < 196$.
\ADFvfyParStart{(212, ((10, 196, ((196, 1), (16, 4)))), ((28, 7), (16, 1)))} 

\adfDgap
\noindent{\boldmath $ 28^{7} 36^{1} $}~
With the point set $\{0, 1, \dots, 231\}$ partitioned into
 residue classes modulo $7$ for $\{0, 1, \dots, 195\}$, and
 $\{196, 197, \dots, 231\}$,
 the design is generated from

\adfLgap {\adfBfont
$\{216, 39, 5, 72, 108\}$,
$\{96, 16, 160, 17, 116\}$,
$\{113, 167, 63, 94, 197\}$,\adfsplit
$\{215, 176, 42, 27, 164\}$,
$\{0, 2, 5, 27, 224\}$,
$\{0, 11, 29, 150, 225\}$,\adfsplit
$\{0, 16, 87, 110, 197\}$,
$\{0, 17, 43, 83, 131\}$,
$\{0, 4, 13, 89, 128\}$,\adfsplit
$\{0, 10, 55, 145, 213\}$,
$\{0, 6, 30, 38, 212\}$,
$\{0, 37, 95, 155, 204\}$

}
\adfLgap \noindent by the mapping:
$x \mapsto x +  j \adfmod{196}$ for $x < 196$,
$x \mapsto (x +  j \adfmod{28}) + 196$ for $196 \le x < 224$,
$x \mapsto (x + 2 j \adfmod{8}) + 224$ for $x \ge 224$,
$0 \le j < 196$.
\ADFvfyParStart{(232, ((12, 196, ((196, 1), (28, 1), (8, 2)))), ((28, 7), (36, 1)))} 

\adfDgap
\noindent{\boldmath $ 28^{8} 12^{1} $}~
With the point set $\{0, 1, \dots, 235\}$ partitioned into
 residue classes modulo $8$ for $\{0, 1, \dots, 223\}$, and
 $\{224, 225, \dots, 235\}$,
 the design is generated from

\adfLgap {\adfBfont
$\{174, 16, 169, 138, 181\}$,
$\{10, 129, 189, 231, 99\}$,
$\{91, 92, 6, 113, 45\}$,\adfsplit
$\{100, 193, 152, 114, 142\}$,
$\{0, 2, 27, 189, 232\}$,
$\{0, 6, 19, 116, 149\}$,\adfsplit
$\{0, 15, 76, 99, 133\}$,
$\{0, 3, 73, 77, 126\}$,
$\{0, 17, 109, 159, 224\}$,\adfsplit
$\{0, 9, 63, 146, 166\}$,
$\{0, 11, 29, 55, 124\}$

}
\adfLgap \noindent by the mapping:
$x \mapsto x +  j \adfmod{224}$ for $x < 224$,
$x \mapsto (x +  j \adfmod{8}) + 224$ for $224 \le x < 232$,
$x \mapsto (x +  j \adfmod{4}) + 232$ for $x \ge 232$,
$0 \le j < 224$.
\ADFvfyParStart{(236, ((11, 224, ((224, 1), (8, 1), (4, 1)))), ((28, 8), (12, 1)))} 

\adfDgap
\noindent{\boldmath $ 28^{8} 32^{1} $}~
With the point set $\{0, 1, \dots, 255\}$ partitioned into
 residue classes modulo $8$ for $\{0, 1, \dots, 223\}$, and
 $\{224, 225, \dots, 255\}$,
 the design is generated from

\adfLgap {\adfBfont
$\{254, 6, 196, 24, 199\}$,
$\{230, 180, 197, 30, 58\}$,
$\{222, 209, 168, 167, 130\}$,\adfsplit
$\{156, 237, 206, 176, 75\}$,
$\{104, 214, 15, 203, 228\}$,
$\{0, 27, 60, 130, 165\}$,\adfsplit
$\{0, 5, 12, 83, 129\}$,
$\{0, 10, 161, 205, 244\}$,
$\{0, 21, 82, 108, 147\}$,\adfsplit
$\{0, 14, 67, 90, 254\}$,
$\{0, 6, 15, 133, 155\}$,
$\{0, 2, 45, 158, 237\}$,\adfsplit
$\{0, 4, 51, 166, 233\}$

}
\adfLgap \noindent by the mapping:
$x \mapsto x +  j \adfmod{224}$ for $x < 224$,
$x \mapsto (x +  j \adfmod{32}) + 224$ for $x \ge 224$,
$0 \le j < 224$.
\ADFvfyParStart{(256, ((13, 224, ((224, 1), (32, 1)))), ((28, 8), (32, 1)))} 

\adfDgap
\noindent{\boldmath $ 28^{8} 52^{1} $}~
With the point set $\{0, 1, \dots, 275\}$ partitioned into
 residue classes modulo $8$ for $\{0, 1, \dots, 223\}$, and
 $\{224, 225, \dots, 275\}$,
 the design is generated from

\adfLgap {\adfBfont
$\{238, 47, 96, 11, 9\}$,
$\{152, 99, 7, 89, 230\}$,
$\{35, 213, 44, 267, 49\}$,\adfsplit
$\{153, 227, 131, 48, 77\}$,
$\{111, 44, 115, 41, 141\}$,
$\{11, 22, 133, 243, 145\}$,\adfsplit
$\{40, 228, 131, 165, 166\}$,
$\{0, 7, 27, 68, 240\}$,
$\{0, 31, 108, 177, 248\}$,\adfsplit
$\{0, 17, 42, 131, 271\}$,
$\{0, 13, 86, 143, 272\}$,
$\{0, 15, 43, 66, 118\}$,\adfsplit
$\{0, 18, 39, 205, 258\}$,
$\{0, 6, 50, 165, 265\}$,
$\{0, 33, 95, 140, 253\}$

}
\adfLgap \noindent by the mapping:
$x \mapsto x +  j \adfmod{224}$ for $x < 224$,
$x \mapsto (x +  j \adfmod{28}) + 224$ for $224 \le x < 252$,
$x \mapsto (x - 252 +  j \adfmod{16}) + 252$ for $252 \le x < 268$,
$x \mapsto (x - 268 +  j \adfmod{8}) + 268$ for $x \ge 268$,
$0 \le j < 224$.
\ADFvfyParStart{(276, ((15, 224, ((224, 1), (28, 1), (16, 1), (8, 1)))), ((28, 8), (52, 1)))} 

\adfDgap
\noindent{\boldmath $ 28^{9} 8^{1} $}~
With the point set $\{0, 1, \dots, 259\}$ partitioned into
 residue classes modulo $9$ for $\{0, 1, \dots, 251\}$, and
 $\{252, 253, \dots, 259\}$,
 the design is generated from

\adfLgap {\adfBfont
$\{253, 165, 162, 131, 24\}$,
$\{258, 190, 153, 112, 155\}$,
$\{147, 29, 118, 103, 53\}$,\adfsplit
$\{165, 83, 213, 226, 97\}$,
$\{139, 244, 188, 39, 51\}$,
$\{0, 4, 25, 30, 83\}$,\adfsplit
$\{0, 1, 20, 132, 230\}$,
$\{0, 6, 46, 110, 201\}$,
$\{0, 8, 60, 92, 185\}$,\adfsplit
$\{0, 16, 33, 71, 166\}$,
$\{0, 11, 39, 124, 190\}$,
$\{0, 7, 76, 172, 182\}$

}
\adfLgap \noindent by the mapping:
$x \mapsto x +  j \adfmod{252}$ for $x < 252$,
$x \mapsto (x - 252 + 2 j \adfmod{8}) + 252$ for $x \ge 252$,
$0 \le j < 252$.
\ADFvfyParStart{(260, ((12, 252, ((252, 1), (8, 2)))), ((28, 9), (8, 1)))} 

\adfDgap
\noindent{\boldmath $ 28^{9} 48^{1} $}~
With the point set $\{0, 1, \dots, 299\}$ partitioned into
 residue classes modulo $9$ for $\{0, 1, \dots, 251\}$, and
 $\{252, 253, \dots, 299\}$,
 the design is generated from

\adfLgap {\adfBfont
$\{295, 196, 78, 81, 103\}$,
$\{290, 45, 84, 22, 151\}$,
$\{289, 5, 76, 110, 187\}$,\adfsplit
$\{225, 46, 76, 128, 177\}$,
$\{208, 272, 70, 15, 134\}$,
$\{11, 282, 165, 243, 226\}$,\adfsplit
$\{17, 10, 85, 2, 258\}$,
$\{54, 82, 240, 38, 140\}$,
$\{0, 10, 43, 130, 286\}$,\adfsplit
$\{0, 2, 214, 228, 285\}$,
$\{0, 13, 60, 89, 173\}$,
$\{0, 12, 31, 140, 265\}$,\adfsplit
$\{0, 1, 5, 211, 263\}$,
$\{0, 6, 116, 201, 260\}$,
$\{0, 11, 91, 156, 278\}$,\adfsplit
$\{0, 21, 53, 148, 217\}$

}
\adfLgap \noindent by the mapping:
$x \mapsto x +  j \adfmod{252}$ for $x < 252$,
$x \mapsto (x +  j \adfmod{36}) + 252$ for $252 \le x < 288$,
$x \mapsto (x +  j \adfmod{12}) + 288$ for $x \ge 288$,
$0 \le j < 252$.
\ADFvfyParStart{(300, ((16, 252, ((252, 1), (36, 1), (12, 1)))), ((28, 9), (48, 1)))} 

\adfDgap
\noindent{\boldmath $ 28^{10} 4^{1} $}~
With the point set $\{0, 1, \dots, 283\}$ partitioned into
 residue classes modulo $10$ for $\{0, 1, \dots, 279\}$, and
 $\{280, 281, 282, 283\}$,
 the design is generated from

\adfLgap {\adfBfont
$\{216, 62, 138, 251, 5\}$,
$\{30, 107, 144, 168, 23\}$,
$\{6, 181, 152, 269, 184\}$,\adfsplit
$\{158, 163, 31, 47, 202\}$,
$\{87, 75, 22, 266, 194\}$,
$\{0, 43, 106, 165, 280\}$,\adfsplit
$\{0, 11, 26, 188, 239\}$,
$\{0, 8, 74, 112, 205\}$,
$\{0, 1, 22, 28, 47\}$,\adfsplit
$\{0, 9, 23, 71, 207\}$,
$\{0, 13, 55, 194, 212\}$,
$\{0, 2, 33, 97, 226\}$,\adfsplit
$\{0, 4, 49, 128, 186\}$

}
\adfLgap \noindent by the mapping:
$x \mapsto x +  j \adfmod{280}$ for $x < 280$,
$x \mapsto (x +  j \adfmod{4}) + 280$ for $x \ge 280$,
$0 \le j < 280$.
\ADFvfyParStart{(284, ((13, 280, ((280, 1), (4, 1)))), ((28, 10), (4, 1)))} 

\adfDgap
\noindent{\boldmath $ 28^{10} 24^{1} $}~
With the point set $\{0, 1, \dots, 303\}$ partitioned into
 residue classes modulo $10$ for $\{0, 1, \dots, 279\}$, and
 $\{280, 281, \dots, 303\}$,
 the design is generated from

\adfLgap {\adfBfont
$\{97, 188, 273, 146, 145\}$,
$\{164, 160, 78, 66, 125\}$,
$\{203, 240, 18, 95, 34\}$,\adfsplit
$\{174, 243, 293, 81, 95\}$,
$\{184, 241, 246, 157, 208\}$,
$\{205, 26, 91, 183, 250\}$,\adfsplit
$\{222, 276, 298, 253, 170\}$,
$\{0, 15, 117, 158, 300\}$,
$\{0, 18, 64, 96, 183\}$,\adfsplit
$\{0, 6, 17, 72, 181\}$,
$\{0, 9, 28, 53, 177\}$,
$\{0, 2, 73, 246, 293\}$,\adfsplit
$\{0, 7, 75, 88, 151\}$,
$\{0, 3, 29, 142, 286\}$,
$\{0, 8, 134, 155, 289\}$

}
\adfLgap \noindent by the mapping:
$x \mapsto x +  j \adfmod{280}$ for $x < 280$,
$x \mapsto (x +  j \adfmod{20}) + 280$ for $280 \le x < 300$,
$x \mapsto (x +  j \adfmod{4}) + 300$ for $x \ge 300$,
$0 \le j < 280$.
\ADFvfyParStart{(304, ((15, 280, ((280, 1), (20, 1), (4, 1)))), ((28, 10), (24, 1)))} 

\adfDgap
\noindent{\boldmath $ 29^{8} 1^{1} $}~
With the point set $\{0, 1, \dots, 232\}$ partitioned into
 residue classes modulo $8$ for $\{0, 1, \dots, 231\}$, and
 $\{232\}$,
 the design is generated from

\adfLgap {\adfBfont
$\{8, 230, 117, 135, 146\}$,
$\{43, 17, 13, 168, 174\}$,
$\{207, 108, 169, 2, 0\}$,\adfsplit
$\{0, 3, 22, 39, 76\}$,
$\{0, 9, 21, 52, 66\}$,
$\{0, 7, 53, 122, 150\}$,\adfsplit
$\{0, 1, 42, 93, 183\}$,
$\{0, 5, 20, 79, 134\}$,
$\{0, 33, 67, 111, 197\}$,\adfsplit
$\{0, 13, 60, 83, 145\}$,
$\{0, 58, 116, 174, 232\}$

}
\adfLgap \noindent by the mapping:
$x \mapsto x +  j \adfmod{232}$ for $x < 232$,
$232 \mapsto 232$,
$0 \le j < 232$
 for the first ten blocks,
$0 \le j < 58$
 for the last block.
\ADFvfyParStart{(233, ((10, 232, ((232, 1), (1, 1))), (1, 58, ((232, 1), (1, 1)))), ((29, 8), (1, 1)))} 

\adfDgap
\noindent{\boldmath $ 29^{8} 21^{1} $}~
With the point set $\{0, 1, \dots, 252\}$ partitioned into
 residue classes modulo $8$ for $\{0, 1, \dots, 231\}$, and
 $\{232, 233, \dots, 252\}$,
 the design is generated from

\adfLgap {\adfBfont
$\{231, 3, 233, 82, 220\}$,
$\{238, 46, 74, 188, 123\}$,
$\{20, 14, 89, 66, 71\}$,\adfsplit
$\{212, 138, 141, 121, 182\}$,
$\{0, 1, 36, 43, 239\}$,
$\{0, 2, 31, 137, 248\}$,\adfsplit
$\{0, 9, 63, 133, 238\}$,
$\{0, 13, 34, 145, 164\}$,
$\{0, 12, 26, 85, 110\}$,\adfsplit
$\{0, 22, 60, 115, 165\}$,
$\{0, 33, 78, 140, 179\}$,
$\{0, 10, 37, 119, 166\}$,\adfsplit
$\{0, 58, 116, 174, 252\}$

}
\adfLgap \noindent by the mapping:
$x \mapsto x +  j \adfmod{232}$ for $x < 232$,
$x \mapsto (x - 232 + 2 j \adfmod{16}) + 232$ for $232 \le x < 248$,
$x \mapsto (x +  j \adfmod{4}) + 248$ for $248 \le x < 252$,
$252 \mapsto 252$,
$0 \le j < 232$
 for the first 12 blocks,
$0 \le j < 58$
 for the last block.
\ADFvfyParStart{(253, ((12, 232, ((232, 1), (16, 2), (4, 1), (1, 1))), (1, 58, ((232, 1), (16, 2), (4, 1), (1, 1)))), ((29, 8), (21, 1)))} 

\adfDgap
\noindent{\boldmath $ 32^{6} 20^{1} $}~
With the point set $\{0, 1, \dots, 211\}$ partitioned into
 residue classes modulo $6$ for $\{0, 1, \dots, 191\}$, and
 $\{192, 193, \dots, 211\}$,
 the design is generated from

\adfLgap {\adfBfont
$\{47, 32, 21, 127, 156\}$,
$\{128, 72, 187, 160, 3\}$,
$\{40, 156, 119, 201, 85\}$,\adfsplit
$\{0, 1, 14, 119, 208\}$,
$\{0, 17, 50, 81, 139\}$,
$\{0, 7, 16, 62, 101\}$,\adfsplit
$\{0, 10, 92, 141, 195\}$,
$\{0, 28, 75, 127, 202\}$,
$\{0, 2, 5, 43, 172\}$,\adfsplit
$\{0, 4, 23, 44, 204\}$

}
\adfLgap \noindent by the mapping:
$x \mapsto x +  j \adfmod{192}$ for $x < 192$,
$x \mapsto (x +  j \adfmod{16}) + 192$ for $192 \le x < 208$,
$x \mapsto (x +  j \adfmod{4}) + 208$ for $x \ge 208$,
$0 \le j < 192$.
\ADFvfyParStart{(212, ((10, 192, ((192, 1), (16, 1), (4, 1)))), ((32, 6), (20, 1)))} 

\adfDgap
\noindent{\boldmath $ 32^{6} 40^{1} $}~
With the point set $\{0, 1, \dots, 231\}$ partitioned into
 residue classes modulo $6$ for $\{0, 1, \dots, 191\}$, and
 $\{192, 193, \dots, 231\}$,
 the design is generated from

\adfLgap {\adfBfont
$\{158, 106, 42, 121, 15\}$,
$\{115, 222, 106, 107, 141\}$,
$\{226, 137, 116, 43, 63\}$,\adfsplit
$\{133, 64, 229, 129, 179\}$,
$\{150, 164, 198, 1, 148\}$,
$\{0, 19, 75, 107, 130\}$,\adfsplit
$\{0, 5, 87, 134, 219\}$,
$\{0, 13, 38, 135, 206\}$,
$\{0, 11, 33, 100, 202\}$,\adfsplit
$\{0, 7, 17, 68, 222\}$,
$\{0, 3, 44, 83, 207\}$,
$\{0, 28, 121, 161, 217\}$

}
\adfLgap \noindent by the mapping:
$x \mapsto x +  j \adfmod{192}$ for $x < 192$,
$x \mapsto (x +  j \adfmod{32}) + 192$ for $192 \le x < 224$,
$x \mapsto (x +  j \adfmod{8}) + 224$ for $x \ge 224$,
$0 \le j < 192$.
\ADFvfyParStart{(232, ((12, 192, ((192, 1), (32, 1), (8, 1)))), ((32, 6), (40, 1)))} 

\adfDgap
\noindent{\boldmath $ 32^{7} 4^{1} $}~
With the point set $\{0, 1, \dots, 227\}$ partitioned into
 residue classes modulo $7$ for $\{0, 1, \dots, 223\}$, and
 $\{224, 225, 226, 227\}$,
 the design is generated from

\adfLgap {\adfBfont
$\{151, 127, 189, 5, 142\}$,
$\{198, 40, 209, 161, 1\}$,
$\{6, 101, 142, 113, 19\}$,\adfsplit
$\{0, 1, 6, 199, 224\}$,
$\{0, 22, 45, 89, 188\}$,
$\{0, 4, 34, 120, 152\}$,\adfsplit
$\{0, 2, 59, 92, 111\}$,
$\{0, 8, 18, 83, 163\}$,
$\{0, 3, 54, 74, 127\}$,\adfsplit
$\{0, 17, 60, 110, 156\}$

}
\adfLgap \noindent by the mapping:
$x \mapsto x +  j \adfmod{224}$ for $x < 224$,
$x \mapsto (x +  j \adfmod{4}) + 224$ for $x \ge 224$,
$0 \le j < 224$.
\ADFvfyParStart{(228, ((10, 224, ((224, 1), (4, 1)))), ((32, 7), (4, 1)))} 

\adfDgap
\noindent{\boldmath $ 32^{7} 24^{1} $}~
With the point set $\{0, 1, \dots, 247\}$ partitioned into
 residue classes modulo $7$ for $\{0, 1, \dots, 223\}$, and
 $\{224, 225, \dots, 247\}$,
 the design is generated from

\adfLgap {\adfBfont
$\{39, 203, 6, 242, 5\}$,
$\{225, 166, 223, 116, 184\}$,
$\{246, 216, 173, 108, 78\}$,\adfsplit
$\{42, 219, 157, 67, 134\}$,
$\{57, 102, 44, 38, 54\}$,
$\{0, 9, 20, 103, 144\}$,\adfsplit
$\{0, 8, 40, 171, 186\}$,
$\{0, 2, 54, 71, 150\}$,
$\{0, 4, 101, 106, 234\}$,\adfsplit
$\{0, 22, 73, 104, 231\}$,
$\{0, 29, 66, 110, 165\}$,
$\{0, 12, 36, 149, 232\}$

}
\adfLgap \noindent by the mapping:
$x \mapsto x +  j \adfmod{224}$ for $x < 224$,
$x \mapsto (x +  j \adfmod{16}) + 224$ for $224 \le x < 240$,
$x \mapsto (x +  j \adfmod{8}) + 240$ for $x \ge 240$,
$0 \le j < 224$.
\ADFvfyParStart{(248, ((12, 224, ((224, 1), (16, 1), (8, 1)))), ((32, 7), (24, 1)))} 

\adfDgap
\noindent{\boldmath $ 32^{7} 44^{1} $}~
With the point set $\{0, 1, \dots, 267\}$ partitioned into
 residue classes modulo $7$ for $\{0, 1, \dots, 223\}$, and
 $\{224, 225, \dots, 267\}$,
 the design is generated from

\adfLgap {\adfBfont
$\{140, 71, 230, 121, 136\}$,
$\{147, 117, 94, 179, 155\}$,
$\{187, 196, 262, 104, 198\}$,\adfsplit
$\{220, 98, 266, 97, 123\}$,
$\{66, 14, 259, 132, 183\}$,
$\{41, 145, 186, 36, 233\}$,\adfsplit
$\{0, 3, 13, 99, 180\}$,
$\{0, 17, 46, 110, 170\}$,
$\{0, 12, 87, 146, 244\}$,\adfsplit
$\{0, 16, 88, 204, 243\}$,
$\{0, 34, 73, 176, 249\}$,
$\{0, 6, 37, 95, 241\}$,\adfsplit
$\{0, 33, 76, 144, 246\}$,
$\{0, 18, 45, 202, 238\}$

}
\adfLgap \noindent by the mapping:
$x \mapsto x +  j \adfmod{224}$ for $x < 224$,
$x \mapsto (x +  j \adfmod{32}) + 224$ for $224 \le x < 256$,
$x \mapsto (x +  j \adfmod{8}) + 256$ for $256 \le x < 264$,
$x \mapsto (x +  j \adfmod{4}) + 264$ for $x \ge 264$,
$0 \le j < 224$.
\ADFvfyParStart{(268, ((14, 224, ((224, 1), (32, 1), (8, 1), (4, 1)))), ((32, 7), (44, 1)))} 

\adfDgap
\noindent{\boldmath $ 32^{8} 8^{1} $}~
With the point set $\{0, 1, \dots, 263\}$ partitioned into
 residue classes modulo $8$ for $\{0, 1, \dots, 255\}$, and
 $\{256, 257, \dots, 263\}$,
 the design is generated from

\adfLgap {\adfBfont
$\{16, 150, 20, 119, 97\}$,
$\{250, 41, 188, 5, 231\}$,
$\{28, 80, 86, 77, 177\}$,\adfsplit
$\{93, 55, 161, 122, 256\}$,
$\{154, 262, 140, 161, 166\}$,
$\{0, 1, 35, 124, 178\}$,\adfsplit
$\{0, 17, 63, 148, 181\}$,
$\{0, 23, 94, 121, 163\}$,
$\{0, 2, 20, 57, 174\}$,\adfsplit
$\{0, 15, 60, 110, 197\}$,
$\{0, 10, 61, 86, 151\}$,
$\{0, 13, 41, 142, 186\}$

}
\adfLgap \noindent by the mapping:
$x \mapsto x +  j \adfmod{256}$ for $x < 256$,
$x \mapsto (x +  j \adfmod{8}) + 256$ for $x \ge 256$,
$0 \le j < 256$.
\ADFvfyParStart{(264, ((12, 256, ((256, 1), (8, 1)))), ((32, 8), (8, 1)))} 

\adfDgap
\noindent{\boldmath $ 32^{8} 28^{1} $}~
With the point set $\{0, 1, \dots, 283\}$ partitioned into
 residue classes modulo $8$ for $\{0, 1, \dots, 255\}$, and
 $\{256, 257, \dots, 283\}$,
 the design is generated from

\adfLgap {\adfBfont
$\{218, 6, 145, 120, 244\}$,
$\{162, 275, 49, 119, 99\}$,
$\{23, 100, 141, 283, 222\}$,\adfsplit
$\{223, 186, 216, 140, 137\}$,
$\{256, 99, 120, 154, 255\}$,
$\{68, 80, 221, 134, 49\}$,\adfsplit
$\{0, 1, 90, 92, 279\}$,
$\{0, 6, 28, 39, 187\}$,
$\{0, 5, 131, 145, 259\}$,\adfsplit
$\{0, 15, 62, 211, 266\}$,
$\{0, 36, 74, 197, 265\}$,
$\{0, 23, 52, 105, 214\}$,\adfsplit
$\{0, 4, 13, 71, 106\}$,
$\{0, 10, 27, 137, 205\}$

}
\adfLgap \noindent by the mapping:
$x \mapsto x +  j \adfmod{256}$ for $x < 256$,
$x \mapsto (x +  j \adfmod{16}) + 256$ for $256 \le x < 272$,
$x \mapsto (x +  j \adfmod{8}) + 272$ for $272 \le x < 280$,
$x \mapsto (x +  j \adfmod{4}) + 280$ for $x \ge 280$,
$0 \le j < 256$.
\ADFvfyParStart{(284, ((14, 256, ((256, 1), (16, 1), (8, 1), (4, 1)))), ((32, 8), (28, 1)))} 

\adfDgap
\noindent{\boldmath $ 32^{9} 12^{1} $}~
With the point set $\{0, 1, \dots, 299\}$ partitioned into
 residue classes modulo $9$ for $\{0, 1, \dots, 287\}$, and
 $\{288, 289, \dots, 299\}$,
 the design is generated from

\adfLgap {\adfBfont
$\{62, 286, 91, 191, 297\}$,
$\{232, 46, 255, 254, 194\}$,
$\{149, 182, 138, 99, 123\}$,\adfsplit
$\{45, 10, 143, 40, 111\}$,
$\{3, 194, 141, 16, 13\}$,
$\{288, 273, 216, 160, 167\}$,\adfsplit
$\{0, 31, 82, 119, 292\}$,
$\{0, 6, 14, 91, 218\}$,
$\{0, 28, 74, 114, 179\}$,\adfsplit
$\{0, 19, 67, 92, 213\}$,
$\{0, 16, 58, 192, 226\}$,
$\{0, 4, 47, 115, 170\}$,\adfsplit
$\{0, 12, 116, 136, 157\}$,
$\{0, 2, 89, 141, 158\}$

}
\adfLgap \noindent by the mapping:
$x \mapsto x +  j \adfmod{288}$ for $x < 288$,
$x \mapsto (x +  j \adfmod{12}) + 288$ for $x \ge 288$,
$0 \le j < 288$.
\ADFvfyParStart{(300, ((14, 288, ((288, 1), (12, 1)))), ((32, 9), (12, 1)))} 

\adfDgap
\noindent{\boldmath $ 36^{6} 20^{1} $}~
With the point set $\{0, 1, \dots, 235\}$ partitioned into
 residue classes modulo $6$ for $\{0, 1, \dots, 215\}$, and
 $\{216, 217, \dots, 235\}$,
 the design is generated from

\adfLgap {\adfBfont
$\{153, 113, 187, 198, 14\}$,
$\{83, 195, 18, 223, 205\}$,
$\{62, 27, 229, 77, 138\}$,\adfsplit
$\{95, 150, 16, 175, 176\}$,
$\{0, 2, 69, 73, 230\}$,
$\{0, 3, 8, 49, 101\}$,\adfsplit
$\{0, 7, 23, 195, 218\}$,
$\{0, 31, 89, 152, 221\}$,
$\{0, 9, 22, 125, 163\}$,\adfsplit
$\{0, 17, 37, 146, 165\}$,
$\{0, 14, 47, 106, 133\}$

}
\adfLgap \noindent by the mapping:
$x \mapsto x +  j \adfmod{216}$ for $x < 216$,
$x \mapsto (x +  j \adfmod{12}) + 216$ for $216 \le x < 228$,
$x \mapsto (x - 228 +  j \adfmod{8}) + 228$ for $x \ge 228$,
$0 \le j < 216$.
\ADFvfyParStart{(236, ((11, 216, ((216, 1), (12, 1), (8, 1)))), ((36, 6), (20, 1)))} 

\adfDgap
\noindent{\boldmath $ 36^{6} 40^{1} $}~
With the point set $\{0, 1, \dots, 255\}$ partitioned into
 residue classes modulo $6$ for $\{0, 1, \dots, 215\}$, and
 $\{216, 217, \dots, 255\}$,
 the design is generated from

\adfLgap {\adfBfont
$\{141, 222, 44, 70, 42\}$,
$\{74, 70, 24, 231, 29\}$,
$\{80, 242, 165, 168, 133\}$,\adfsplit
$\{179, 195, 116, 221, 10\}$,
$\{92, 141, 130, 37, 210\}$,
$\{0, 1, 15, 22, 252\}$,\adfsplit
$\{0, 33, 94, 146, 232\}$,
$\{0, 23, 62, 130, 231\}$,
$\{0, 13, 40, 140, 165\}$,\adfsplit
$\{0, 8, 67, 141, 223\}$,
$\{0, 20, 135, 179, 239\}$,
$\{0, 9, 143, 160, 236\}$,\adfsplit
$\{0, 10, 29, 87, 121\}$

}
\adfLgap \noindent by the mapping:
$x \mapsto x +  j \adfmod{216}$ for $x < 216$,
$x \mapsto (x +  j \adfmod{36}) + 216$ for $216 \le x < 252$,
$x \mapsto (x +  j \adfmod{4}) + 252$ for $x \ge 252$,
$0 \le j < 216$.
\ADFvfyParStart{(256, ((13, 216, ((216, 1), (36, 1), (4, 1)))), ((36, 6), (40, 1)))} 

\adfDgap
\noindent{\boldmath $ 36^{7} 12^{1} $}~
With the point set $\{0, 1, \dots, 263\}$ partitioned into
 residue classes modulo $7$ for $\{0, 1, \dots, 251\}$, and
 $\{252, 253, \dots, 263\}$,
 the design is generated from

\adfLgap {\adfBfont
$\{261, 210, 176, 149, 243\}$,
$\{262, 161, 176, 242, 205\}$,
$\{263, 148, 97, 59, 180\}$,\adfsplit
$\{147, 30, 50, 125, 33\}$,
$\{12, 7, 219, 132, 85\}$,
$\{0, 2, 41, 144, 194\}$,\adfsplit
$\{0, 4, 68, 150, 156\}$,
$\{0, 9, 57, 122, 181\}$,
$\{0, 1, 24, 55, 200\}$,\adfsplit
$\{0, 10, 46, 72, 183\}$,
$\{0, 8, 19, 93, 109\}$,
$\{0, 12, 25, 148, 234\}$

}
\adfLgap \noindent by the mapping:
$x \mapsto x +  j \adfmod{252}$ for $x < 252$,
$x \mapsto (x +  j \adfmod{12}) + 252$ for $x \ge 252$,
$0 \le j < 252$.
\ADFvfyParStart{(264, ((12, 252, ((252, 1), (12, 1)))), ((36, 7), (12, 1)))} 

\adfDgap
\noindent{\boldmath $ 36^{7} 32^{1} $}~
With the point set $\{0, 1, \dots, 283\}$ partitioned into
 residue classes modulo $7$ for $\{0, 1, \dots, 251\}$, and
 $\{252, 253, \dots, 283\}$,
 the design is generated from

\adfLgap {\adfBfont
$\{281, 11, 248, 237, 134\}$,
$\{280, 226, 111, 176, 193\}$,
$\{267, 129, 36, 109, 28\}$,\adfsplit
$\{16, 92, 48, 150, 189\}$,
$\{174, 176, 2, 147, 258\}$,
$\{256, 123, 248, 154, 191\}$,\adfsplit
$\{0, 3, 22, 186, 268\}$,
$\{0, 34, 87, 162, 275\}$,
$\{0, 5, 45, 211, 271\}$,\adfsplit
$\{0, 6, 16, 152, 204\}$,
$\{0, 4, 59, 71, 89\}$,
$\{0, 1, 24, 96, 132\}$,\adfsplit
$\{0, 13, 38, 148, 191\}$,
$\{0, 9, 60, 122, 169\}$

}
\adfLgap \noindent by the mapping:
$x \mapsto x +  j \adfmod{252}$ for $x < 252$,
$x \mapsto (x - 252 + 2 j \adfmod{24}) + 252$ for $252 \le x < 276$,
$x \mapsto (x - 276 + 2 j \adfmod{8}) + 276$ for $x \ge 276$,
$0 \le j < 252$.
\ADFvfyParStart{(284, ((14, 252, ((252, 1), (24, 2), (8, 2)))), ((36, 7), (32, 1)))} 

\adfDgap
\noindent{\boldmath $ 36^{8} 24^{1} $}~
With the point set $\{0, 1, \dots, 311\}$ partitioned into
 residue classes modulo $8$ for $\{0, 1, \dots, 287\}$, and
 $\{288, 289, \dots, 311\}$,
 the design is generated from

\adfLgap {\adfBfont
$\{237, 80, 107, 274, 302\}$,
$\{180, 282, 214, 265, 149\}$,
$\{151, 221, 82, 225, 4\}$,\adfsplit
$\{168, 293, 126, 124, 42\}$,
$\{127, 254, 77, 240, 153\}$,
$\{14, 131, 106, 240, 25\}$,\adfsplit
$\{0, 1, 7, 10, 22\}$,
$\{0, 5, 35, 54, 183\}$,
$\{0, 38, 93, 146, 191\}$,\adfsplit
$\{0, 18, 188, 265, 304\}$,
$\{0, 57, 123, 181, 309\}$,
$\{0, 13, 59, 268, 301\}$,\adfsplit
$\{0, 29, 89, 132, 227\}$,
$\{0, 36, 173, 225, 307\}$,
$\{0, 28, 75, 114, 197\}$

}
\adfLgap \noindent by the mapping:
$x \mapsto x +  j \adfmod{288}$ for $x < 288$,
$x \mapsto (x +  j \adfmod{24}) + 288$ for $x \ge 288$,
$0 \le j < 288$.
\ADFvfyParStart{(312, ((15, 288, ((288, 1), (24, 1)))), ((36, 8), (24, 1)))} 

\adfDgap
\noindent{\boldmath $ 40^{6} 20^{1} $}~
With the point set $\{0, 1, \dots, 259\}$ partitioned into
 residue classes modulo $6$ for $\{0, 1, \dots, 239\}$, and
 $\{240, 241, \dots, 259\}$,
 the design is generated from

\adfLgap {\adfBfont
$\{93, 62, 47, 114, 245\}$,
$\{229, 143, 202, 152, 153\}$,
$\{233, 130, 74, 57, 254\}$,\adfsplit
$\{238, 96, 21, 246, 104\}$,
$\{0, 2, 109, 237, 256\}$,
$\{0, 7, 104, 178, 253\}$,\adfsplit
$\{0, 33, 68, 115, 155\}$,
$\{0, 11, 37, 116, 160\}$,
$\{0, 13, 38, 58, 101\}$,\adfsplit
$\{0, 19, 89, 111, 140\}$,
$\{0, 14, 53, 183, 199\}$,
$\{0, 4, 32, 145, 179\}$

}
\adfLgap \noindent by the mapping:
$x \mapsto x +  j \adfmod{240}$ for $x < 240$,
$x \mapsto (x +  j \adfmod{20}) + 240$ for $x \ge 240$,
$0 \le j < 240$.
\ADFvfyParStart{(260, ((12, 240, ((240, 1), (20, 1)))), ((40, 6), (20, 1)))} 

\adfDgap
\noindent{\boldmath $ 44^{6} 20^{1} $}~
With the point set $\{0, 1, \dots, 283\}$ partitioned into
 residue classes modulo $6$ for $\{0, 1, \dots, 263\}$, and
 $\{264, 265, \dots, 283\}$,
 the design is generated from

\adfLgap {\adfBfont
$\{278, 238, 169, 215, 26\}$,
$\{283, 105, 109, 176, 186\}$,
$\{127, 122, 23, 213, 96\}$,\adfsplit
$\{129, 118, 13, 271, 74\}$,
$\{197, 163, 214, 195, 270\}$,
$\{0, 3, 65, 79, 271\}$,\adfsplit
$\{0, 16, 45, 103, 227\}$,
$\{0, 15, 83, 130, 194\}$,
$\{0, 8, 41, 141, 166\}$,\adfsplit
$\{0, 20, 63, 113, 172\}$,
$\{0, 7, 28, 135, 170\}$,
$\{0, 1, 39, 119, 208\}$,\adfsplit
$\{0, 9, 22, 49, 176\}$

}
\adfLgap \noindent by the mapping:
$x \mapsto x +  j \adfmod{264}$ for $x < 264$,
$x \mapsto (x +  j \adfmod{12}) + 264$ for $264 \le x < 276$,
$x \mapsto (x - 276 +  j \adfmod{8}) + 276$ for $x \ge 276$,
$0 \le j < 264$.
\ADFvfyParStart{(284, ((13, 264, ((264, 1), (12, 1), (8, 1)))), ((44, 6), (20, 1)))} 

\adfDgap
\noindent{\boldmath $ 44^{6} 40^{1} $}~
With the point set $\{0, 1, \dots, 303\}$ partitioned into
 residue classes modulo $6$ for $\{0, 1, \dots, 263\}$, and
 $\{264, 265, \dots, 303\}$,
 the design is generated from

\adfLgap {\adfBfont
$\{270, 83, 205, 54, 249\}$,
$\{285, 170, 21, 179, 175\}$,
$\{266, 199, 261, 154, 30\}$,\adfsplit
$\{264, 177, 78, 131, 31\}$,
$\{6, 73, 89, 194, 166\}$,
$\{108, 110, 125, 145, 28\}$,\adfsplit
$\{0, 7, 75, 245, 279\}$,
$\{0, 10, 71, 163, 301\}$,
$\{0, 8, 31, 87, 232\}$,\adfsplit
$\{0, 1, 58, 128, 131\}$,
$\{0, 13, 34, 125, 213\}$,
$\{0, 39, 89, 148, 272\}$,\adfsplit
$\{0, 11, 25, 52, 302\}$,
$\{0, 22, 65, 103, 298\}$,
$\{0, 49, 123, 178, 288\}$

}
\adfLgap \noindent by the mapping:
$x \mapsto x +  j \adfmod{264}$ for $x < 264$,
$x \mapsto (x - 264 + 5 j \adfmod{40}) + 264$ for $x \ge 264$,
$0 \le j < 264$.
\ADFvfyParStart{(304, ((15, 264, ((264, 1), (40, 5)))), ((44, 6), (40, 1)))} 

\adfDgap
\noindent{\boldmath $ 44^{7} 8^{1} $}~
With the point set $\{0, 1, \dots, 315\}$ partitioned into
 residue classes modulo $7$ for $\{0, 1, \dots, 307\}$, and
 $\{308, 309, \dots, 315\}$,
 the design is generated from

\adfLgap {\adfBfont
$\{73, 46, 126, 6, 131\}$,
$\{310, 73, 86, 148, 103\}$,
$\{107, 55, 214, 220, 173\}$,\adfsplit
$\{223, 207, 121, 113, 131\}$,
$\{44, 291, 313, 21, 190\}$,
$\{185, 253, 289, 0, 257\}$,\adfsplit
$\{0, 11, 48, 135, 192\}$,
$\{0, 9, 97, 179, 208\}$,
$\{0, 12, 90, 163, 227\}$,\adfsplit
$\{0, 20, 74, 134, 229\}$,
$\{0, 1, 3, 25, 131\}$,
$\{0, 15, 46, 187, 258\}$,\adfsplit
$\{0, 34, 103, 142, 186\}$,
$\{0, 26, 59, 176, 219\}$

}
\adfLgap \noindent by the mapping:
$x \mapsto x +  j \adfmod{308}$ for $x < 308$,
$x \mapsto (x - 308 + 2 j \adfmod{8}) + 308$ for $x \ge 308$,
$0 \le j < 308$.
\ADFvfyParStart{(316, ((14, 308, ((308, 1), (8, 2)))), ((44, 7), (8, 1)))} 

\adfDgap
\noindent{\boldmath $ 48^{6} 20^{1} $}~
With the point set $\{0, 1, \dots, 307\}$ partitioned into
 residue classes modulo $6$ for $\{0, 1, \dots, 287\}$, and
 $\{288, 289, \dots, 307\}$,
 the design is generated from

\adfLgap {\adfBfont
$\{304, 76, 167, 165, 210\}$,
$\{140, 289, 245, 237, 226\}$,
$\{224, 84, 149, 172, 301\}$,\adfsplit
$\{272, 299, 125, 120, 33\}$,
$\{76, 269, 110, 123, 79\}$,
$\{221, 171, 301, 55, 54\}$,\adfsplit
$\{0, 7, 16, 219, 278\}$,
$\{0, 14, 77, 123, 178\}$,
$\{0, 27, 83, 157, 208\}$,\adfsplit
$\{0, 22, 93, 133, 206\}$,
$\{0, 29, 64, 182, 249\}$,
$\{0, 15, 115, 143, 176\}$,\adfsplit
$\{0, 4, 25, 57, 251\}$,
$\{0, 20, 58, 139, 209\}$

}
\adfLgap \noindent by the mapping:
$x \mapsto x +  j \adfmod{288}$ for $x < 288$,
$x \mapsto (x +  j \adfmod{16}) + 288$ for $288 \le x < 304$,
$x \mapsto (x +  j \adfmod{4}) + 304$ for $x \ge 304$,
$0 \le j < 288$.
\ADFvfyParStart{(308, ((14, 288, ((288, 1), (16, 1), (4, 1)))), ((48, 6), (20, 1)))} 

\adfDgap
\noindent{\boldmath $ 1^{16} 9^{5} $}~
With the point set $\{0, 1, \dots, 60\}$ partitioned into
 residue classes modulo $5$ for $\{0, 1, \dots, 44\}$, and
 residue classes modulo $16$ for $\{45, 46, \dots, 60\}$,
 the design is generated from

\adfLgap {\adfBfont
$\{0, 14, 16, 17, 58\}$,
$\{0, 1, 8, 46, 60\}$,
$\{0, 11, 32, 50, 52\}$,\adfsplit
$\{0, 29, 38, 45, 55\}$,
$\{1, 5, 17, 23, 57\}$,
$\{0, 13, 31, 44, 59\}$,\adfsplit
$\{0, 26, 34, 37, 43\}$,
$\{0, 2, 6, 24, 28\}$,
$\{0, 3, 36, 48, 54\}$,\adfsplit
$\{0, 7, 19, 49, 56\}$,
$\{1, 22, 48, 51, 52\}$

}
\adfLgap \noindent by the mapping:
$x \mapsto x + 3 j \adfmod{45}$ for $x < 45$,
$x \mapsto (x +  j \adfmod{15}) + 45$ for $45 \le x < 60$,
$60 \mapsto 60$,
$0 \le j < 15$.
\ADFvfyParStart{(61, ((11, 15, ((45, 3), (15, 1), (1, 1)))), ((9, 5), (1, 16)))} 

\adfDgap
\noindent{\boldmath $ 1^{46} 5^{3} $}~
With the point set $\{0, 1, \dots, 60\}$ partitioned into
 residue classes modulo $46$ for $\{0, 1, \dots, 45\}$, and
 residue classes modulo $3$ for $\{46, 47, \dots, 60\}$,
 the design is generated from

\adfLgap {\adfBfont
$\{16, 42, 32, 13, 18\}$,
$\{27, 9, 11, 24, 32\}$,
$\{0, 1, 33, 37, 46\}$,\adfsplit
$\{0, 9, 31, 51, 55\}$,
$\{0, 6, 44, 49, 59\}$,
$\{0, 17, 20, 58, 60\}$,\adfsplit
$\{0, 25, 26, 45, 56\}$,
$\{0, 10, 28, 34, 53\}$,
$\{0, 7, 11, 41, 54\}$,\adfsplit
$\{1, 13, 41, 47, 55\}$,
$\{1, 8, 14, 16, 26\}$,
$\{1, 23, 32, 53, 54\}$

}
\adfLgap \noindent by the mapping:
$x \mapsto x + 3 j \adfmod{45}$ for $x < 45$,
$45 \mapsto 45$,
$x \mapsto (x - 46 +  j \adfmod{15}) + 46$ for $x \ge 46$,
$0 \le j < 15$.
\ADFvfyParStart{(61, ((12, 15, ((45, 3), (1, 1), (15, 1)))), ((1, 46), (5, 3)))} 

\adfDgap
\noindent{\boldmath $ 4^{5} 8^{5} $}~
With the point set $\{0, 1, \dots, 59\}$ partitioned into
 residue classes modulo $5$ for $\{0, 1, \dots, 39\}$, and
 residue classes modulo $5$ for $\{40, 41, \dots, 59\}$,
 the design is generated from

\adfLgap {\adfBfont
$\{0, 2, 21, 33, 51\}$,
$\{0, 1, 7, 9, 54\}$,
$\{1, 5, 27, 45, 49\}$,\adfsplit
$\{0, 13, 37, 44, 46\}$,
$\{0, 29, 47, 53, 56\}$,
$\{0, 4, 27, 28, 57\}$,\adfsplit
$\{0, 3, 26, 32, 58\}$,
$\{0, 18, 41, 48, 49\}$

}
\adfLgap \noindent by the mapping:
$x \mapsto x + 2 j \adfmod{40}$ for $x < 40$,
$x \mapsto (x +  j \adfmod{20}) + 40$ for $x \ge 40$,
$0 \le j < 20$.
\ADFvfyParStart{(60, ((8, 20, ((40, 2), (20, 1)))), ((8, 5), (4, 5)))} 
%
%
\end{proof}


\section*{ORCID}

\noindent A. D. Forbes     \url{https://orcid.org/0000-0003-3805-7056}


\end{document}